\documentclass[a4paper]{amsart}

\usepackage[T1]{fontenc}
\usepackage[utf8]{inputenc}
\usepackage{amsmath, amsthm, amsfonts, amssymb}

\usepackage[colorlinks=true,hyperindex, linkcolor=magenta, pagebackref=false, citecolor=cyan,pdfpagelabels]{hyperref}
\usepackage{url}
\usepackage[all]{xy}
\usepackage{mathabx}
\usepackage{tensor}
\usepackage{placeins} % barriers for floats
\usepackage[usenames, dvipsnames]{xcolor} % colourfultext
\usepackage{euscript} % Lurie-type Euler script
\usepackage{bbm} % mathbb numbers
\usepackage{enumitem} % fancy enumerate
\usepackage[nameinlink,capitalise,noabbrev]{cleveref} % \Cref{prop:1.4} prints ``Proposition 1.4''
\usepackage{mathtools}
\usepackage{colonequals}

\definecolor{dark-red}{rgb}{0.5,0.15,0.15}
\definecolor{dark-blue}{rgb}{0.15,0.15,0.6}
\definecolor{dark-green}{rgb}{0.15,0.6,0.15}
\hypersetup{
    colorlinks, linkcolor=RoyalBlue,
    citecolor=dark-blue, urlcolor=dark-green
}

\usepackage{tikz-cd}
\usepackage{tikz}
\usetikzlibrary{arrows,backgrounds,
    decorations.pathreplacing,
    decorations.pathmorphing
}
\usetikzlibrary{matrix,arrows}

\newcommand{\euscr}[1]{\EuScript{#1}} % Euler script
\newcommand{\acat}{\euscr{A}} % category A in Euler script
\newcommand{\bcat}{\euscr{B}} % category A in Euler script
\newcommand{\ccat}{\euscr{C}} % category C in Euler script 
\newcommand{\dcat}{\euscr{D}} % category D in Euler script
\newcommand{\ecat}{\euscr{E}} % category E in Euler script 
\newcommand{\kcat}{\euscr{K}} % category E in Euler script 
\newcommand{\hcat}{\euscr{H}} % category H in Euler script
\newcommand{\pcat}{\euscr{P}} % category P in Euler script
\newcommand{\Fun}{\textnormal{Fun}} % functor category
\newcommand{\Hom}{\textnormal{Hom}} % homomorphims 
\newcommand{\Ext}{\textnormal{Ext}} % Ext 
\newcommand{\Ind}{\textnormal{Ind}} % Ext 
\newcommand{\map}{\textnormal{map}} % mapping space
\newcommand{\Map}{\textnormal{Map}} % mapping space with uppercase M 
\newcommand{\abeliangroups}{\euscr{A}b} % the category of abelian groups
\newcommand{\Ab}{\abeliangroups} % alias for the category of abelian groups
\newcommand{\Vect}{\euscr{V}ect} % the category of vector spaces
\newcommand{\catinfty}{\euscr{C}at_{\infty}} % the category of infty categories
\newcommand{\spectra}{\euscr{S}p} % the category of spectra
\newcommand{\Mod}{\euscr{M}od} % module infinity-category
\newcommand{\Comod}{\euscr{C}omod} % comodule infinity-category
\newcommand{\fieldp}{\mathbb{F}_{p}} % field with p elements
\newcommand{\monunit}{\mathbbm{1}} % the monoidal unit
\newcommand{\quiversteenrod}{\mathcal{Q}} % quiver Steenrod algebra
\newcommand{\bbZ}{\mathbb{Z}} % integers
\newcommand{\adapted}{\euscr{A}\mathrm{dp}} % the (infty,2)-category of adapted homology theories
\newcommand{\stableinftycats}{\euscr{S}\mathrm{t}} % the (infty,2)-category of adapted homology theories
\newcommand{\Hrm}{\mathrm{H}} 
\newcommand{\sheaves}{\mathrm{Sh}}

\DeclareMathOperator*{\colim}{\mathrm{colim}}

\def\Syn{\mathrm{Syn}}
\def\Z{\bbZ}
\def\Sp{\mathrm{Sp}}
\def\F{\mathbb{F}}

\def\DD{\mathcal{D}}

\def\E{\mathbb{E}}
\def\Def{\mathrm{Def}}

\def\Ss{\mathbb{S}}
\def\BP{\mathrm{BP}}
\def\op{\mathrm{op}}

\newcommand{\X}{\mathsf{X}}

\newcommand{\Y}{\mathsf{Y}} 
\newcommand{\A}{\mathsf{A}}

\newcommand{\C}{\mathsf{C}} 
\newcommand{\Hsf}{\mathsf{H}} 
\newcommand{\Psf}{\mathsf{P}} 
\newcommand{\Isf}{\mathsf{I}} 
\newcommand{\BPn}[1]{\BP \langle #1 \rangle}

\newcommand{\coker}{\mathrm{coker}}
\newcommand{\fib}{\mathrm{fib}}

\newcommand{\cof}{\mathrm{cof}}

\newcommand{\triplerightarrow}{%
\tikz[minimum height=0ex]
  \path[->]
   node (a)            {}
   node (b) at (1em,0) {}
  (a.north)  edge (b.north)
  (a.center) edge (b.center)
  (a.south)  edge (b.south);%
}

\newcommand{\pullback}{\arrow[dr, phantom, "\lrcorner", very near start]}

\theoremstyle{plain}

\newtheorem{theorem}{Theorem}[section]
\newtheorem{lemma}[theorem]{Lemma}
\newtheorem{proposition}[theorem]{Proposition}
\newtheorem{corollary}[theorem]{Corollary}

\newtheorem*{theorem*}{Theorem}

\theoremstyle{definition}
\newtheorem{example}[theorem]{Example}
\newtheorem{warning}[theorem]{Warning}

\newtheorem{definition}[theorem]{Definition}
\newtheorem{recollection}[theorem]{Recollection}
\newtheorem{remark}[theorem]{Remark}
\newtheorem{notation}[theorem]{Notation}
\newtheorem{variant}[theorem]{Variant}
\newtheorem{construction}[theorem]{Construction}

\newtheorem{convention}[theorem]{Convention}

\newtheorem*{remark*}{Remark}
\newtheorem*{interpretation*}{Interpretation}
\newtheorem*{definition*}{Definition}
\newtheorem*{conjecture*}{Conjecture}
\newtheorem*{notation*}{Notation}
\newtheorem*{convention*}{Convention}

\theoremstyle{remark}

\numberwithin{equation}{section}

\makeatletter
  \def\subsection{\@startsection{subsection}{1}%
  \z@{.7\linespacing\@plus\linespacing}{.5\linespacing}%
  {\normalfont\bfseries\centering}}% NEW
\makeatother

\let\oldtocsection=\tocsection
\let\oldtocsubsection=\tocsubsection
\let\oldtocsubsubsection=\tocsubsubsection
\renewcommand{\tocsection}[2]{\hspace{0em}\oldtocsection{#1}{#2}}
\renewcommand{\tocsubsection}[2]{\hspace{1em}\oldtocsubsection{#1}{#2}}
\renewcommand{\tocsubsubsection}[2]{\hspace{2em}\oldtocsubsubsection{#1}{#2}}

\calclayout

\DeclareMathOperator{\Rep}{Rep}
\DeclareMathOperator{\Tor}{Tor}
\DeclareMathOperator{\gr}{gr}

\mathchardef\mhyphen="2D

\newcommand{\KO}{\mathrm{KO}} 
\newcommand{\KU}{\mathrm{KU}} 
\newcommand{\KT}{\mathrm{KT}} 
 
\newcommand{\K}{\mathrm{K}} 
\newcommand{\CRT}{\mathrm{CRT}} 
\newcommand{\dashepi}{\mhyphen \mathrm{epi}}

\usepackage{subfiles} % Best loaded last in the preamble

\begin{document}
\title[Quivers and the Adams spectral sequence]{Quivers and the Adams spectral sequence}

\author{Robert Burklund}
\address{Department of Mathematical Sciences, University of Copenhagen, Denmark}
\email{rb@math.ku.dk}

\author{Piotr Pstr\k{a}gowski}
\address{Department of Mathematics, Harvard and Institute for Advanced Study, USA}
\email{pstragowski.piotr@gmail.com}

\begin{abstract}
In this paper, we describe a novel way of identifying Adams spectral sequence $E_{2}$-terms in terms of homological algebra of quiver representations. Our method applies much more broadly than the standard techniques based on descent-flatness, bearing on a varied array of ring spectra. In the particular case of $p$-local integral homology, we are able to give a decomposition of the $E_{2}$-term, describing it completely in terms of the classical Adams spectral sequence. In the appendix, which can be read independently from the main body of the text, we develop functoriality of deformations of $\infty$-categories of the second author and Patchkoria. 
\end{abstract}

\maketitle 

\tableofcontents

% the beginning sections do not go into the table of contents 
% \addtocontents{toc}{\protect\setcounter{tocdepth}{-1}}

\section{Introduction}
\label{sec:intro}

% In this paper, we describe a novel way of identifying Adams spectral sequence $E_{2}$-terms in terms of homological algebra of quiver representations. Our method applies much more broadly than the standard techniquess based on descent-flatness, applying to a varied array of ring spectra. In the particular case of integral homology, we are able to give a decomposition of the $E_{2}$-term, describing it completely in terms of the classical Adams spectral sequence. 

Many problems in algebraic geometry and stable homotopy theory can be attacked using the method of descent. In the case most important to homotopy theorists, one picks a ring spectrum $R$ and forms the Amitsur resolution of the sphere spectrum: 
\[ \begin{tikzcd}
    \Ss \ar[r] & R \arrow[r,yshift=-1ex] \arrow[r,yshift=1ex] & R \otimes R \ar[l, shorten=4pt] \ar[r,yshift=-2ex] \ar[r] \ar[r,yshift=2ex] & R \otimes R \otimes R \ar[l, shorten=4pt,yshift=-1ex] \ar[l, shorten=4pt, yshift=1ex] \cdots  
\end{tikzcd} \]
Associated to this we have the $R$-based Adams spectral sequence (sometimes called the descent spectral sequence) which is of signature 
\[
E_{1}^{s, t} \coloneqq \pi_{t}(R^{\otimes s+1}) \Rightarrow \pi_{t-s} \Ss.
\]
In practice, the $E_{1}$-term above is enormous and a first, crucial, goal is to identify the $E_{2}$-term in terms of homological algebra.\footnote{To drive our point home, let us quote Ravenel \cite{ravenel_complex_cobordism}: "The cobar complex [\emph{the $E_{1}$-term}] is so large that one wants to avoid using it directly at all costs."} 

The standard assumption that allows one to do this is descent-flatness; that is, asking that $R_{*}R$ be flat as an $R_{*}$-module. If this is the case, one can show that the pair $(R_{*}, R_{*}R)$ becomes a Hopf algebroid in graded commutative rings. We then have a canonical identification 
\[
E_{2}^{s, t} \cong \Ext_{R_{*}R}^{s, t}(R_{*}, R_{*}),
\]
where the $\Ext$-groups are computed in $R_{*}R$-comodules. It is this identification that strongly ties homotopy theory to algebraic geometry: the Hopf algebroid presents a stack and the category of comodules can be identified with quasi-coherent sheaves on this stack\footnote{To be more precise, since $(R_{*}, R_{*}R)$ is a Hopf algebroid in \emph{graded-commutative} rings, the associated stack is Dirac stack in the sense of \cite{hesselholt2023dirac}.}, relating the homological algebra of the latter to the stable homotopy groups of spheres. 

Perhaps the most important example is given by taking $R$ to be the complex bordism spectrum $\mathrm{MU}$, for which the associated stack is the moduli stack of formal groups. This creates a bridge between homotopy theory and arithmetic geometry,  leading to the subject of chromatic homotopy theory. 

Unfortunately, the descent-flatness assumption fails in many important examples, such as topological modular forms, real $K$-theory, truncated Brown-Peterson spectra and integral homology. Some of these ring spectra are closer to the sphere than either $\mathrm{MU}$ or $\fieldp$, so that the corresponding Adams spectral sequence would be very efficient, yet the lack of a uniform description of the corresponding $E_{2}$-term makes large-scale calculations incredibly complicated. 

The standard strategy in these cases is to start the analysis at the $E_{1}$-term. While involved, this is a powerful technique, the famous example being Mahowald's celebrated resolution of the telescope conjecture at $p = 2$ and $n=1$ through the use of the $ko$-based Adams spectral sequence  \cite{mahowald1981bo, mahowald1982image}. A similar analysis at odd primes appears in the work of Gonzalez \cite{gonzalez_regular_complex, gonzalez2000vanishing}. At height two, $\mathrm{tmf}$-resolutions were studied using this technique in the work of B\`{e}audry, Behrens, Bhattacharya, Culver and Xu \cite{beaudry2019tmf, beaudry2021telescope}. 

In this paper we give a new method of identifying the $E_{2}$-terms of these Adams spectral sequences in terms of homological algebra of quiver representations. %\subsection{Statement of results} 
For simplicity and to keep things relatively familiar, in this introduction we will state the results under the assumption that $R$ is an $\mathbb{E}_{\infty}$-ring spectrum. The main body of the paper is written only under the assumption that the ring spectrum in question is associative (which is important, as many examples of interest are not commutative, such as connective Morava $K$-theory), but the results are easier to state in the fully commutative case. 

\begin{definition}
\label{definition:introduction_multiplicative_class_of_simples}
  A \emph{a multiplicative class of simples} $\pcat$ is a symmetric monoidal, full subcategory of compact $R$-modules 
  closed under finite sums, (de)suspensions, retracts and linear duals. 
  A \emph{representation of $\pcat$} is an additive functor
  \[ \X: \pcat \rightarrow \abeliangroups \]
  to the category of abelian groups.\footnote{In the main body of the text, we define representations as \emph{contravariant} functors, rather than covariant, see \cref{definition:pcat_quiver}. Since in the introduction we assume that $R$ is commutative and $\pcat$ is closed under taking duals, covariant and contravariant functors can be identified, see \cref{variant:covariant_description_of_quivers}.}\footnote{Note that since $\abeliangroups$ is a $1$-category, $\Rep(\pcat)$ depends only on the homotopy category of $\pcat$.} We say $R$ is \emph{Adams-type} with respect to $\pcat$ if every simple $S \in \pcat$ can be written $S \simeq \varinjlim S_{\alpha}$ as a filtered colimit of finite spectra such that $R \otimes_{A} S_{\alpha}$ is simple for each $\alpha$. 
\end{definition}

The category $\Rep(\pcat)$ of representations of $\pcat$ is a Grothendieck abelian category, and it inherits a local grading\footnote{A \emph{local grading} on a category is an action of $\bbZ$; equivalently, a choice of an automorphism. In the case of $\pcat$, the chosen automorphism is given by the suspension functor.} and a symmetric monoidal structure from that of $\pcat$. For any spectrum $A$ its $\pcat$-homology $\Hsf(A)$ is the $\pcat$-representation given by 
\[ \Hsf(A)(P) \coloneqq \pi_{0}(P \otimes A). \]
The classical case is recovered by letting $\pcat$ consist of free $R$-modules, in which case there is a canonical equivalence $\Rep(\pcat) \cong \Mod_{R_{*}}$. In general, $\Hsf(A)$ is a richer invariant of $A$ than $R_*(A)$, encoding the homology of $A$ with respect to a wider class of  $R$-modules. 

\begin{theorem}[{\ref{theorem:comodule_homology_theory_adapted_in_pcat_flat_case}}]
\label{theorem:ext_identification_in_general_case_in_introduction}
Let $R$ be $\pcat$-Adams-type. Then there exists a unique exact, cocontinuous comonad $\quiversteenrod$ on $\Rep(\pcat)$ such that 
\[
\quiversteenrod(\Hsf(A)) \simeq \Hsf(R \otimes A)
\]
for any spectrum $A$. Moreover, for any spectrum $A$ the homology $\Hsf(A)$ has a natural structure of a $\quiversteenrod$-comodule and there is a canonical identification
\[
{}^R E_{2}^{s, t}(\Ss, A) \cong \Ext^{s, t}_{\quiversteenrod}(\Hsf(\Ss), \Hsf(A))
\]
between the $E_{2}$-term of the $R$-based Adams spectral sequence computing $\pi_{*}A$ and $\Ext$-groups in the abelian category of $\quiversteenrod$-comodules. 
\end{theorem}

Note that the condition that $R$ is $\pcat$-Adams-type is always satisfied if $\pcat$ is sufficiently large. However, the power of \cref{theorem:ext_identification_in_general_case_in_introduction} comes from the fact that it is often possible to choose a remarkably small $\pcat$. Once this is done, the representations themselves can be described in a compact manner, allowing one to make explicit calculations.

The category $\Rep(\pcat)$ of quiver representations can be informally thought of as ``modules over a ring with many objects''. Through Morita theory, the cocontinuous functor $\quiversteenrod$ can be interpreted as tensoring with a $\pcat$-bimodule, and the comonad multiplication on $\quiversteenrod$ as endowing the bimodule with the structure of a Hopf algebroid, mimicking the classical picture. In particular, the category of comodules has a canonical symmetric monoidal structure for which the forgetful functor into representations is strongly symmetric monoidal. 

\subsection{Integral homology}
As a concrete example, let us focus on the case of $p$-local integral homology. 
It is well-known that $\pi_*(\Z_{(p)} \otimes \Z_{(p)})$ contains torsion, and hence is not flat as a graded abelian group, dooming the standard approach to identifying its $E_{2}$-term. On the other hand, the only torsion which appears is simple $p$-torsion, suggesting the following class of simples.

\begin{notation}
Let $\pcat$ denote the collection of $\Z_{(p)}$-modules which are equivalent to a finite sum of shifts of $\bbZ_{(p)}$ and $\fieldp$. This is a multiplicative class of simples in the sense of \cref{definition:introduction_multiplicative_class_of_simples}.
\end{notation}

Additivity implies that a $\pcat$-representation $\X$ can be identified with a diagram of the form 
\[
\begin{tikzcd}
	{\X(\bbZ_{(p)})} & {\X(\fieldp) }
	\arrow["\pi", bend left, from=1-1, to=1-2]
	\arrow["\delta", bend left, from=1-2, to=1-1]
\end{tikzcd}
\]
where $X(\Z_{(p)})$ is a graded abelian group, 
$X(\F_p)$ is a graded $\F_p$ vector space, 
$\pi$ has degree $0$ and
$\delta$ has degree $-1$, subject to the relation 
$\delta \circ \pi = 0$. Thus, $\pcat$-representations are graded representations of a quiver with relations, which suggests thinking of $\Rep(\pcat)$ as a category of quiver representations and of $\Hsf(-)$ as a multi-object homology theory, or a homology theory landing in representations of some quiver. The representation theory of this particular quiver is well-understood, and arises in various places, for examples as blocks in representation theory of Schur algebras, $\mathfrak{sl}_{2}(\mathbb{C})$, or the Virasoro algebra \cite[\S 6.1]{chan2019representation}, \cite{stroppel2003category}.

Concretely, the $\pcat$-homology of a spectrum $X$ is given by the quiver representation
\[
\begin{tikzcd}
	{\Hrm_{*}(X; \bbZ_{(p)})} & { \Hrm_{*}(X; \fieldp) }
	\arrow["\pi", bend left, from=1-1, to=1-2]
	\arrow["\delta", bend left, from=1-2, to=1-1]
\end{tikzcd}
\]
where $\pi$ is the reduction mod $p$ map and $\delta$ is the integral $p$-Bockstein.
In this case \cref{theorem:ext_identification_in_general_case_in_introduction} specializes to the statement that any such $\pcat$-representation has a coaction of an exact comonad $\quiversteenrod$\footnote{One can show that the comonad $\quiversteenrod$ is given by tensoring with the ``quiver integral dual Steenrod algebra'', which is the representation $\Hrm_{*}(\bbZ; \bbZ_{(p)}) \rightleftarrows \Hrm_{*}(\bbZ; \fieldp)$ attached to $\bbZ$ itself. However, to make this precise one should treat this representation as a $\pcat$-bimodule, which is probably more trouble than its worth. As our results show, it is usually easier to work directly with the comonad, omitting the language of Hopf algebroids.} and that we have an identification of the $\bbZ_{(p)}$-Adams $E_{2}$-term 
\[
{}^{\Z_{(p)}} E_{2}^{s, t}(\Ss, X) \cong \Ext^{s, t}_{\quiversteenrod} \left( \begin{tikzcd} \Z_{(p)} \ar[d, bend left, two heads] \\ \F_p \ar[u, bend left, "0"] \end{tikzcd}, \begin{tikzcd} \Hrm_{*}(X; \bbZ_{(p)}) \ar[d, bend left] \\ \Hrm_{*}(X; \fieldp) \ar[u, bend left] \end{tikzcd} \right).
\]

After the $E_{2}$-term has been given an identification in the language of homological algebra it becomes much easier to analyze how various constructions differ and comparisons with more classical notions becomes fruitful.  
In the case of integral homology, this works remarkably well.

%of well-behaved Grothendieck abelian categories has been made, one can make comparisons with more classical notions.  

For each spectrum $X$, the $\F_p$-homology $\Hrm_{*}(X; \fieldp)$ has a coaction of the dual Steenrod algebra $\acat_{*}$, encoding the action of the Steenrod operations. For formal reasons, this coaction is compatible with the coaction of the comonad $\quiversteenrod$ on $\Hsf(X)$ in the sense that there is a forgetful functor
$ \Comod_\quiversteenrod \to \Comod_{\acat_*} $
and the action of $\pi \circ \delta$ on $\Hsf(X)(\F_p)$ 
coincides with the action of the Bockstein, $\beta$, 
on the $\F_p$-homology of $X$.
Encoding the action of the Bockstein via the quotient Hopf algebra 
$\acat(0)_{*}$ we obtain a commuting square of comparison functors:

\begin{equation}
\label{equation:comparison_functor_between_abelian_pullback_in_introduction}
\begin{tikzcd}
\Comod_{\quiversteenrod} \ar[r] \ar[d] &
\Comod_{\acat} \ar[d] \\
\Rep(\pcat) \ar[r] &
\Comod_{\acat(0)}
\end{tikzcd}
\end{equation}

%from quiver comodules into witnessed Bockstein modules with a compatible $\acat_{*}$ coaction on $V$. 

Perhaps surprisingly,
this approximation captures all of the structure of a $\quiversteenrod$-comodule, even at the derived level.

\begin{theorem}
\label{theorem:introduction_embedding_into_pullback-of_derived_infty_cats}
The comparison functors in  (\ref{equation:comparison_functor_between_abelian_pullback_in_introduction}) induce a fully faithful embedding  
\[
\dcat^{b}(\Comod_{\quiversteenrod}) \hookrightarrow \dcat^{b}(\Rep(\pcat)) \times _{\dcat^{b}(\acat(0))} \dcat^{b}(\acat)
\]
of the bounded derived category of $\Comod_{\quiversteenrod}$
into the pullback of the 
bounded derived categories of 
$\acat_*$-comodules and $\pcat$-representation over $\acat(0)_*$-comodules.

Associated with this embedding we obtain, for any $\quiversteenrod$-comodules $\X, \Y$, a Mayer--Vietoris long exact sequence of $\Ext$-groups 
\[
0 \rightarrow \Ext^{0}_{\quiversteenrod}(\X, \Y) \rightarrow \Ext^{0}_{\pcat}(\X, \Y) \oplus \Ext^{0}_{\acat}({\X}, {\Y}) \rightarrow \Ext_{\acat(0)}^{0}({\X}, {\Y}) \rightarrow  \Ext^{1}_{\quiversteenrod}(\X, \Y) \rightarrow \cdots.
\]
\end{theorem}

As the $\bbZ_{(p)}$-Adams $E_{2}$-term can be identified with $\Ext$-groups in $\quiversteenrod$-comodules by \cref{theorem:ext_identification_in_general_case_in_introduction}, the above result gives a concrete way to calculate it in more classical terms.

The Adams spectral sequence can be categorified using the theory of synthetic spectra \cite{pstrkagowski2018synthetic}, which is a deformation of stable $\infty$-categories whose special fibre can be identified with the derived $\infty$-category of comodules. In \S\ref{section:synthetic_spectra_based_on_quivers} we extend the synthetic construction to quiver-based homology theories, and show that the embedding of \cref{theorem:introduction_embedding_into_pullback-of_derived_infty_cats} can be strenghtened to an embedding of deformation $\infty$-categories. 

In more detail, in \cref{cnstr:comparison-square} we construct a commutative square of 
\begin{equation}
\label{sq:main}
  \begin{tikzcd}
    \Syn_{\pcat}(\Sp) \ar[r] \ar[d] & \Syn_{\F_p}(\Sp) \ar[d] \\
    \Syn_{\pcat}(\Mod_\Z) \ar[r] & \Syn_{\F_p}(\Mod_{\Z})
  \end{tikzcd}
\end{equation}
of synthetic $\infty$-categories associated to either $\pcat$- of $\fieldp$-homology on spectra and $\bbZ$-modules. Each arrow in the square is a symmetric monoidal left adjoint and the square itself is right adjointable by the general theory developed in \cref{sec:def}. 

\begin{theorem}[{\ref{thm:pullback-Z}}] \label{thm:integral-main}
The comparison functor 
\[
    \Syn_{\pcat}(\Sp) \hookrightarrow \Syn_{\F_p}(\Sp) \times_{\Syn_{\pcat}(\Mod_\Z)}  \Syn_{\F_p}(\Mod_{\Z})
\]
induced by (\ref{sq:main}) is fully faithful. 
\end{theorem}

Stated less formally, \Cref{thm:integral-main} says that  
the $\Z_{(p)}$-Adams spectral sequence for $X$
is assembled in a straightforward way from
the $\F_p$-Adams spectral sequence for $X$, the $\F_p$-Bockstein spectral sequence computing $\Hrm_*(X; \Z_{(p)})$, and the integral homology $\Hrm_*(X; \Z_{(p)})$ itself.

As consequences of \Cref{thm:integral-main} we obtain the following results
which illuminate the relationship between the $\Z_{(p)}$-based and $\F_p$-based Adams spectral sequences.

\begin{enumerate}
\item[(\ref{cor:Fp-Z-iso})]
Let $X$ be a spectrum whose integral homology groups are simple $p$-torsion. Then, the comparison map 
\[ {}^{\Z_{(p)}}E_r^{s,t}(X) \to {}^{\F_p}E_r^{s,t}(X) \]
between the $\bbZ_{(p)}$-based Adams spectral sequence and 
the $\fieldp$-based Adams spectral sequence
is an isomorphism for each $r \geq 2$.
\item[(\ref{proposition:mv_lss_between_hsf_and_hfp_e2_terms})] For any spectrum $X$, the canonical comparison morphism
%\begin{equation}
%\label{equation:comparison_between_hsf_and_hfp_ass_in_mv_seq_in_introduction}
\[ {}^{\bbZ_{(p)}}{E}_2(\Ss, X) \rightarrow {}^{\fieldp}{E}_2(\Ss, X) \]
%\end{equation}
between the $\bbZ_{(p)}$-based and $\fieldp$-based Adams $E_{2}$-terms 
computing $\pi_{*}X$ 
%and the algebraic comparison map 
%\begin{equation}
%\label{equation:comparison_between_abm_and_bm_ext_in_mv_seq_in_introduction}
%\Ext_{w \Vect_{\beta}}(\Hsf(S^{0}), \Hsf(X)) \rightarrow \Ext_{\acat(0)}%(\fieldp, \Hrm_{*}(X; \fieldp))
%\end{equation}
fit into a Mayer--Vietoris long exact sequence
\[
\cdots \rightarrow {}^{\bbZ_{(p)}}{E}_2^{s, t}(X) \rightarrow \Ext^{s, t}_{\pcat}(\Hsf(X)) \oplus {}^{\fieldp}{E}_2^{s, t}(X) \rightarrow \Ext^{s, t}_{\acat(0)}(X) \rightarrow {}^{\bbZ_{(p)}}{E}_2^{s+1, t}(X)  \rightarrow \cdots.
\]
%In particular, if (\ref{equation:comparison_between_abm_and_bm_ext_in_mv_seq_in_introduction}) are isomorphisms, then so is (\ref{equation:comparison_between_hsf_and_hfp_ass_in_mv_seq_in_introduction}).  
\item[(\ref{theorem:e2_page_of_the_sphere})] The $E_2$-page of the $\Z_{(p)}$-Adams spectral sequence for $\Ss$ is given by
\[
{}^{\bbZ_{(p)}}{E}_2^{s, t}(\Ss) \cong \begin{cases}
  \bbZ_{(p)}  & t-s = 0,\quad s = 0 \\
  0 & t-s = 0,\quad s \neq 0 \\
  \Ext_{\acat}^{s, t}(\fieldp; \fieldp) & t-s \neq 0
\end{cases}
\]
and the comparison map 
\[
{}^{\bbZ_{(p)}}{E}_2^{s, t}(\Ss) \rightarrow {}^{\fieldp}{E}_2^{s, t}(\Ss)
\]
between $\bbZ_{(p)}$-based and $\fieldp$-based Adams spectral sequences for the sphere is an isomorphism outside of the $0$-stem.     
\end{enumerate}

%We have a canonical isomorphism 
%\[
%\tensor[^{H\bbZ}]{E}{_2^{s, t}}(S^{0}) \cong \Ext^{s, t}_{w \Vect_{\beta}}(\Hsf(S^{0})) \times_{\Ext^{s, t}_{\acat(0)}(\fieldp; \fieldp)} \Ext^{s, t}_{\acat}(\fieldp; \fieldp) \cong \mathbb{Z} \times_{\fieldp[h_{0}]} \Ext^{s, t}_{\acat}(\fieldp; \fieldp)
%\]

%Of particular significance is that the $\Ext$-groups of the dual Steenrod algebra $\acat$ can be themselves identified with the $E_{2}$-page of the $\fieldp$-Adams spectral sequence. Restating the above result in this language gives the following result. 

%\begin{theorem}[{}]
%\label{theorem:mv_lss_between_hsf_and_hfp_e2_terms_in_introduction}
%\end{theorem}

%Since the category of witnessed Bockstein modules is a presheaf category, in practice its $\Ext$-groups are very easy to compute - for example, they always vanish in positive degree if $\Hrm_{*}(X; \bbZ_{(p)})$ is torsion-free. Thus, \cref{theorem:mv_lss_between_hsf_and_hfp_e2_terms_in_introduction} allows one to easily compute the $H\bbZ_{(p)}$-Adams $E_{2}$-term from the knowledge of the $H\fieldp$-one; in the particular case of the sphere, we obtain the following result. 

%\begin{theorem}[{\ref{theorem:e2_page_of_the_sphere}}]
%\label{theorem:e2_page_of_the_sphere_in_introduction}
%\end{theorem}

% \subsection{The integral Adams spectral sequence and the construction of $\BP$}

While the differences between the $\bbZ_{(p)}$-based and $\F_p$-based 
Adams spectral sequences for the sphere are slight, there are other spectra for which these two spectral sequences differ substantially.
As an illustration of this, in \S\ref{subsec:BP}
we consider the $\Z_{(p)}$-based Adams spectral sequence for $\BP$, the Brown--Peterson spectrum.

This spectral sequence differs from the $\F_p$-Adams spectral sequence for $\BP$ in two essential ways:
First, since the Hurewicz map $\pi_*\BP \rightarrow \Hrm_{*}(\BP; \bbZ_{(p)})$ is a monomorphism, all classes from $\pi_*\BP$ are detected on the $0$-line.
Second, unlike its classical counterpart which collapses on the second page, the $\bbZ_{(p)}$-Adams for $\BP$ has non-zero differentials of arbitrary length.

Despite its somewhat more involved nature, the integral Adams $E_{2}$-page for $\BP$ does have some advantages over its $\fieldp$-counterpart. Recall that the first proof of existence of this spectrum, due to Brown and Peterson themselves \cite{brown1965spectrum}, proceeds by an inductive construction based on the homology.  A natural guess is that their arguments are an instance of Toda's obstruction theory for the existence of a spectrum with prescribed mod $p$ homology \cite{toda1971spectra}, but this is not quite the case---the relevant obstruction groups in fact do \emph{not} vanish. 

In order to overcome this, Brown and Peterson go beyond the mod $p$ homology---they also maintain control over the \emph{integral} homology of the spectrum they're inductively constructing as well as the relation between the two given by the Bockstein. This is precisely the information visible from the point of view of the $\bbZ_{(p)}$-Adams spectral sequence.
Inspired by this we show that the $\Z_{(p)}$-based obstruction groups do indeed vanish, recovering the original proof of Brown and Peterson in a systematic manner. 

\begin{proposition}[{\ref{theorem:integral_todas_obstructions_groups_for_bp_vanish}}]
\label{theorem:integral_todas_obstructions_groups_for_bp_vanish_in_introduction}
The $\bbZ_{(p)}$-based Toda obstruction groups
\[
\Ext^{n+2, n}_{\quiversteenrod}(\Hsf(BP), \Hsf(BP))
\]
to the construcion of the Brown-Peterson spectrum vanish for all $n \geq 0$.
\end{proposition}

\subsection{Further examples}
In \S\ref{sec:exm} we consider some further examples of quiver homology theories for non-Adams-type ring spectra. In particular, we discuss:
\begin{enumerate}
    \item connective Morava $K$-theory $k(n)$, which similarly to the case of integral homology admits a class of simples with only two generating modules, 
    \item the truncated Brown--Peterson spectra $\BPn{n}$, where the problem of finding a suitable class of simples is closely related to the classical problem of decomposing $\BPn{n} \otimes \BPn{n}$, as studied by Mahowald, Kane and Tatum \cite{mahowald1981bo, kane1981operations, tatum2022spectrum}. 
    \item Bousfield's united $K$-theory \cite{bousfield1990classification}, which is an important historical antecedent to the current work. 
\end{enumerate}

\subsection{The calculus of deformations}

At its technical heart this paper relies on the theory of deformations of stable $\infty$-categories along homology theories developed by Patchkoria and the second author in \cite{patchkoria2021adams}.
In an appendix we provide a review of this theory as well as a further development of the $(\infty,2)$-categorical functoriality of the constructions from \cite{patchkoria2021adams}.

\subsection*{An outline of the paper}

In \S\ref{section:quiver_homology_theories} we set up the basic theory of quiver homology theories.
In \S\ref{section:synthetic_spectra_based_on_quivers} we construct categories of synthetic spectra based on quiver homology theories satisfying an Adams-type condition.
In \S\ref{sec:integral} we analyze the quiver homology theory $\Hsf(-)$ which enhances $\Z_{(p)}$-homology, proving \Cref{thm:integral-main}.
In \S\ref{sec:exm} we consider other examples where non-Adams-type ring spectra can be equipped with an Adams-type quiver homology theory.
In \Cref{sec:def} develop technical machinery for manipulating deformations of stable $\infty$-categories, this material is mainly used in \S\ref{sec:integral} where it allows us to give relatively straightforward proofs.

\subsection*{Acknowledgments}

The authors would like to thank Jeremy Hahn, Lars Hesselholt, Pengkun Huang, Sven van Nigtevecht and Dylan Wilson for helpful conversations. We would like to thank the anonymous referee for their comments and suggestions. 

During the course of this work the first author was supported by NSF grant DMS-2202992 and by the Danish National Research Foundation through the Copenhagen Centre for Geometry and Topology (DNRF151). The second author received support from NSF Grant number DMS-1926686 and Deutsche Forschungsgemeinschaft under Germany's Excellence Strategy – EXC-2047/1 – 390685813.

\subsection*{Notations and Conventions}

Throughout the body of the paper the term \emph{category} will refer to an $\infty$-category as developed by Joyal and Lurie.
In some places we will use the term $1$-category is shorthand for a $1$-truncated category.
We will generally assume the read is familiar with higher algebra as developed in \cite{higher_algebra}.

\begin{enumerate}
\item Our notation for a prestable category will typically      
use a subscript ``$\geq 0$'' with the convention that dropping the subscript corresponds to passing to the corresponding Spanier--Whitehead category.
As an example we might write $\DD_{\geq 0}$ for a prestable category and $\DD$ for its 
colimit stabilization.
\item Given an $\E_1$-algebra in category $\ccat$ we will write $\Mod_R(\ccat)$ for the category of \emph{left} $R$-modules in $\ccat$. When working with right modules we will be explicit about it. 
\end{enumerate}

\section{Quiver homology theories} 
\label{section:quiver_homology_theories}

In this section we develop the basic theory of classes of simples and their associated quiver homology theories.

\begin{notation}
Throughout this section, $R$ will denote an $\E_{1}$-algebra in spectra and the word \emph{module} will mean a left $R$-module.
\end{notation}

\subsection{Quivers of simple modules} 
\label{subsection:quiver_of_simple_modules}

In this section, we develop the basic formalism of quiver homology theories on $R$-modules. 

\begin{definition}
\label{definition:class_of_simple_modules}
We say a full subcategory, $\pcat$, of compact $R$-modules is \emph{a class of simples} if it is closed under 
finite direct sums, (de)suspensions and retracts. 
\end{definition}

\begin{example}
\label{example:all_perfects_form_a_class_of_simples}
The category $\Mod_R^{\omega}$ of all compact $R$-modules is a class of simples. 
\end{example}

\begin{example}
\label{example:finite_free_modules_form_a_class_of_simples}
The category $\Mod_{R}^{\mathrm{fp}}$ of \emph{finite projective} $R$-modules; that is, retracts of finite sums of $\Sigma^{n} R$ for varying $n$, is a class of simples.
\end{example}

Associated to any class of simples we have its category of representations. 

\begin{definition}
\label{definition:pcat_quiver}
Let $\pcat$ be a class of simples. A \emph{representation of the quiver $\pcat$} is an additive presheaf
\[
\X\colon \pcat^{op} \rightarrow \abeliangroups.
\]
We denote the category of representations of $\pcat$ and natural transformations by $\Rep(\pcat)$.
\end{definition}

\begin{construction}
\label{example:associated_quiver}
Let $\pcat$ be a class of simples and 
let $M$ be an $R$-module. 
Then, the \emph{associated representation} $\Psf(M)$ is defined by 
\[
\Psf(M)(S) \coloneqq \pi_0\Map_R(S, M)
\]
\end{construction}

The category $\Rep(\pcat)$ of representations has a natural local grading (that is, an action of $\bbZ$) 
\[
[1] \colon \Rep(\pcat) \rightarrow \Rep(\pcat)
\]
given by $ (Q[1])(S) \colonequals Q(\Sigma^{-1} S)$. 
The following two remarks now explain how objects of $\Rep(\pcat)$ can be viewed as quiver representations in graded abelian groups. 

\begin{remark}[Values in graded abelian groups]
Since the category of representations of $\pcat$ has a local grading, for any representation $\X$ we can identify $\X(S)$ with the internal degree zero part of the graded abelian group 
\[
\X(S)_{n} \coloneqq (\X[n])(S) \cong \X(\Sigma^{-n} S)
\]
We will often blur the distinction and think of $\pcat$-reps as taking values in graded abelian groups. Note that according to this convention, for an $R$-module $M$ we have 
\[
\Psf(M)(S)_{*} \colonequals \pi_*\Map_R(S,M).
\]
\end{remark}

\begin{remark}[$\pcat$-reps as quiver representations]
As the terminology suggests, $\pcat$-reps can be visualised as certain types of quivers representations valued in graded abelian groups. To see this, assume that we are in the situation where $\pcat$ is generated under finite direct sums and (de)suspensions by the simples $S_{1}, S_{2}, \ldots, S_{k}$. 

As any $\pcat$-rep $\X\colon \pcat \rightarrow \Ab$ is by definition additive, it follows that its values are determined by $\X(S_{i})$, considered as a graded abelian group. Thus, any such $\X$ determines and is uniquely determined by the following data: 
\begin{enumerate}
    \item for any $i$, a graded abelian group $\X(S_{i})_{*}$ and 
    \item for any $i, j$, a morphism of graded abelian groups 
    \[
    \X(S_{j})_{*} \otimes \pi_*\Map_R(S_{i}, S_{j}) \rightarrow \X(S_{i})_{*}.
    \]
\end{enumerate}
The latter morphisms have to be compatible under composition, and identities have to act by the identity. 

As an explicit example, suppose that we have three simples $S_{1}, S_{2}, S_{3}$ which generate $\pcat$. In this case, we can identify any $\X$ with a quiver representations of the form 
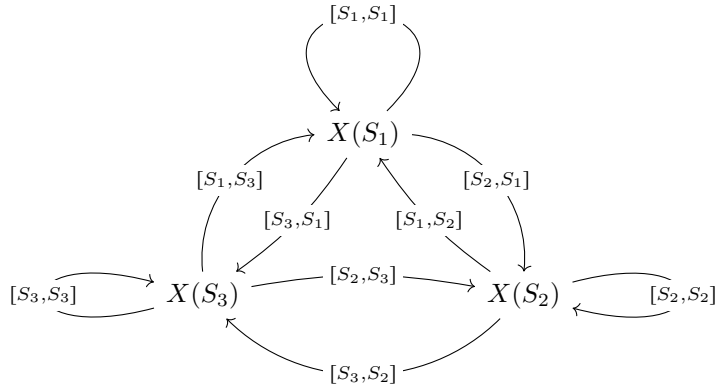
\[
\begin{tikzcd}
	& {X(S_{1})} \\
	& {} \\
	{X(S_{3})} && {X(S_{2})}
	% loops
	\arrow["{[S_{1}, S_{1}]}"{description}, loop, looseness=10, from=1-2, to=1-2]
	\arrow["{[S_{3}, S_{3}]}"{description}, loop left, looseness=15, from=3-1, to=3-1]
	\arrow["{[S_{2}, S_{2}]}"{description}, loop right, looseness=15, from=3-3, to=3-3]
	% non-loops
	\arrow["{[S_{1}, S_{3}]}"{description}, bend left=45, from=3-1, to=1-2]
	\arrow["{[S_{3}, S_{1}]}"{description}, bend left=10, from=1-2, to=3-1]
	\arrow["{[S_{3}, S_{2}]}"{description}, bend left=45, from=3-3, to=3-1]
	\arrow["{[S_{2}, S_{3}]}"{description}, bend left=10, from=3-1, to=3-3]
	\arrow["{[S_{2}, S_{1}]}"{description}, bend left=45, from=1-2, to=3-3]
	\arrow["{[S_{1}, S_{2}]}"{description}, bend left=10, from=3-3, to=1-2]
\end{tikzcd}
\]
In practice, for specific choices of $\pcat$ we can often write down a compact presentation of the relevant quiver in terms of finitely many generating maps and relations.
\end{remark}

\begin{lemma}[Yoneda lemma]
\label{lemma:yoneda_lemma_for_quivers}
For any simple $S \in \pcat$ and any $X \in \Rep(\pcat)$, we have 
\[
\Hom_{\Rep(\pcat)}(\Psf(S), X) \cong X(S).
\]
\end{lemma}

\begin{notation}
We use caligraphic font for quivers and sans serif for representations of quivers. In the case of the representation associated to an $R$-module, we think of $\Psf$ as being a sans serif version of $\pi$, as $\Psf(M)$ is a version of $\pcat$-indexed homotopy groups. 
\end{notation}

The category of representations of a quiver has excellent categorical properties, as we now show. We recall from \cite[Definition 2.8]{patchkoria2021adams} that a \emph{homology theory} $\Hsf \colon \ccat \rightarrow \acat$ is a functor from a stable category (equipped with a local grading coming from the suspension) into a locally graded abelian category, with the property that if $c \rightarrow d \rightarrow e$ is a cofibre sequence, then $\Hsf(c) \rightarrow \Hsf(d) \rightarrow \Hsf(e)$ is exact in the middle. 

\begin{proposition}
\label{proposition:cat_of_quiveres_groth_abelian_and_psf_a_homology_theory}
The category $\Rep(\pcat)$ is Grothendieck abelian and compactly generated by objects of the form $\Psf(S)$ where $S \in \pcat$. 
Moreover, $\Psf\colon \Mod_{R} \rightarrow \Rep(\pcat)$ is a homology theory; 
that is $\Psf(-)$ is compatible with the local grading in the sense that $\Psf(\Sigma M) \cong \Psf(M)[1]$.
\end{proposition}

\begin{proof}
Each $\Psf(S)$ for $S \in \pcat$ is compact projective by \cref{lemma:yoneda_lemma_for_quivers}, and since $\Rep(\pcat)$ is presentable as a presheaf category it is necessarily Grothendieck abelian. The second part is clear, as the Yoneda lemma forces the local grading to be of this form on compact projectives. 
\end{proof}

\begin{remark} \label{rmk:identify-projectives}
Since we assume that $\pcat$ is closed under retracts, any compact projective of $\Rep(\pcat)$ is isomorphic to one the form $\Psf(S)$ for a simple $S$. 
\end{remark}

We recall from \cite[{\S 2.2}]{patchkoria2021adams} that a homology theory $\Hsf \colon \ccat \rightarrow \acat$ is called \emph{adapted} if for every injective object $\Isf \in \acat$ there exists an object $\Isf_{\ccat}$ together with an isomorphism 
\begin{equation}
\label{equation:adaptedness_condition}
\pi_{0} \Map_{\ccat}(-, \Isf_{\ccat}) \cong \Hom_{\acat}(\Hsf(-), \Isf)
\end{equation}
such that the induced counit map $\Hsf(\Isf_{\ccat}) \rightarrow \Isf$ is an isomorphism. This is a natural condition (going back to the work of Moss \cite{moss1968composition}) which allows one to construct an $\Hsf$-based Adams spectral sequence of signature 
\[
\Ext_{\acat}^{s, t}(\Hsf(c_{1}), \Hsf(c_{2})) \Rightarrow \pi_{t-s} \Map_{\ccat}(c_{1}, c_{2})
\]
see \cite[{Construction 2.24}]{patchkoria2021adams}. We have the following standard result: 

\begin{proposition}
\label{proposition:psf_homology_theory_on_r_modules_is_adapted}
The homology theory $\Psf\colon \Mod_{R} \rightarrow \Rep(\pcat)$ is adapted. 
\end{proposition}

\begin{proof}
Since each $S \in \pcat$ is compact as an $R$-module and $\Psf$ preserves arbitrary direct sums, for any $\Isf \in \Rep(\pcat)$, the left hand side of (\ref{equation:adaptedness_condition}) satisfies the conditions of Brown representability, so that the needed $R$-module $\Isf_{\Mod_{R}}$ exists. To verify the counit condition, we observe that for any $S \in \pcat$ we have 
\[
\Isf(S) \cong \Hom_{\Rep(\pcat)}(\Psf(S), \Isf) \cong \pi_0\Map_R(S, \Isf_{\Mod_{R}}) \cong \Psf(\Isf_{\Mod_{R}})(S),
\]
where the first equivalence is \cref{lemma:yoneda_lemma_for_quivers}, the second one is the defining property of $\Isf_{\Mod_{R}}$ and the last one is the definition of the homology theory $\Psf$. 
\end{proof}

\begin{variant}[Covariant description of quivers]
\label{variant:covariant_description_of_quivers}
One can give a covariant description of $\pcat$-reps, under which the homology theory $\Psf$ takes the familiar tensor product formula. If $S$ is a (left) $R$-module, then its $R$-linear dual 
\[
\mathbb{D}^{R}S \colonequals \Hom_{R}(S, R)
\]
is canonically a right $R$-module. Restricted to compact objects, this yields an equivalence 
\[
(\Mod(R)^{\omega})^{\op} \cong \Mod(R^{op})^{\omega}
\]
between the opposite of compact modules over $R$ and the category of compact modules over the opposite ring. Let us write $\pcat^{\prime}$ for the essential image of $\pcat$ under this equivalence; this is the subcategory of those compact $R^{op}$-modules whose linear dual is simple. 

Through the above equivalence, we can identify $\Rep(\pcat)$ with the category of additive, covariant functors 
\[
\X\colon \pcat^{\prime} \rightarrow \abeliangroups. 
\]
Under this description, the homology theory $\Psf$ is given by the familiar formula
\[
\Psf(M)(S^{\prime}) \coloneqq \pi_{0}(S^{\prime} \otimes_{R} M).
\]
\end{variant}

The following will be useful later. 

\begin{lemma}
\label{lemma:p_epis_of_simple_modules_are_split}
Let $p \colon S \rightarrow S'$ be a map of simple $R$-modules such that $\Psf(S) \rightarrow \Psf(S')$ is an epimorphism of $\pcat$-representations. Then $p$ is a split epimorphism. 
\end{lemma}

\begin{proof}
By assumption, $\mathrm{id}_{S'} \in \Psf(S')(S')$ is in the image of $\Psf(S)(S') \cong \pi_{0} \Map_{\Mod_{R}}(S', S)$. A lift gives the required section. 
\end{proof}

\subsection{Homology operations} 
\label{subsection:homology_operations_for_general_quivers}

In the previous section, we associated to an $\E_{1}$-algebra $R$ together with a class of simple modules $\pcat$ an adapted homology theory 
\[
\Psf\colon \Mod_{R} \rightarrow \Rep(\pcat). 
\]
Now suppose that we have a map of $\E_1$-algebras $f \colon A \rightarrow R$. By precomposing $\Psf$ with extension of scalars along $f$ we obtain a homology theory on $A$-modules 
\begin{equation}
\label{equation:induced_homology_theory_on_a_modules}
\Psf(R \otimes_{A} -) \colon \Mod_{A} \rightarrow \Rep(\pcat).
\end{equation}
As in the classical case, $\pcat$-representations of the form $\Psf(R \otimes_{A} X)$ have additional structure
coming from homology operations. In this section, we describe this additional structure explicitly using the technology developed in the \hyperref[sec:appendix]{Appendix A}. 

For a ring spectrum $E$, the usual strategy to encode $E$-homology operations on spectra  is by giving $E_*X$ the structure of an $E_*E$-comodule. Note that in order to do this we must first assume that $E_*E$ is flat as an $E_*$-module. We follow an analogous approach for quiver-based homology theories under a corresponding flatness assumption.

We recall that a map $M \rightarrow N$ of $R$-modules is called a $\Psf$-epimorphism if the induced map $\Psf(M) \rightarrow \Psf(N)$ is an epimorphism. As we review in \cref{recollection:epimorphism_classes_and_homology_theories}, as an adapted homology theory $\Psf$ is completely determined by the associated class of $\Psf$-epimorphisms. 

\begin{definition}
\label{definition:pflat_ring_spectrum}
We say that a map $f \colon A \rightarrow R$ of $\mathbb{E}_{1}$-algebras is \emph{$\Psf$-flat} if for any $\Psf$-epimorphism $M \rightarrow N$ of $R$-modules, $R \otimes_{A} M \rightarrow R \otimes_{A} N$ is again a $\Psf$-epimorphism. 
\end{definition}

\begin{remark}[Why ``flat''?] \label{rmk:why-flat}
Consider the case $A=\Ss \to R$ and let $\pcat$ be the collection of finite projective $R$-modules, as in \cref{example:all_perfects_form_a_class_of_simples}, so that we can identify the homology theory $\Psf$ with the homotopy groups functor 
\[
\pi_{*} \colon \Mod_{R} \rightarrow \Mod_{R_{*}}.
\]
Rewriting $R \otimes M$ as $ R \otimes R \otimes_{R} M$ we can construct a K\"{u}nneth spectral sequence
\[
\textnormal{Tor}_{R_{*}}^{s, t}(R_{*}R, \pi_{*}M) \Rightarrow \pi_{s+t}(R \otimes M)
\]
From this we deduce that the unit map $\Ss \rightarrow R$ is $\Psf$-flat if and only if $R_{*}R$ is flat as a right $R_{*}$-module. 
\end{remark}

Note that a map of $\E_{1}$-algebras is flat in the sense of \cref{definition:pflat_ring_spectrum} if and only if the induced extension/restriction of scalars adjunction 
\[
R \otimes_A - : \Mod_{A} \rightleftarrows \Mod_{R} 
\]
is flat in the sense of \cref{definition:flat_adjunction}. This allows us to apply the results of \cref{subsection:induced_homology_theories} and deduce the following: 

\begin{theorem}
\label{theorem:comodule_homology_theory_adapted_in_pcat_flat_case}
\label{proposition:for_at_rings_r_otimes_descends_to_a_quiver_comonad}
\label{notation:tensoring_with_quiversteenrod_comonad}
Let $\pcat$ be a class of simples for $R$ such that $f : A \rightarrow R$ is $\Psf$-flat.  
\begin{enumerate}
\item The comonad $R \otimes_{A} -$ on $R$-modules descends 
uniquely to an exact comonad $\mathcal{Q}$ on $\Rep(\pcat)$.
Explicitly, this is the unique exact comonad for which 
\[ \Psf(R \otimes_{A} M) \cong \quiversteenrod (\Psf(M))\]
naturally in the $R$-module $M$. 
\item For any $A$-module $X$, the $R$-module $R \otimes_{A} X$ has a canonical structure of a comodule over the comonad $R \otimes_{A} -$ on $R$-modules. It follows that $\Psf(R \otimes_{A} X)$ has a canonical structure of a comodule over the comonad $\quiversteenrod$ on $\Rep(\pcat)$. 
Through this we obtain a homology theory $i_*\Psf$ and a map of homology theories
\[\begin{tikzcd}
\Mod_{A} \ar[r, "R \otimes_{A} -"] \ar[d, "i_*\Psf"] & \Mod_R \ar[d, "\Psf"] \\
\Comod_{\quiversteenrod} \ar[r] & \Rep(\pcat). 
\end{tikzcd}\]
\item The homology theory $i_*\Psf$ is adapted.
\end{enumerate}
\end{theorem}

\begin{proof}
    This follows from \Cref{proposition:flat_induction_yields_adjunction_of_adapted_homology_theories}.
\end{proof}

\begin{notation}
When it does not lead to ambiguity we will write $\Psf$ for $i_*\Psf$.
\end{notation}

\begin{remark}
Intuitively, the structure of a $\quiversteenrod$-comodule on $\Psf(R \otimes_{A} M)$ records the action of homology operations and the fact that the lift to comodules is  adapted tells us that we have captured \emph{all} homology operations.
\end{remark}

\begin{warning} 
The abstract formalism of \cite{patchkoria2021adams} guarantees that the composite
\[
\Mod_{A} \rightarrow \Mod_{R} \rightarrow \Rep(\pcat) 
\]
has an adapted replacement whose target is a category of comodules for a left exact comonad on $\Rep(\pcat)$, irrespective of any flatness assumptions on $A \rightarrow R$. 

In practice, when the comonad on $\Rep(\pcat)$ is not exact it becomes exceedingly difficult to work with this adapted replacement. For a classical homology theory $E$ on spectra non-exactness of this comonad corresponds to non-flatness of $E_*E$ and a large body of work exists which analyzes specific examples, see \cite{mahowald1981bo, mahowald1982image, gonzalez_regular_complex, gonzalez2000vanishing, beaudry20202, beaudry2019tmf, beaudry2021telescope, tatum2022spectrum}.  
The heart of \Cref{theorem:comodule_homology_theory_adapted_in_pcat_flat_case},
and of this paper more generally, lies in the observation that 
in many situations where $R_*R$ is not flat there is a reasonable class of simples $\pcat$ for which $\Ss \rightarrow R$ \emph{is} $\pcat$-flat, allowing one to make explicit calculations.
\end{warning} 

In the classical case, one usually establishes flatness of $R_{*}R$ 
by showing that $R$ satisfies the stronger condition of being Adams-type, see \cite[Condition 13.3]{adams1995stable} and \cite[\S3.3]{pstrkagowski2018synthetic}.  There is a similarly powerful method for quiver homology theories.

\begin{definition}
\label{definition:pcat_adams_type_ring_spectrum}
We say $A \rightarrow R$ is \emph{$\pcat$-Adams-type} if for every simple $S \in \pcat$, we can write the underlying right $A$-module of its $R$-linear dual $\mathbb{D}^{R}(S) \coloneqq \map_{R}(S, R)$ as a filtered colimit
\[
\mathbb{D}^{R}(S) \cong \colim \mathbb{D}^{A}S_{\alpha}
\]
of $A$-linear duals of left $A$-modules $S_{\alpha}$ with the property that $R \otimes_{A} S_{\alpha}$ is a simple $R$-module. 
\end{definition}

\begin{lemma}
\label{proposition:formula_for_tensoring_with_r_in_adams_type_case}
Let $A \rightarrow R$ be a $\pcat$-Adams-type morphism of ring spectra and let $M$ be an $R$-module. 
Then, for any simple $S$ and any 
filtered colimit presentation $\mathbb{D}^{R}S \cong \colim \mathbb{D}^A S_{\alpha}$ as in \cref{definition:pcat_adams_type_ring_spectrum}, we have a canonical isomorphism 
\[
\Psf(R \otimes_{A} M)(S) \cong \colim \Psf(M)(R \otimes_{A} S_{\alpha})
\]
\end{lemma}

\begin{proof}
Recall that when $B$ is an $\E_{1}$-algebra, $\map_{B}$ denotes morphisms of left $B$-modules and $\map_{B^{\op}}$ denotes morphisms of right $B$-modules. We have a string of equivalences 
\begin{align*}
\Psf (R \otimes_{A} M)(S) 
&\cong \pi_0\Map_R(S, R \otimes_{A} M) 
\cong \pi_0\Map_{A^{\op}}(A, \mathbb{D}^{R}S \otimes_{R} R \otimes_{A} M) \\
&\cong \pi_0\Map_{A^{\op}}(A, \mathbb{D}^{R}S \otimes_{A} M) 
\cong \pi_0\Map_{A^{\op}}(A, (\colim \mathbb{D}^{A}S_{\alpha}) \otimes_{A} M) \\
&\cong \colim \pi_0\Map_{A^{\op}}(A, \mathbb{D}^{A}S_{\alpha} \otimes_{A} M) 
\cong \colim \pi_0\Map_{A}(S_{\alpha}, M) \\
&\cong \colim \pi_0\Map_R(R \otimes_{A} S_{\alpha}, M) 
\cong \colim \Psf(M)(R \otimes_{A} S_{\alpha}).    
\end{align*}
\end{proof}

\begin{corollary}
\label{corollary:r_otimes_preserves_psf_epis_for_at_rings}
If $A \rightarrow R$ is $\pcat$-Adams-type, then $A \rightarrow R$ is $\Psf$-flat. 
\end{corollary}

\begin{proof}
We must verify that if $M \rightarrow N$ is a $\Psf$-epimorphism, then so is $R \otimes_{A} M \rightarrow R \otimes_{A} N$. This is immediate from \cref{proposition:formula_for_tensoring_with_r_in_adams_type_case}, as epimorphism of quivers are levelwise, and epimorphisms of abelian groups are stable under filtered colimits. 
\end{proof}

\begin{remark}
\label{remark:formula_for_values_of_quiver_tensor_q_in_adams_type_case}
Note that if $A \rightarrow R$ is $\pcat$-Adams-type, then for a general $\pcat$-representation $\X$ we have 
\[
(\quiversteenrod \X)(S) \cong \colim \X(R \otimes_{A} S_{\alpha}) \cong \colim \Hom_{\Rep(\pcat)}(\Psf(R \otimes_{A} S_{\alpha}), \X)
\]
where we pick a filtered colimit presentation 
$\mathbb{D}^{R}S \cong \colim \mathbb{D}^{A}S_{\alpha}$ as in \cref{definition:pcat_adams_type_ring_spectrum}. 
Indeed, both sides define cocontinuous, exact functors of $\X$ and agree on projectives by \cref{proposition:formula_for_tensoring_with_r_in_adams_type_case}. 
This gives an explicit formula for the comonad $\quiversteenrod $.
\end{remark}

\begin{lemma}
\label{remark:formula_for_values_of_quiver_itself_in_adams_type_case}
Let $A \rightarrow R$ be $\pcat$-Adams-type, let $\X$ be a $\quiversteenrod$-comodule and let $\mathbb{D}^{R}S \cong \colim \mathbb{D}^{A}S_{\alpha}$ be a filtered colimit presentation as in \cref{definition:pcat_adams_type_ring_spectrum}.
Then 
\[
\X(S) \cong \colim \Hom_{\quiversteenrod}(\Psf(R \otimes_{A} S_{\alpha}), \X).
\]
\end{lemma}

\begin{proof}
We write $X$ as an equalizer
    $\X \rightarrow \quiversteenrod \X \rightrightarrows \quiversteenrod \quiversteenrod \X$
    in $\quiversteenrod$-comodules.
    Then, using \cref{remark:formula_for_values_of_quiver_tensor_q_in_adams_type_case} we have isomorphisms
    \begin{align*}
    \X(S) 
    &\cong \mathrm{eq} \left( \Hom_{\Rep(\pcat)}(\Psf(S), \quiversteenrod \X) \rightrightarrows \Hom_{\Rep(\pcat)}(\Psf(S), \quiversteenrod \quiversteenrod X) \right) \\    
    &\cong \mathrm{eq} \left(\colim \Hom_{\Rep(\pcat)}(\Psf(R \otimes_A S_{\alpha}), \X) \rightrightarrows \colim \Hom_{\Rep(\pcat)}(\Psf(R \otimes_A S_{\alpha}), \quiversteenrod X) \right) \\
    &\cong \colim \Hom_{\quiversteenrod}(\Psf(R \otimes_A S_{\alpha}), \X)    
    \end{align*}
\end{proof}

One useful consequence of the Adams condition is the existence of enough finite spectra whose homology is projective. This is often used in constructions of synthetic spectra-like categories. 

\begin{corollary}
\label{corollary:in_the_adams_type_case_enough_finite_projectives}
Let $A \rightarrow R$ be $\pcat$-Adams-type. Then $\quiversteenrod$-comodules of the form $\Hsf(P)$, where $P$ is a compact $A$-module such that $\Hsf(P)$ is projective as a $\pcat$-rep, generate the category $\Comod_{\quiversteenrod}$ under colimits.
\end{corollary}

\begin{proof}
It's enough to check that if $\X$ is a comodule such that 
\[
\Hom_{\quiversteenrod}(\Hsf(P), \X) = 0 
\]
for each such $P$, then $\X = 0$. This is immediate from \cref{remark:formula_for_values_of_quiver_itself_in_adams_type_case}.
\end{proof}

\subsection{Ind-perfect $\quiversteenrod$-comodules}
\def\IndPerf{\mathrm{IndPerf}}

As a Grothendieck abelian category, the category of comodules  $\Comod_{\quiversteenrod}(\Rep(\pcat))$ has an associated derived category. However, for some applications it is more convenient to work with a renormalized variant of the derived category first introduced by Hovey \cite{hovey2003homotopy}, which we recall here. 

\begin{definition}
\label{definition:relatively_perfect_object_of_Derived_infty_category} 
Let us say that an object $X \in \dcat(\Comod_{\quiversteenrod})$ is \emph{relatively perfect} if it belongs to the thick subcategory 
\[
\dcat(\Comod_{\quiversteenrod})^{\mathrm{rel. perf}} \subseteq \dcat(\Comod_{\quiversteenrod})
\]
which contains all objects of
\[
\dcat(\Comod_{\quiversteenrod})^{\heartsuit} \cong \Comod_{\quiversteenrod}
\]
whose underlying $\pcat$-representation is compact projective as an object of $\Rep(\pcat)$.
\end{definition}

\begin{definition}
\label{definition:hoveys_stable_infinity_category} 
Let $A \rightarrow R$ be $\pcat$-Adams-type. 
We define the category of  \emph{ind-perfect $\quiversteenrod$-comodules}   
\[
\IndPerf_{\quiversteenrod} \colonequals \Ind(\dcat(\Comod_{\quiversteenrod})^{\mathrm{rel. perf}}),
\]
to be given by the $\Ind$-completion of the category of relative perfect $\quiversteenrod$-comodules.
\end{definition}

\begin{remark}
As a consequence of \cref{corollary:in_the_adams_type_case_enough_finite_projectives}, the subcategory of relative perfects generates the derived category under colimits. It follows that the inclusion of relative perfects induces a reflective localization 
\[
\IndPerf_{\quiversteenrod} \rightleftarrows \dcat(\Comod_{\quiversteenrod}).
\]
Thus, one can think of $\IndPerf_{\quiversteenrod}$ as a particular decompletion of the derived category, obtained by forcing all relative perfects to be compact.
\end{remark}

\begin{remark}
\label{remark:hoveys_stable_infinity_category_as_sheaves} 
Let $\Comod_{\quiversteenrod}^{\mathrm{fp}}$ denote the subcategory of those $\quiversteenrod$-comodules whose underlying $\pcat$-representation is compact projective. Then as a consequence of \cite[Theorem 2.58]{pstrkagowski2018synthetic}, the inclusion of this subcategory induces an equivalence 
\[
\sheaves_{\Sigma}(\Comod_{\quiversteenrod}^{\mathrm{fp}}; \abeliangroups) \cong \Comod_{\quiversteenrod},
\]
where on the left hand side we have sheaves of abelian groups with respect to the epimorphism topology. By \cite[Theorem 3.7]{pstrkagowski2018synthetic}, this extends to an equivalence
\[
\sheaves_{\Sigma}(\Comod_{\quiversteenrod}^{\mathrm{fp}}; \spectra) \cong \IndPerf_{\quiversteenrod},
\]
between sheaves of spectra and the ind-perfect category of \cref{definition:hoveys_stable_infinity_category}. 
\end{remark}

\subsection{Monoidal structure on quivers}
\label{subsection:monoidal_structure_on_quivers}

Let $R$ be an $\E_{2}$-algebra, so that the category $\Mod_{R}$ of left $R$-modules has a canonical monoidal structure given by the tensor product. In this section we explore the induced additional structure on the category $\Rep(\pcat)$ of representations. 

\begin{definition}
\label{definition:multiplicative_class_of_simples}
We say that a class of simples $\pcat$ is \emph{multiplicative} if 
it is closed under tensor products and it contains the monoidal unit, ie. $R \in \pcat$.
\end{definition}

\begin{example}
Both the class of all perfect modules and the class of finite projective modules of 
\cref{example:finite_free_modules_form_a_class_of_simples} are multiplicative for any $\E_{2}$-algebra $R$. 
\end{example}

\begin{proposition}
\label{proposition:quivers_have_a_monoidal_structure}
Let $R$ be an $\E_{n}$-algebra and $\pcat$ a multiplicative class of simples. Then, $\Rep(\pcat)$ has a canonical $\E_{n-1}$-monoidal structure commuting with colimits seperately in each variable, uniquely specified by the property that
\[
\Psf\colon \Mod_{R} \rightarrow \Rep(\pcat)
\]
is lax $\E_{n-1}$-monoidal and 
$\E_{n-1}$-monoidal when restricted to simples. 
\end{proposition}

\begin{proof}
This follows from the universal property of Day convolution.
\end{proof}

\begin{remark}
\label{remark:monoidality_of_quivers_depending_on_en_structure}
Since $\Rep(\pcat)$ is a $1$-category, the above can be phrased using more classical language as saying that the category of representations is canonically:
\begin{enumerate}
    \item[(a)] monoidal if $R$ is $\E_{2}$,
    \item[(b)] braided monoidal if $R$ is $\E_{3}$
    \item[(c)] symmetric monoidal when $R$ is $\E_{4}$. 
\end{enumerate}
\end{remark}

\section{Synthetic spectra based on a quiver homology theory} 
\label{section:synthetic_spectra_based_on_quivers}
A recent development in stable homotopy theory is the ability to encode various Adams spectral sequences through categorical deformations, first constructed in the Adams-type case in \cite{pstrkagowski2018synthetic}. In this section, we generalize this construction to the setting of quiver homology theories considered in the current work. 

While categorical deformations have proven very useful in various analyzing spectral sequences in general, our primary motivation is to set up ground for the description of the Adams spectral sequence based on integral homology given in \S\ref{sec:integral}. One of our main results says that the integral Adams spectral sequence can be essentially decomposed in terms of classical Adams spectral sequences. The appropriate context for making such a statement precise is a square of deformation categories, see \cref{thm:pullback-Z}. The purpose of this section is to construct the needed deformations. 

This section is somewhat technical and a reader already familiar with synthetic spectra can safely skip it on the first reading. In more detail, in \S\ref{subsection:construction_of_synthetic_spectra}, we generalize the construction of synthetic spectra in the following ways:
\begin{enumerate}
    \item we allow the Adams-type condition in the setting of quivers, as in \cref{definition:pcat_adams_type_ring_spectrum}, 
    \item we work relative to an arbitrary map $A \rightarrow R$ of $\mathbb{E}_{1}$-rings, constructing an category of \emph{$R$-based synthetic $A$-modules}. The case of synthetic \emph{spectra} is recovered by specializing to the unit map $\Ss \rightarrow R$.
\end{enumerate}
Both generalizations are mild, as the arguments of \cite{pstrkagowski2018synthetic} go through essentially verbatim. In \S\ref{subsection:comparison_of_deformations}, we compare our construction with the various deformations of \cite{patchkoria2021adams}. 

\subsection{Construction of synthetic modules} 
\label{subsection:construction_of_synthetic_spectra}

To keep things short, we will assume that the reader is familiar with the main properties and the broad construction of synthetic spectra in the classical case. The main reference is \cite{pstrkagowski2018synthetic}. 

\begin{notation}
Let $f \colon A \rightarrow R$ be a map of $\mathbb{E}_{1}$-rings and let $\pcat$ be a class of simples such that $f$ is $\pcat$-Adams-type. It follows that we have the push-forward adapted homology theory
\[
f_{*} \Psf: \Mod_{A} \rightarrow \Comod_{\quiversteenrod}(\Rep(\pcat))
\]
whose underlying quiver representation is given by
\[
(f_{*}\Psf)(X)(S) \colonequals \pi_{0} \Map_{\Mod_{R}}(S, R \otimes_{A} X).
\]

If there is no risk of confusion, we will write $\Psf \colonequals f_{*}\Psf$. Note that with this convention we can distinguish $\Psf \colon \Mod_{A} \rightarrow \Comod_{\quiversteenrod}$ and $\Psf \colon \Mod_{R} \rightarrow \Rep(\pcat)$ by type-checking; one is defined on $A$-modules and the other on $R$-modules. 
\end{notation}

\begin{definition}
We say that an $A$-module $M$ is \emph{finite $\pcat$-projective} if it is compact and $R \otimes_{A} M$ is simple. We denote the full subcategory spanned by finite projectives by 
\[
(\Mod_{A})^{\mathrm{fp}}_{\pcat} \subseteq \Mod_{A}.
\]

\end{definition}

\begin{lemma}
\label{lemma:conditions_for_epis_of_pcat_projectives}
Let $M_{1} \rightarrow M_{2}$ be a map of finite $\pcat$-projectives. Then the following two conditions are equivalent: 
\begin{enumerate}
    \item $\Psf(M_{1}) \rightarrow \Psf(M_{2})$ is an epimorphism in $\Rep(\pcat)$. 
    \item $R \otimes_{A} M_{1} \rightarrow R \otimes_{A} M_{2}$ is a split epimorphism of $R$-modules. 
\end{enumerate}
\end{lemma}

\begin{proof}
The backward implication is clear and the forward one is \cref{lemma:p_epis_of_simple_modules_are_split}. 
\end{proof}

\begin{definition}
\label{definition:coverings_of_a_modules}
We say that a morphism $M_{1} \rightarrow M_{2}$ of finite $\pcat$-projectives is a \emph{covering} if it satisfies the equivalent conditions of \cref{lemma:conditions_for_epis_of_pcat_projectives}. 
\end{definition}

Note that since coverings of \cref{definition:coverings_of_a_modules} are stable under base-change, the category $(\Mod_{A})^{\mathrm{fp}}_{\pcat}$ becomes an an additive site in the sense of \cite[Definition 2.3]{pstrkagowski2018synthetic} with respect to this class of coverings. 

\begin{definition}
\label{definition:hsf_based_synthetic_spectra}
A $\pcat$-based \emph{synthetic $A$-module} is an additive sheaf 
\[
X \colon (\Mod_{A})_{\pcat}^{\mathrm{fp}, op} \rightarrow \spectra
\]
of spectra on finite $\pcat$-projectives. We denote the category of synthetic $A$-modules by 
\[
\Syn_{\pcat}(\Mod_{A})
\]
\end{definition}

Note that by construction, the category of synthetic $A$-modules is stable and presentable.

\begin{example}
If $M$ is an $A$-module, we denote by $\nu(M)$ the synthetic $A$-module given by the sheafification of the presheaf 
\[
P \mapsto \tau_{\geq 0} \map_{\Mod_{A}}(P, M),
\]
where the right hand side is the connective part of the mapping spectrum. We call $\nu(M)$ the \emph{synthetic analogue}.
\end{example}

\begin{remark}
If $P$ is finite projective, the Yoneda lemma shows that for any $X \in \Syn_{\pcat}(\Mod_{A})$ we have a canonical equivalence of spectra
\[
\map(\nu(P), X) \cong X(P).
\]
It follows that synthetic $A$-modules of the form $\nu(P)$ generate $\Syn_{\pcat}(\Mod_{A})$ as a stable, presentable category. These generators are compact as a consequence of \cite[Corollary 2.9]{pstrkagowski2018synthetic}.
\end{remark}

\begin{construction}[$t$-structure]
As a subcategory of sheaves of spectra closed under cofibers, $\Syn_{\pcat}(\Mod_{A})$ inherits a canonical $t$-structure in which a synthetic $A$-module $X$ is coconnective if and only if $X(P) \in \spectra_{\leq 0}$ for all finite $\pcat$-projectives $P$. 
\end{construction}

\begin{construction}[Local grading and $\tau$]
The category of synthetic $A$-module has a local grading (that is, an autoequivalence) 
\[
[1] \colon \Syn_{\pcat}(\Mod_{A}) \rightarrow \Syn_{\pcat}(\Mod_{A})
\]
defined by 
\[
(X[1])(P) \colonequals X(\Sigma^{-1}P).
\]
Note that this is distinct from the categorical suspension, but there is a canonical comparison map 
\[
\tau \colon \Sigma X \rightarrow X[1].
\]
\end{construction}

As in case of synthetic spectra of \cite{pstrkagowski2018synthetic}, the map $\tau$ presents $\Syn_{\pcat}(\Mod_{A})$ as an appropriate categorical deformation, interpolating between $A$-modules and the category of ind-perfect $\quiversteenrod$-comodules introduced in \cref{definition:hoveys_stable_infinity_category}. 

\begin{theorem}[Special fiber]
\label{thm:syn props}
The category of $\pcat$-based synthetic $A$-modules has the following properties: 
\begin{enumerate}
    \item There exists a canonical equivalence $\Syn_{\pcat}(\Mod_{A})^{\heartsuit} \cong \Comod_{\quiversteenrod}$. 
    \item For any $M \in \Mod_{A}$, there is a natural isomorphism $\pi^{\heartsuit}_{0}(\nu(M)) \cong \Psf(M)$.
    \item There exists a canonical structure of a monad on the endofunctor of $\Syn_{\pcat}$
    \[
    C\tau(X):=\mathrm{cofib}(\tau: X \rightarrow X)
    \]
    and a fully faithful embedding 
    \[
    \Mod_{C\tau}(\Syn_{\pcat}) \hookrightarrow \IndPerf_{\quiversteenrod}
    \]
    into the category of ind-perfect $\quiversteenrod$-comodules in $\Rep(\pcat)$. This embedding becomes an equivalence of categories after hypercompletion. 
\end{enumerate}
\end{theorem}

\begin{proof}
The arguments are the same as in the case of a ring spectrum which is Adams-type in the classical sense, as in \cite[\S 4.5]{pstrkagowski2018synthetic}. For the convenience of the reader, we sketch the arguments here. Taking $\Psf$-homology induces a morphism of sites 
\[
(\Mod_{A})^{\mathrm{fp}}_{\pcat} \rightarrow \Comod_{\quiversteenrod}^{\mathrm{fp}},
\]
where the target is the category of comodules whose underlying $\pcat$-representation is compact projective equipped with the epimorphism topology. This yields an adjunction between categories of sheaves which by \cref{remark:hoveys_stable_infinity_category_as_sheaves} we can identify with an adjunction
\[
\Psf^{*} \colon \Syn_{\pcat}(\Mod_{A}) \rightleftarrows \IndPerf_{\quiversteenrod} \colon \Psf_{*} .
\]
This induces an equivalence on hearts by the argument of \cite[Theorem 3.27]{pstrkagowski2018synthetic}, which shows $(1)$ and, after unwrapping the definitions, $(2)$. For $(3)$, we observe that for any $P \in (\Mod_{A})^{\mathrm{fp}}_{\pcat}$ we have a cofiber sequence
\[
\Sigma \nu(P)[-1] \rightarrow \nu(P) \rightarrow \Psf(P)
\]
where the first map is $\tau$ and we consider $\Psf(P)$ as an object of the heart and thus 
\[
\C\tau(\nu(P)) \cong \Psf_{*} \Psf^{*} \nu(P)
\]
since $\Psf^{*}(\nu(P)) \cong \Psf(P) \in \Comod_{\quiversteenrod}$ by construction. This allows one to identify the monad $\Psf_{*} \Psf^{*}$ with $C\tau$, giving a cocontinuous functor $\Mod_{C\tau}(\Syn_{\pcat}) \rightarrow \IndPerf_{\quiversteenrod}$. This is fully faithful since both sides are generated under colimits by elements of the heart, and the functor is an equivalence between the hearts by $(1)$. 
\end{proof}
The description of the generic fiber is somewhat more involved, since we need to take into account that finite projectives need not generate $\Mod_{A}$ under colimits. 

\begin{notation}
We write $\Mod_{A}^{\pcat} \subseteq \Mod_{A}$ for the subcategory of $A$-modules generated under colimits by $(\Mod_{A})^{\mathrm{fp}}_{\pcat}$. We call its objects the \emph{$\pcat$-generated} $A$-modules. 
\end{notation}

\begin{theorem}[Generic fiber]
\label{theorem:generic_fiber_of_synthetic_spectra}
The category of $\pcat$-based synthetic $A$-modules has the following properties: 
\begin{enumerate}
    \item the composite $\tau^{-1}(-) \circ \nu$ induces an equivalence $\Mod^{\pcat}_{A}~\cong~\Syn_{\pcat}(\Mod_{A})^{\tau = 1}$ between $\pcat$-generated $A$-modules and the subcategory of those synthetic $A$-modules on which $\tau$ acts invertibly, 
    \item the functor $\nu \colon \Mod^{\pcat}_{A} \rightarrow \Syn_{\pcat}(\Mod_{A})$ is fully faithful when restricted to $\pcat$-generated $A$-modules. 
\end{enumerate}
\end{theorem}

\begin{proof}
Let $\Syn_{\pcat}(\Mod_{A})_{\geq 0}^{per}$ denote the subcategory of \emph{$\tau$-periodic} connective synthetic $A$-modules; that is, those $X$ for which the map $\tau \colon \Sigma X[-1] \rightarrow X$ is a $1$-connective cover. By \cite[Proposition 2.22]{pstrkagowski2021abstract}, taking connective covers and $\tau$-localization induces an equivalence 
\[
\Syn_{\pcat}(\Mod_{A})_{\geq 0}^{per} \cong \Syn_{\pcat}(\Mod_{A})^{\tau = 1}.
\]
In particular, $\Syn_{\pcat}(\Mod_{A})_{\geq 0}^{per}$ is stable. Since $\nu$ factors through an exact functor into periodic synthetic $A$-modules, we can prove both $(1)$ and $(2)$ at once by showing that $\nu$ induces an equivalence 
\[
\Mod^{\pcat}_{A} \cong \Syn_{\pcat}(\Mod_{A})_{\geq 0}^{per}.
\]
By construction, the map is fully faithful on the subcategory of finite projectives. As both sides are compactly generated by finite projectives, the result follows. 
\end{proof}

\begin{example}
If $R \in \pcat$, then $\Mod_{A}^{\pcat} = \Mod_{A}$. In practice, this is the most common situation, in which case in the context of \cref{theorem:generic_fiber_of_synthetic_spectra} the assumption of $\pcat$-generation is superfluous. 
\end{example}

\begin{example}
Let $E$ be a Lubin-Tate spectrum, and let $\pcat$ be the class of \emph{molecular} $E$-modules of Hopkins and Lurie \cite[Definition 3.6.11]{lurie_hopkins_brauer_group}. In this case, the $K(n)$-completion functor identifies $\Mod_{E}^{\pcat}$ with the subcategory of $K(n)$-local $E$-modules. Moreover, we can identify connective $\pcat$-based synthetic $E$-modules of \cref{definition:hsf_based_synthetic_spectra} with synthetic $E$-modules of \cite[\S 4]{lurie_hopkins_brauer_group}. This example is the origin of the word ``synthetic'' being used to describe objects of deformation categories. 
\end{example}

One main advantage of working with synthetic $A$-modules as in \cref{definition:hsf_based_synthetic_spectra} is the ability to introduce symmetric monoidal structures related to the derived tensor product, as we now explain. For simplicity, we will work with $\mathbb{E}_{\infty}$-rings, but the results have analogs for less structured multiplications. 

\begin{construction}[Tensor product of synthetic $A$-modules]
\label{construction:tensor_product_of_synthetic_spectra}
Let $A \rightarrow R$ be a map of $\mathbb{E}_{\infty}$-rings and assume that $\pcat$ is a multiplicative class of simples\footnote{More generally, one can work with $\mathbb{E}_{n}$-rings, in which case \cref{construction:tensor_product_of_synthetic_spectra} will yield a $\mathbb{E}_{n-1}$-monoidal structure on $\Syn$}. We will construct a canonical presentable symmetric monoidal structure on $\Syn_{\pcat}(\Mod_{A})$.  

Observe that since $\pcat$ is multiplicative and for any $P_{1}, P_{2} \in (\Mod_{A})_{\pcat}^{\mathrm{fp}}$ we have 
\[
R \otimes_{A} (P_{1} \otimes_{A} P_{2}) \cong (R \otimes_{A} P_{1}) \otimes_{R} (R \otimes_{A} P_{2}),
\]
$(\Mod_{A})_{\pcat}^{\mathrm{fp}} \subseteq \Mod_{A}$ is a symmetric monoidal subcategory. We claim that the symmetric monoidal structure on $(\Mod_{A})_{\pcat}^{\mathrm{fp}}$ induces through left Kan extension a symmetric monoidal structure on $\Syn$ compatible with colimits. Since $\Syn$ is defined as the category of additive sheaves, to check this, we have to verify that for any $P \in (\Mod_{A})_{\pcat}^{\mathrm{fp}}$, the functor 
\[
P \otimes_{A} - \colon (\Mod_{A})_{\pcat}^{\mathrm{fp}} \rightarrow (\Mod_{A})_{\pcat}^{\mathrm{fp}}
\]
has the following two properties: 
\begin{enumerate}
    \item it is additive,
    \item it preserves $\Psf$-epimorphism. 
\end{enumerate}
The first part is clear and the second one follows from \cref{lemma:p_epis_of_simple_modules_are_split}, since if $P_{1} \rightarrow P_{2}$ is a $\Psf$-epimorphism in $(\Mod_{A})_{\pcat}^{\mathrm{fp}}$, then $R \otimes_{A} P_{1} \rightarrow R \otimes_{A} P_{2}$ admits a section and so will stay a $\Psf$-epimorphism after tensoring with an arbitrary $R$-module. 
\end{construction}

\begin{proposition}
\label{proposition:properties_of_the_symmetric_monoidal_structure_on_syn}
Let $A \rightarrow R$ be a map of $\mathbb{E}_{\infty}$-rings and let $\pcat$ be multiplicative a class of simples on $R$. Then the symmetric monoidal structure on $\Syn_{\pcat}(\Mod_{A})$ of \cref{construction:tensor_product_of_synthetic_spectra} has the following properties: 
\begin{enumerate}
\item the functor $\nu \colon \Mod_{A} \rightarrow \Syn$ is lax symmetric monoidal, and strongly symmetric monoidal when restricted to $A$-modules which are are filtered colimits of objects in $(\Mod_{A})_{\pcat}^{\mathrm{fp}}$, 
\item the equivalence $\Mod_{A}^{\pcat} \cong \Syn^{\tau=1}$ induced by $\tau$-localization is strongly symmetric monoidal, 
\item the inclusion $\Mod_{C\tau}(\Syn) \hookrightarrow \IndPerf_{\quiversteenrod}$ is strongly symmetric monoidal. 
\end{enumerate}
\end{proposition}

\begin{proof}
This is the same as \cite[Lemma 4.4, Lemma 4.24, Theorem 4.37, Theorem 4.46]{pstrkagowski2018synthetic}.
\end{proof}

\begin{remark}
The symmetric monoidal structure on the $\IndPerf$-category appearing in \cref{proposition:properties_of_the_symmetric_monoidal_structure_on_syn} is the one induced by the tensor product of \S \ref{subsection:monoidal_structure_on_quivers}. Note that this is the derived tensor product; that is, the unique one which coincides with the ordinary tensor product on comodules whose underlying quiver representation is projective and which preserves colimits in each variable. 
\end{remark}

\subsection{Comparison of deformations}
\label{subsection:comparison_of_deformations}

The construction of synthetic spectra given in \cite{pstrkagowski2018synthetic}, which requires the homology theory to be Adams-type, was generalized in the work of Patchkoria and the second author to cover the case of an arbitrary adapted homology theory \cite{patchkoria2021adams}. Restricting to the case of homology theories on spectra for simplicity, this means that to a ring spectrum $R$ together with a class of simples $\pcat$ and an associated quiver-based homology theory $\Psf \colon \spectra \rightarrow \Comod_{\quiversteenrod}$ we can associate a variety of closely related deformations, in particular: 
\begin{enumerate}
    \item $\dcat_{R}^{\omega}(\Sp)$ of \cite[\S 5.2]{patchkoria2021adams}, the perfect deformation of $\Sp$ along the epimorphism class of $R \otimes -$-split epis, which categorifies the descent spectral sequence based on $R$, 
    \item $\dcat_{\Psf}^{\omega}(\Sp)$ \cite[\S 5.2]{patchkoria2021adams}, the perfect deformation of $\Sp$ along $\Psf$-epis, which categorifies the spectral sequence associated to $\Psf$-injective resolutions,
    \item $\widecheck{\dcat}_{\Psf}(\Sp)$ of \cite[\S 6.4]{patchkoria2021adams}, the (unseparated) Grothendieck deformation of $\Sp$ along $\Psf$-epis, which categorifies the spectral sequence associated to $\Psf$-injective resolutions and is a Grothendieck prestable category in the sense of \cite[Appendix C]{lurie_spectral_algebraic_geometry}, 
    \item $\Syn_{\pcat}$, the category of $\pcat$-synthetic spectra of \cref{definition:hsf_based_synthetic_spectra} attached to $\Ss \rightarrow R$, which is a renormalization of $\widecheck{\dcat}_{\Psf}(\Sp)$.
\end{enumerate}

\begin{remark}
Note that the second and third deformations depend only on the homology theory $\Psf$ (equivalently, on the associated epimorphism class on spectra), but the category of synthetic spectra depends on the class of simples $\pcat$. In \cref{cor:def-equals-syn}, we give a condition for $(3)$ and $(4)$ to coincide, in which case $\Syn_{\pcat}$ is also an invariant of $\Psf$. 
\end{remark}

Our goal in this section is to describe the relationships between these different deformation categories. It will be convenient to use the language of \S\ref{subsection:local_and_colocal_objects}; in particular, we recall from \cref{definition:local_and_colocal_objects_for_a_functor} that given a functor $F \colon \ccat \to \dcat$ we say an object $X$ is \emph{$F$-colocal} if  
    \[ \Map_{\ccat}(X, Y) \to \Map_{\dcat}(F(X), F(Y)) \]
is an equivalence for every $Y \in \ccat$.

\begin{construction}
\label{cnstr:comparison-functors}
We construct comparison functors between the above deformations:
\begin{enumerate}
\item If a map of $R$-modules is split epi, then it is in particular $\pcat$-epi, thus by \cite[Thm 5.35]{patchkoria2021adams} we have a comparison functor 
  \[ \dcat_{R}^{\omega}(\Sp) \to \dcat_{\Psf}^{\omega}(\Sp) \]
  which is compatible with $\nu$.
\item The universal property of $\dcat_{\Psf}^{\omega}(\Sp)$ also allows us to construct a comparison functor
  \[ \dcat_{\Psf}^{\omega}(\Sp) \to \widecheck{\dcat}_{\Psf}(\Sp) \]
  which is compatible with $\nu$.
\item From \Cref{thm:syn props} and the definition of $\widecheck{\dcat}_{\Psf}(\Sp)$ we know that
  \[ \widecheck{\dcat}_{\Psf}(\Sp)^{\heartsuit} \cong \mathrm{Comod}_{\mathcal{Q}} \cong \Syn_{\pcat}^\heartsuit. \]
  In particular we can apply the universal property of \cite[Thm 6.40]{patchkoria2021adams}
  to the identity functor on $\widecheck{\dcat}_{\Psf}(\Sp)^{\heartsuit}$ together with the enhancement given by $\nu : \Sp \to (\Syn_{\pcat})_{\geq 0}$ to produce a $t$-exact, colimit preserving functor 
  \[ \widecheck{\dcat}_{\Psf}(\Sp) \to \Syn_{\pcat} \]
  compatible with $\nu$.
\end{enumerate}
\end{construction}
In summary, we have comparison functors which fit into the following commutative diagram
\[ 
\begin{tikzcd}
& \Sp \ar[dl, "\nu_R"'] \ar[d, "\nu_{\pcat}"'] \ar[dr, "\nu_{\pcat}"'] \ar[drr, "\nu_{\pcat}"] & &  \\ 
\dcat_{R}^{\omega}(\Sp) \ar[r] &
\dcat_{\Psf}^{\omega}(\Sp) \ar[r] & 
\widecheck{\dcat}_{\Psf}(\Sp) \ar[r] & 
\Syn_{\pcat}.
\end{tikzcd} 
\]
We will not name these comparison functors. Since they only go left to right, this should not lead to any confusion. We now investigate how close these functors are to being equivalences; this is a categorified version of study of how the different spectral sequences interact. 

\begin{lemma}
\label{lemma:P_colocal_from_amitsur_to_quiver}
    If $X$ is $\pcat$-finite projective, then $\nu_R(X)$ is colocal for
    the comparison functor
    \[ \dcat_{R}^{\omega}(\Sp) \to \dcat_{\Psf}^{\omega}(\Sp). \]
\end{lemma}

\begin{proof}
Let $i \colon \Ss \to R$ denote the unit map, $y \colon \Mod_{R} \rightarrow A(\Mod_{R})$ the universal homology theory (cf. \cite[Section 2.3]{patchkoria2021adams}) and $\Psf \colon \Mod_{R} \rightarrow \Rep(\pcat)$ the homology theory valued in quiver representations. 

We will use results on functoriality of deformations from the appendix \S\ref{subsection:functoriality_of_deformations} as well as the corresponding notation. In particular, if $\Hsf \colon \ccat \rightarrow \acat$ is an adapted homology theory, which we can identify with a class of epimorphisms, we will write $\acat(\ccat; \Hsf) \colonequals \acat$ for its target. Note that in this notation we have $\acat(\Mod_{R}, y) \cong A(\Mod_{R})$, $\acat(\Mod_{R}, \Psf) \cong \Rep(\pcat)$ and $\acat(\spectra, i_{*} \Psf) \cong \Comod_{\quiversteenrod}$. 

We have a right adjointable square of adapted homology theories of the form 
  \[ \begin{tikzcd}
    (\Sp, i_*y) \ar[r, "i^*"] \ar[d, "i_*q"] & (\Mod(R), y) \ar[d, "q"] \\
    (\Sp, i_*\pcat) \ar[r, "i^*"] & (\Mod(R), \pcat). 
  \end{tikzcd} \]
The comparison functor we wish to study can now be described as the result of applying $\dcat^{\omega}(-)$ to the left vertical arrow $i_*q$.

We start by using \Cref{prop:colocal-to-abelian} to reduce to showing that $(i_*y)(X)$ is $i_*q$-colocal. Next, after applying $\dcat^b(\acat(-))$ to the right adjointable square above we may use \Cref{lem:square-colocal} to reduce to showing that $y(R \otimes X)$ is $q$-colocal. Here we have used the comonadicity of the $(i^*, i_*)$ adjunction between $\acat(\Sp, i_*y)$ and $\acat(\Mod(R), y)$ to ensure that $\dcat^b(\acat(\Sp, i_*y))$ is cogenerated by the image of $i_*$. 

Finally, \Cref{exm:H projective} lets us check that $y(R \otimes X)$ is $q$-colocal by observing that since $X$ was assumed $\pcat$-finite projective $\pcat(R \otimes X)$ is a projective object of $\acat(\Mod(R), \pcat)$ and that there are equivalences
\[ 
\Hom_{\acat(\Mod(R), \pcat)}(\pcat(R \otimes X), \pcat(Y)) \cong \pcat(Y)(R \otimes X) \cong \pi_{0} \Map_{\Mod_{R}}(R \otimes X, Y).
\] 
\end{proof}

\begin{remark}
The comparison functor
\[
\dcat_{R}^{\omega}(\Sp) \to \dcat_{\Psf}^{\omega}(\Sp) 
\]
produced in \Cref{cnstr:comparison-functors} is exact, therefore the colocality statement of \cref{lemma:P_colocal_from_amitsur_to_quiver} implies that there are equivalences
\[ 
\Map_{\dcat_R^{\omega}(\Sp)}(Y,Z) \to \Map_{\dcat_{\Psf}^{\omega}(\Sp)}(Y,Z) 
\]
whenever $Y$ is in the full subcategory of $\dcat_R^{\omega}(\Sp)$ generated under finite colimits and desuspensions by objects of the form $\nu_R X$ with $X \in \Sp_{\pcat}^{\mathrm{fp}}$.
\end{remark}

\begin{remark}
Reformulating \cref{lemma:P_colocal_from_amitsur_to_quiver} in terms of spectral sequences we learn that the natural comparison map 
\[ 
{}^RE_r^{s,t}(X,Y) \to {}^{\pcat}E_r^{s,t}(X,Y)
\]
from the $R$-based Adams spectral sequence computing calculating $\pi_{*}\Map(X, Y)$  to the $\pcat$-based Adams spectral sequence is an equivalence whenever $X$ is $\pcat$-finite projective.
\end{remark}

\begin{lemma}[{\cite[Prop. 6.51]{patchkoria2021adams}}] \label{lem:easy-comparison}
The comparison functor 
\[ \dcat_{\Psf}^{\omega}(\Sp) \to \widecheck{\dcat}_{\Psf}(\Sp)\]
is fully faithful.
\end{lemma}

%\begin{proof}
%This is {}. 
%\end{proof}

\begin{lemma} \label{lem:fp-colocal}
The comparison functor 
\[ 
\widecheck{\dcat}_{\Psf}(\Sp) \to \Syn_{\pcat} 
\]
between the unseparated deformation and synthetic spectra is an equivalence after hypercompletion. Moreover, $\nu_{\pcat}(X)$ is colocal for any $X$ which $\pcat$-finite projective. 
\end{lemma}

\begin{proof}
The first part follows from the argument given in \cite[Proposition 6.57]{patchkoria2021adams}, so we move tot he second part. Since $X$ is finite, $\nu_\pcat(X)$ is a compact object in $\widecheck{\dcat}_{\Psf}(\Sp)$. Since it is $\pcat$-finite projective, its image in $\nu_\pcat(X)$ in $\Syn_\pcat$ is also compact. This allows us to reduce the proof of the lemma to the case of mapping to $\nu_\pcat(Y)$.

Next we observe that in order to show that the comparison of mapping spectra
\[ 
  \map_{\widecheck{\dcat}_{\Psf}(\Sp)}(\nu_\pcat(X), \nu_\pcat(Y)) \to \map_{\Syn_\pcat}(\nu_\pcat(X), \nu_\pcat(Y)) 
\]
is an equivalence it suffices to show that we have equivalences 
\[ 
\map_{\widecheck{\dcat}_{\Psf}(\Sp)}(\nu_\pcat(X), \nu_\pcat(Y)^{\tau=1}) \cong \map_{\Syn_\pcat}(\nu_\pcat(X), \nu_\pcat(Y)^{\tau=1}) 
\]
and
\[ 
\map_{\widecheck{\dcat}_{\Psf}(\Sp)}(\nu_\pcat(X), C\tau(\nu_\pcat(Y))) \cong \map_{\Syn_\pcat}(\nu_\pcat(X), C\tau(\nu_\pcat(Y))).
\]
 The former follows from the equivalences. $\widecheck{\dcat}_\pcat(\Sp)^{\tau=1} \cong \Sp \cong \Syn_\pcat^{\tau=1}$ For the latter we use the equivalences $C\tau(\nu_\pcat(Y)) \cong \pi_0^\heartsuit(Y) \cong \pcat(Y)$ to recognize that $C\tau(\nu_\pcat(Y))$ is hypercomplete, hence in this case the result follows from the first part. 
\end{proof}

\begin{corollary} \label{cor:def-equals-syn}
If every finite spectrum has a finite, $\Psf$-exact resolution by $\pcat$-finite projective spectra, then the comparison map
\[ 
\widecheck{\dcat}_{\Psf}(\Sp) \to \Syn_{\pcat}(\Sp) 
\]
is an equivalence.
\end{corollary}

\begin{proof}
Observe that $\widecheck{\dcat}_{\Psf}(\Sp)$ is generated as a stable, presentable category by the collection of objects $\nu_{\pcat}(X)$ with $X \in \Sp^{\omega}$. Since $\nu_{\pcat}$ sends finite ${\pcat}$-exact resolutions to resolutions, therefore the given condition implies that $\widecheck{\dcat}_{\Psf}(\Sp)$ is generated under colimits by the collection of objects $\nu_{\pcat}(X)$ with $X$ a $\pcat$-finite projective. From \Cref{lem:fp-colocal} we know that each such $\nu_{\pcat}(X)$ is colocal for the comparison map. In particular, the comparison functor is fully faithful. To finish the proof we observe that by construction $\Syn_{\pcat}$ is generated under colimits by images of $\nu_{\pcat}(X)$.
\end{proof}

\section{Integral homology} 
\label{sec:integral}
In this section, we will specialize the formalism of quiver homology theories to the case of integral homology. The two main results are 

\begin{enumerate}
    \item \cref{theorem:hbbz_and_hsf_ass_are_iso_when_hsfx_projective} which identifies the second page of the $\bbZ$-Adams spectral sequence in terms of homological algebra of representations of a very explicit quiver, which we call Bockstein couples, 
    \item \cref{thm:pullback-Z}, which proves that a certain natural square involving synthetic spectra based on $\bbZ$- and $\fieldp$-homology is a pullback square of $\infty$-categories. 
\end{enumerate}
Both of these are somewhat abstract statements, and the reader is invited to first look at \S\ref{subsection:comparison_of_adams_sseqs}, where we derive concrete calculational consequences, in particular a complete description of the $\bbZ$-Adams spectral sequence for the sphere.

\def\op{\mathrm{op}}

\begin{notation}
We work $p$-locally and we write $\bbZ \colonequals \bbZ_{(p)}$ for the $p$-local integers. 
We will not distinguish in notation between $\bbZ$ the abelian group and the Eilenberg-MacLane spectrum $\bbZ \in \spectra^{\heartsuit}$. 
%We will use the unadorned symbol $\otimes$ to denote the tensor product of spectra and explicitly write $\otimes_{\bbZ}$ if we mean the derived tensor product relative to the integers. 
\end{notation}

\subsection{Bockstein couples}
\label{subsection:couples}

It is well-known that $\pi_{*}(\bbZ \otimes \bbZ)$ contains torsion in positive degrees, and so is not flat as a graded abelian group \cite{cartan1955seminaire}. Thus, the standard methods of identifying the Adams $E_{2}$-term fail. By analyzing the Bockstein spectral sequence, one can show that all of this torsion is simple $p$-torsion \cite{kochman1982integral}, so that we have an isomorphism 
\[
\pi_{*}(\bbZ \otimes \bbZ) \cong \bbZ \oplus W
\]
where $W$ is a graded $\fieldp$-vector space. Since a $\bbZ$-module is determined up to equivalence by its homotopy groups, this isomorphism of graded abelian groups can be lifted to an equivalence
\[
\bbZ \otimes \bbZ \cong \bbZ \oplus V
\]
of $\bbZ$-modules, where $V \in \Mod_{\fieldp}(\spectra)$. 

\begin{definition}
\label{definition:simple_hz_module}
We say that a $\bbZ$-module in spectra is \emph{simple} if it is equivalent to a finite sum of (de)suspensions of $\bbZ$ and $\fieldp$. We denote the full subcategory of simple $\bbZ$-modules by 
\[
\pcat \subseteq \Mod_{\bbZ}(\spectra).
\]
\end{definition}

\begin{definition}
\label{definition:bockstein_couple}
A \emph{Bockstein couple} is a representation of $\pcat$; that is, an additive functor
\[
\X \colon \pcat^{op} \rightarrow \abeliangroups.
\]
We denote the category of Bockstein couples by $\Rep(\pcat)$. 
\end{definition}

Note that since $\bbZ$ is a commutative algebra, the category of $\bbZ$-mdoules has a canonical symmetric monoidal structure. As $\bbZ$ is the monoidal unit and
$ \mathbb{F}_{p} \otimes _{\bbZ} \mathbb{F}_{p} \cong \mathbb{F}_{p} \oplus \Sigma \mathbb{F}_{p},$
which is again simple, $\pcat$ is multiplicative class of simples in the sense of \cref{definition:multiplicative_class_of_simples}. Moreover, since $ \map_{\bbZ}(\fieldp; \bbZ) \cong \Sigma^{-1} \fieldp$,
$\pcat$ is also closed under taking linear duals, which gives the following: 

\begin{corollary}
\label{corollary:simple_modules_are_dualizable}
Taking $\bbZ$-linear duals gives a canonical symmetric monoidal equivalence
$\pcat \cong \pcat^{op}$.
\end{corollary}

\begin{convention}
As a consequence of \cref{corollary:simple_modules_are_dualizable},  Bockstein couples of \cref{definition:bockstein_couple} can be canonically identified with additive \emph{covariant} functors
\[
\X \colon \pcat \rightarrow \Ab.
\]
As it is somewhat easier to visualize, in what follows we will think of Bockstein couples as covariant rather than contravariant functors. 
\end{convention}

The category of Bockstein couples can be described explicitly, as we now explain. As $\pcat$ is generated under direct sums by (de)suspensions of $\bbZ$ and $\mathbb{F}_{p}$, the category of Bockstein couples is equivalent to the category of $\abeliangroups_{*}$-valued covariant functors out of the $\abeliangroups_{*}$-enriched category 
\[
\begin{tikzcd}
	{} & \bullet &&& \bullet & {}
	\arrow["{\pi_*\Map_{\Z}^{\Sp}(\bbZ, \mathbb{F}_{p})}", bend left, from=1-2, to=1-5]
	\arrow["{\pi_*\Map_{\Z}^{\Sp}(\mathbb{F}_{p}; \bbZ)}", bend left, from=1-5, to=1-2]
	\arrow["{\pi_*\Map_{\Z}^{\Sp}(\bbZ,\bbZ)}", loop left, looseness=35, from=1-2, to=1-2]
	\arrow["{\pi_*\Map_{\Z}^{\Sp}(\mathbb{F}_{p}, \mathbb{F}_{p})}", loop right, looseness=35, from=1-5, to=1-5]
\end{tikzcd}
\]
The homotopy groups of these mapping spectra are easy to identify using the cofiber sequence
\[\begin{tikzcd}
	{\bbZ} & {\bbZ} & {\fieldp} & {\Sigma \bbZ}
	\arrow["p", from=1-1, to=1-2]
	\arrow["\pi", from=1-2, to=1-3]
	\arrow["\delta", from=1-3, to=1-4]
\end{tikzcd}\]
where $\delta$ is the connecting homomorphism. 
\begin{enumerate}
    \item $\pi_*\Map_{\Z}^{\Sp}(\mathbb{F}_{p}; \bbZ) \cong \F_p[-1]$ with generator $\delta$.
    \item $\pi_*\Map_{\Z}^{\Sp}(\bbZ, \mathbb{F}_{p}) \cong \F_p[0]$ with generator $\pi$.
    \item $\pi_*\Map_{\Z}^{\Sp}(\mathbb{F}_{p}, \mathbb{F}_{p}) \cong \F_p[0] \oplus \F_p[-1]$ is an $\fieldp$-exterior algebra with generator $\beta \colonequals \pi \circ \delta$.
 \end{enumerate}
 This analysis yields the following result: 
 
\begin{proposition}
\label{proposition:explicit_description_of_augmented_bockstein_modules}
A Bockstein couple $\X : \pcat \rightarrow \Ab$ determines and is uniquely determined by the following data:
\begin{enumerate}
    \item a graded abelian group $A_{*} \colonequals \X(\Sigma^{-*} \bbZ)$,
    \item a graded $\mathbb{F}_{p}$-vector space $V_{*} \coloneqq \X( \Sigma^{-*} \mathbb{F}_{p})$, 
    \item an additive morphism $\pi : A_{*} \rightarrow V_{*}$ of degree zero and 
    \item an additive morphism $\delta : V_{*} \rightarrow A_{*}$ of degree minus one. 
\end{enumerate}
This data is subject to the following condition: 
\begin{enumerate}
    \item $\delta \circ \pi = 0$ as an endomorphism of $A_{*}$.
\end{enumerate}
\end{proposition}

\begin{convention}
Keeping \cref{proposition:explicit_description_of_augmented_bockstein_modules} in mind, we will alternate between thinking of a Bockstein couple as a covariant functor $\X$ or as the associated quadruple $(A, V, \pi, \delta)$, which we will often write as
\[
\X = (A \rightleftarrows V).
\]
\end{convention}

\begin{definition}
If $\X = (A \rightleftarrows V)$ is a Bockstein couple, we call the self-map of $V$ defined by 
$ \beta \coloneqq \pi \circ \delta $
the \emph{Bockstein}. 
\end{definition}

\begin{remark}
Since $\delta \circ \pi = 0$, we have $\beta^{2} = 0$, and so $V$ is a \emph{Bockstein module}; that is, a graded vector space with a differential . This motivates the terminology of \cref{definition:bockstein_couple}: the graded abelian group $A$ can be thought of as giving a lift of the pair $(V, \beta)$ to an exact couple with differential the Bockstein---except that we do not require exactness. Hence the term "Bockstein couple".
\end{remark}

A prototypical example of a Bockstein couple is one associated to a $\bbZ$-module.
\begin{example}
Let $M$ be an $\bbZ$-module in spectra. Then, we have an associated Bockstein couple given by 
\[
\Psf(M)(S) \colonequals \pi_{*}(S \otimes_{\bbZ} M).
\]
Through \cref{proposition:explicit_description_of_augmented_bockstein_modules} we can identify this with the diagram 
\[
\begin{tikzcd}
	{\pi_{*}M} & { \pi_{*} (\F_p \otimes_{\Z} M)}
	\arrow["\pi", bend left, from=1-1, to=1-2]
	\arrow["\delta", bend left, from=1-2, to=1-1]
\end{tikzcd}
\]
\end{example}

Bockstein couples are by definition representations of the quiver $\pcat$, and $\Psf(M)$ can be identified with the associated quiver of \cref{example:associated_quiver}. Thus, the following is a summary of \cref{proposition:cat_of_quiveres_groth_abelian_and_psf_a_homology_theory} and \cref{proposition:quivers_have_a_monoidal_structure} in the present context: 

\begin{proposition}
The category $\Rep(\pcat)$ of Bockstein couples has the following properties: 
\begin{enumerate}
    \item It is a Grothendieck abelian category. 
    \item It has a canonical local grading for which 
\[
\Psf \colonequals \Mod_{\bbZ}(\spectra) \rightarrow \Rep(\pcat)
\]
    is a homology theory, explicitly given by $(\X[1])(S) \colonequals \X(\Sigma S)$. 
    \item The objects $\Psf(\bbZ), \Psf(\fieldp)$ and their shifts under the local grading are compact projective generators. 
    \item It has a has a canonical symmetric monoidal structure for which $\Psf$ is lax symmetric monoidal, and symmetric monoidal when restricted to simples. 
\end{enumerate}
\end{proposition}

As we have an explicit presentation of Bockstein couples given in \cref{proposition:explicit_description_of_augmented_bockstein_modules} as certain quiver representations, we can make both the compact generators and the symmetric monoidal structure quite explicit. 

\begin{remark}[Explicit presentation the compact projective generators]
\label{remark:explicit_presentation_of_cp_generators_of_wvectbeta}
The category $\Rep(\pcat)$ is compactly generated 
by the projective objects $\Psf(\bbZ)$ and $\Psf(\F_p)$ (and their shifts).  
The Bockstein couple $\Psf(\Z)$ can be identified with the diagram
\[
\Bigg(
\begin{tikzcd}
	{\pi_{*} \bbZ} & { \pi_{*} \mathbb{F}_{p}}
	\arrow["\pi", bend left, from=1-1, to=1-2]
	\arrow["\delta", bend left, from=1-2, to=1-1]
\end{tikzcd}
\Bigg)\cong \Bigg(
\begin{tikzcd}
	{\bbZ} & {\mathbb{F}_{p}}
	\arrow[two heads, bend left, from=1-1, to=1-2]
	\arrow["0", bend left, from=1-2, to=1-1]
\end{tikzcd} \Bigg),
\]
where $\pi$ is the natural quotient map. 
By the Yoneda lemma, it is the free Bockstein couple on an integral variable of degree zero in the sense that 
\[
\Hom_{\Rep(\pcat)}(\Psf(\bbZ), A \rightleftarrows V) \cong A_{0}.
\]
for any other Bockstein couple $A \rightleftarrows V$. The other generator, $\Psf(\mathbb{F}_{p})$, can be identified with the diagram 
\[
\Bigg(
\begin{tikzcd}
	{\pi_{*} \F_p} & { \pi_{*} (\F_p \otimes_{\Z} \mathbb{F}_{p})}
	\arrow["\pi", bend left, from=1-1, to=1-2]
	\arrow["\delta", bend left, from=1-2, to=1-1]
\end{tikzcd}
\Bigg)\cong \Bigg(
\begin{tikzcd}
	{\mathbb{F}_{p}} & {\mathbb{F}_{p} \langle \beta \rangle}
	\arrow[hook, bend left, from=1-1, to=1-2]
	\arrow[hook, bend left, from=1-2, to=1-1]
\end{tikzcd} \Bigg),
\]
where $\pi$ is injective and $\delta$ is surjective. This $\pcat$-rep is freely generated in the degree zero $\fieldp$-variable in the sense that we have a canonical isomorphism
\[
\Hom_{\Rep(\pcat)}(\Psf(\fieldp), A \rightleftarrows V) \cong V_{0}.
\]
\end{remark}

\begin{warning}
Observe that all Bockstein couples of the form $\Psf(M)$, where $M$ is a $\bbZ$-module, have the special property that the sequence
\[
\begin{tikzcd}
	\ldots & {\Psf(M)(\bbZ)} & {\Psf(M)(\bbZ)} & {\Psf(M)(\mathbb{F}_{p})} & {\Psf(M)(\Sigma \bbZ)} & \ldots
	\arrow["p", from=1-2, to=1-3]
	\arrow["\pi", from=1-3, to=1-4]
	\arrow[from=1-1, to=1-2]
	\arrow["\delta", from=1-4, to=1-5]
	\arrow[from=1-5, to=1-6]
\end{tikzcd}
\]
is a long exact sequence of abelian groups. This is \emph{not} true for all Bockstein couples, as the property is not stable under taking quotients.  

Note that since this sequence is exact for the compact projective generators of \cref{remark:explicit_presentation_of_cp_generators_of_wvectbeta}, any Bockstein couple is a quotient of one with this property. 
\end{warning}

\begin{lemma}
\label{lemma:symmetric_monoidality_of_functors_from_quivers_to_vect_and_ab}
The symmetric monoidal structure on Bockstein couples has the following properties: 
\begin{enumerate}
    \item The functor $\Rep(\pcat) \rightarrow \Vect_{*}(\fieldp)$ given by
    \[
    (A \rightleftarrows V) \mapsto V
    \]
    is canonically symmetric monoidal. 
    \item The functor $\Rep(\pcat) \rightarrow \Ab_{*}$ given by 
    \[
    (A \rightleftarrows V) \mapsto A
    \]
    is canonically lax symmetric monoidal. 
\end{enumerate}
\end{lemma}

\begin{proof}
The first functor is cocontinuous, and hence a left Kan extension of the functor
\[
\pcat \rightarrow \Vect_{*}(\fieldp)
\]
given by 
\[
M \mapsto \pi_{*}(\mathbb{F}_{p} \otimes_{\bbZ} M)
\]
This is symmetric monoidal by the K\"{u}nneth formula in $\fieldp$-homology and hence so is its left Kan extension. The second statement follows from the same argument applied to 
\[
M \mapsto \pi_{*}(M)
\]
which is lax symmetric monoidal when viewed as a functor valued in graded abelian groups. 
\end{proof}

\begin{warning}
The integral part of \cref{lemma:symmetric_monoidality_of_functors_from_quivers_to_vect_and_ab} cannot be improved due to the failure of the K\"{u}nneth isomorphism in integral homology; that is, the association
\[
(A \rightleftarrows V) \mapsto A
\]
is not (strongly) symmetric monoidal. 
To see this, note that 
\[
(\Psf(\fieldp) \otimes \Psf(\fieldp))(\Z) \cong \Psf(\fieldp \otimes \fieldp)(\Z) \cong \Psf(\fieldp \oplus \Sigma \fieldp)(\Z) \cong \F_p \oplus \F_p[1] 
\]
while on the other hand 
\[
 \Tor^{0}_{\mathbb{Z}}(\Psf(\fieldp)(\bbZ), \Psf(\fieldp)(\bbZ))
 \cong \Tor^{0}_{\mathbb{Z}}(\fieldp; \fieldp) \cong \fieldp. 
\]
\end{warning}

\subsection{Integral homology cooperations}
\label{subsection:couples}

As we've shown in \cref{proposition:psf_homology_theory_on_r_modules_is_adapted},  the homology theory 
\[
\Psf \colon \Mod_{\bbZ}(\spectra) \rightarrow \Rep(\pcat)
\]
valued in the Bockstein couples is adapted and so leads to an Adams spectral sequence computing homotopy classes of maps between $\bbZ$-modules. To instead obtain an adapted homology theory on spectra, we apply the theory of \S\ref{subsection:homology_operations_for_general_quivers}.

\begin{definition}
\label{definition:pcat_integral_homology_of_a_spectrum}
For a spectrum $X$, we define
\[
\Psf(X) \colonequals \Psf(H\bbZ \otimes X).
\]
This yields a homology theory 
\[
\Psf \colon \spectra \rightarrow \Rep(\pcat)
\]
valued in Bockstein couples of \cref{definition:bockstein_couple}. 
\end{definition}

\begin{warning}
Note that the notation of \cref{definition:pcat_integral_homology_of_a_spectrum} is potentially abusive, as we denote both $\Psf \colon \spectra \rightarrow \Rep(\pcat)$ and $\Psf \colon \Mod_{\bbZ}(\spectra) \rightarrow \Rep(\pcat)$ with the same letter. Usually, it is possible to dintinguish between the two on the basis of whether the object we apply it to is a $\bbZ$-module or a spectrum. In the cases where it is useful to disambiguate, we will distinguish between the two functors using subscripts.
We will write $\Psf_{\spectra}$ for $\Psf \colon \spectra \rightarrow \Rep(\pcat)$ and
$\Psf_{\Mod_{\bbZ}}$ or $\Psf_{\bbZ}$ for $\Psf \colon \Mod_{\bbZ}(\spectra) \rightarrow \Rep(\pcat)$.
\end{warning}

\begin{remark}
Using \cref{proposition:explicit_description_of_augmented_bockstein_modules}, we see that the Bockstein couple $\Psf(X)$ attached to a spectrum can be identified with the diagram
\[
\begin{tikzcd}
	{\Hrm_{*}(X; \bbZ)} & {\Hrm_{*}(X, \mathbb{F}_{p})}
	\arrow["\pi", bend left, from=1-1, to=1-2]
	\arrow["\delta", bend left, from=1-2, to=1-1]
\end{tikzcd}
\]
and so should be thought of as encoding the integral and mod $p$ homology at the same time. 
\end{remark}

To have a good theory of homology cooperations, we need a flatness assumption, which in this case is easy to verify. 

\begin{lemma}
\label{lemma:finite_spectrum_is_projective_if_homology_has_only_zs_and_fps}
Let $X$ be a finite spectrum. Then the following are equivalent: 
\begin{enumerate}
    \item $X$ is $\pcat$-projective; that is, $\bbZ \otimes X$ is simple, 
    \item $\Hrm_{*}(X; \fieldp)$ has only $\bbZ$ and $\fieldp$ summands. 
\end{enumerate}
\end{lemma}

\begin{proof}
This follows from the fact that $\mathrm{H}_{*}(X; \bbZ) \coloneqq \pi_{*}(\bbZ \otimes X)$ and that a $\bbZ$-module in spectra is determined up to equivalence by its homotopy groups. 
\end{proof}

\begin{lemma}
\label{lemma:hbbz_is_adams_type}
The map $\Ss \rightarrow \bbZ$ is $\pcat$-Adams-type. 
\end{lemma}

\begin{proof}
By \cref{lemma:finite_spectrum_is_projective_if_homology_has_only_zs_and_fps}, we have to show that we can write both $\bbZ$ and $\fieldp$ as filtered colimits of finite spectra $P$ with the property that $\Hrm_{*}(P, \mathbb{Z})$ has no $\mathbb{Z} / p^{s}$ summands for $s > 1$. We claim that any connective spectrum $X$ whose integral homology has this property can be written as a filtered colimit of the needed form. Since both $\Hrm_{*}(\bbZ, \mathbb{Z})$ and $\Hrm_{*}(\fieldp, \mathbb{Z})$ have no $\mathbb{Z} / p^{s}$ summands for $s > 1$ \cite{kochman1982integral}, the claim will end the argument. 

Let $X_{0} \subseteq X_{1} \subseteq \ldots \subseteq X$ be a CW-filtration. Since $X \cong \varinjlim X_{k}$, it's enough to show that for each $k$, there are no $\mathbb{Z} / p^{s}$ summands inside $\Hrm_{*}(X_{k}, \mathbb{Z})$. However, $\Hrm_{k}(X_{k}, \mathbb{Z})$ is necessarily torsion-free and $\Hrm_{*}(X_{k}, \mathbb{Z}) \cong \Hrm_{*}(X, \mathbb{Z})$ for $* < k$, ending the argument. 
\end{proof}

\begin{theorem}
\label{theorem:integral_homology_in_q_comodules_is_adapted}
We have that
\begin{enumerate}
    \item there exists an exact, cocontinuous comonad $\quiversteenrod$ on $\Rep(\pcat)$, uniquely determined by the property that 
    \[
\quiversteenrod(\Psf_{\bbZ}(M)) \cong \Psf_{\bbZ}(\bbZ \otimes M) 
    \]
    naturally in $M \in \Mod_{\bbZ}$, 
    \item for any spectrum $A$, the $(\bbZ \otimes -)$-comodule structure on $\bbZ \otimes A$ induces a $\quiversteenrod$-comodule structure on  $\Psf_{\spectra}(A) \coloneqq \Psf_{\bbZ}(\bbZ \otimes A)$ and the resulting homology theory
\begin{equation}
\label{equation:comodule_valued_quiver_homology_theory}
\Psf \colon \spectra \rightarrow \Comod_{\quiversteenrod}(\Rep(\pcat))
\end{equation}
is adapted. 
\end{enumerate}
\end{theorem}

\begin{proof}
This is a combination of \cref{theorem:comodule_homology_theory_adapted_in_pcat_flat_case}, \cref{corollary:r_otimes_preserves_psf_epis_for_at_rings} and \cref{lemma:hbbz_is_adams_type}. 
\end{proof}

Since the homology theory of (\ref{equation:comodule_valued_quiver_homology_theory}) is adapted, the second page of the corresponding Adams spectral sequence (cf. \cite[Construction 2.24]{patchkoria2021adams}) is given by $\Ext$-groups in the target category: 
\begin{corollary}
\label{corollary:the_hsf_adams_has_signature_based_on_comodules}
For any pair of spectra $X, Y$, the $\Psf$-Adams spectral sequence is of signature 
\[
E_{2}^{s, t} \colonequals \Ext_{\quiversteenrod}(\Psf(X), \Psf(Y)) \Rightarrow \pi_{t-s}\map(X, Y),
\]
where the left hand side denotes the $\Ext$-groups in the category of $\quiversteenrod$-comodules. 
\end{corollary}

\subsection{Comparison with the \texorpdfstring{$\bbZ$-Adams}{integral Adams} spectral sequence}

In the previous section, we constructed the homology theory 
\[
\Psf \colon \spectra \rightarrow \Comod_{\quiversteenrod}
\]
which encodes integral and mod $p$ homology at once as a representation of a quiver, together with all of the relevant homology operations. This homology theory has an associated $\Psf$-Adams spectral sequence as in \cref{corollary:the_hsf_adams_has_signature_based_on_comodules}, obtained by mapping into an $\Psf$-injective resolution. 

The goal of this section is to compare this spectral sequence to the classical $\bbZ$-Adams spectral sequence obtained by mapping into
\[
\bbZ \otimes Y \rightrightarrows \bbZ \otimes \bbZ \otimes Y \triplerightarrow \ldots, 
\]
the $\bbZ$-based Amitsur resolution. Since the $\Psf$-based Adams spectral sequence has an $E_{2}$-term given by $\Ext$-groups in $\quiversteenrod$-comodules, this gives a concrete description of the second page of the classical $\bbZ$-based Adams spectral sequence in terms of homological algebra. 

\begin{proposition}
\label{theorem:hbbz_and_hsf_ass_are_iso_when_hsfx_projective}
Let $X, Y$ be spectra and assume that $X$ is finite and $H_*(X;\Z)$ is a sum of copies of $\Z$ and $\F_p$. Then the canonical comparison map 
\[
    {}^{\Z}{E}_r^{s,t}(X, Y) \rightarrow {}^{\Psf}E_r^{s,t}(X, Y)
\]
from the $\Z$-Adams spectral sequence to the $\Psf$-Adams spectral sequence is an isomorphism. In particular,
\[
{}^{\Z}E_2^{s,t}(X, Y) \cong \Ext^{s, t}_{\quiversteenrod}(\Psf(X), \Psf(Y)).
\]
\end{proposition}

\begin{proof}
It will be most natural to work with the language of deformations of categories, as studied in \S\ref{section:synthetic_spectra_based_on_quivers} and \S\ref{sec:def}. In these terms, 
\begin{enumerate}
\item $\Psf$-Adams is encoded by the class 
\[
\Psf\dashepi \subseteq \Fun(\Delta^{1}, \spectra)
\]
of maps of spectra which become epimorphisms after applying the functor  $\Psf$, 
\item the classical $\bbZ$-Adams spectral sequence is encoded by the class
\[
(\bbZ \otimes -)\dashepi \subseteq \Fun(\Delta^{1}, \spectra)
\]
of maps of spectra which are split epimorphisms after tensoring with $\bbZ$. 
\end{enumerate}
Any map of the second type is also of the first type, so that applying \cref{construction:deformation_as_a_2functor} to the identity endofunctor of $\spectra$ induces a comparison functor 
\begin{equation}
\label{equation:comparison_functor_from_z_derived_cat_to_psf_derived_cat}
F \colon \dcat_{\bbZ \otimes \dashepi}^{\omega}(\Sp) \to \dcat^{\omega}_{\Psf}(\Sp).
\end{equation}
between the corresponding derived categories.  

We move to the proof of the proposition. As a consequence of \cref{lemma:finite_spectrum_is_projective_if_homology_has_only_zs_and_fps}, the assumptions on $X$ guarantee that it is $\pcat$-finite projective. By a combination of This is a combination of \Cref{lemma:P_colocal_from_amitsur_to_quiver} and \Cref{lem:easy-comparison}, it is colocal for the functor of (\ref{equation:comparison_functor_from_z_derived_cat_to_psf_derived_cat}), so that in particular the induced map 
\[
\pi_{s} \map_{\dcat_{\bbZ \otimes \dashepi}^{\omega}(\Sp)}(\nu(X), \nu(\Sigma^{t} Y) / \tau) \rightarrow \pi_{s} \map_{\dcat^{\omega}_{\Psf}(\spectra)}(\nu(X), \nu(\Sigma^{t} Y) / \tau) 
\]
is an isomorphism for any $s, t \in \mathbb{Z}$. Up to reindexing, this can be identified with the $E_{2}$-terms of the corresponding Adams spectral sequences. 
\end{proof}

\subsection{Synthetic spectra as a derived $\infty$-category} 
Since $\Psf \colon \spectra \rightarrow \Comod_{\quiversteenrod}$ is Grothendieck, the perfect derived category admits a fully faithful embedding 
\[
\dcat^{\omega}_{\Psf}(\spectra) \rightarrow \widecheck{\dcat}_{\Psf}(\spectra)
\]
into the unseparated Grothendieck derived category of \cite[\S 6.4]{patchkoria2021adams}. Often, instead of working with the latter, it is more convenient to work with synthetic spectra constructed in \cref{section:synthetic_spectra_based_on_quivers}, as they are canonically symmetric monoidal.  In this short section we show that  in our case the two are equivalent: 

\begin{proposition} 
\label{lem:syn-def-equiv}
The natural comparison map
\[
\widecheck{\dcat}_{\Psf}(\Sp) \to \Syn_{\pcat}(\spectra)
\]
between the unseparated Grothendieck deformation along $\Psf$ and $\pcat$-synthetic spectra is an equivalence. In particular, the left hand side inherits a symmetric monoidal structure from the target. 
\end{proposition}

\begin{proof}
Recall that in \Cref{cor:def-equals-syn} we gave the following sufficient conditions for this comparison map to be an equivalence of categories, namely that every finite spectrum has a finite, $\Psf$-exact resolution by $\pcat$-finite projective spectra. In the case at hand, we know from \cref{lemma:finite_spectrum_is_projective_if_homology_has_only_zs_and_fps} that a finite spectrum $X$ is $\pcat$-finite projective if and only if its integral homology has only simple $p$-torsion. Given a finite spectrum $X$, we construct a $\Psf$-exact resolution of $X$ by $\pcat$-finite projectives in three steps:
\begin{enumerate}
    \item Using the Hurewicz theorem we pick a skeletal filtration of $X$ such that 
    \[
    \mathrm{H}_*(\gr_n(X); \Z) \cong \mathrm{H}_n(X;\Z).
    \]
    As every $\Z$-module in spectra is formal, this filtration becomes split on inducing up to $\Mod_{\bbZ}$. In particular, this filtration is $\Psf$-exact.
    \item We observe that a finite spectrum whose integral homology is concentrated in a single degree splits as a sum of spheres, Moore spectra $\Ss/m$ where $m$ is coprime to $p$, and Moore spectra $\Ss/p^k$. This splitting and invariance under suspension reduces the lemma to constructing resolutions for $\Ss$, $\Ss/m$ and each of 
    $\Ss/p^k$.
    \item  We now construct explicit $\Psf$-exact resolutions of the three types of spectra from the second step. Note that $\Ss$, $\Ss/m$ and $\Ss/p$ are already $\Psf$-finite projective, in the second case since according to our convention that $\bbZ$ denotes the $p$-local integers, so that
    \[
    \bbZ \otimes \Ss/m \cong \bbZ_{(p)} \otimes \Ss/m = 0.
    \]
    
    For $\Ss/p^{k}$ with $k > 1$, we use the following cofiber sequence as our resolution:
\[ 
\Ss \xrightarrow{\begin{pmatrix} p^{k-1} \\ -1 \end{pmatrix}} \Ss \oplus \Ss/p \xrightarrow{\begin{pmatrix} 1 & p^{k-1} \end{pmatrix}} \Ss/p^k. 
\]
We claim that this is $\Psf$-exact.  Indeed, on integral homology this cofiber sequence gives the short exact sequence of graded abelian groups
  \[ 0 \to \Z \xrightarrow{\begin{pmatrix} p^{k-1} \\ -1 \end{pmatrix}} \Z \oplus \Z/p \xrightarrow{\begin{pmatrix} 1 & p^{k-1} \end{pmatrix}} \Z/p^k \to 0. \]
  and on $\F_p$-homology this cofiber sequence gives the short exact sequence of graded $\F_p$-modules
  \[ 0 \to \F_p \xrightarrow{\begin{pmatrix} 0 \\ -1 \\ 0 \end{pmatrix}} \F_p \oplus \F_p \oplus \F_p[1] \xrightarrow{\begin{pmatrix} 1 & 0 & 0 \\ 0 & 0 & 1 \end{pmatrix}} \F_p \oplus \F_p[1] \to 0. \]
  Exactness on integral and mod $p$ homology implies that this is a $\pcat$-exact resolution as desired.
\end{enumerate}
\end{proof}

\begin{remark}
\label{remark:synthetic_pcat_z_modules_is_an_unseparated_derived_infty_cat}
Since every perfect $\bbZ$-module is a direct sum of (de)suspensions of $\bbZ$, $\bbZ/p^{k}$ and $\bbZ/m$ for $m$ coprime to $p$, the $\bbZ$-linear analogues of resolutions in \cref{lem:syn-def-equiv} show that the natural comparison map 
\[
\widecheck{\dcat}_{\pcat}(\Mod_{\bbZ}) \rightarrow \Syn_{\pcat}(\Mod_{\bbZ})
\]
is also an equivalence of categories. This equips the left hand side with a canonical symmetric monoidal structure. 
\end{remark}

\subsection{A pullback of synthetic categories}
\label{subsection:pullback_decomposition_of_integral_homology}

In this subsection, which lies at the technical heart of the current work, we compare the categories of synthetic spectra based on $\pcat$ and $\fieldp$. This encodes a relationship between the corresponding Adams spectral sequences, which we will explore in more detail in \S\ref{subsection:comparison_of_adams_sseqs}. 

Before we begin proving the comparison, it will be convenient to have a description of the $\fieldp$-Adams spectral sequence in terms of quiver representations. 

\begin{definition}
We say that a $\fieldp$-module in spectra is $\hcat$-simple if it is a direct sum of (de)suspensions of $\fieldp$. 
\end{definition}
The class of simples determines an induced adapted homology theory 
\[
\Hsf \colon \Mod_{\fieldp}(\spectra) \rightarrow \Vect_{*}
\]
valued in graded $\fieldp$-vector spaces. Since $\fieldp$ is a field, any perfect $\fieldp$-module is simple, so that both $i \colon \Ss \rightarrow \fieldp$ or $j \colon \bbZ \rightarrow \fieldp$ are $\hcat$-Adams-type. It follows that we have associated categories of synthetic spectra and synthetic $\bbZ$-modules. 

\begin{notation}[Induced homology theory on $\bbZ$-modules]
We will  write 
\[
\Hsf_{\bbZ} \colon \Mod_{\bbZ}(\spectra) \rightarrow \Comod_{\acat(0)_{*}}(\Vect_{*})
\]
for the induced adapted homology theory on $\bbZ$-modules given by $\Hsf_{\bbZ}(M) \colonequals \Hsf(\fieldp \otimes_{\bbZ} M)$. This homology theory corresponds to the pushforward class of epimorphisms $j_{*} \Hsf$ in the notation of \cref{construction:class_of_epis_induced_by_an_adjunction}. The target is given by the category of comodules over 
\[
\acat(0)_{*} \colonequals \Hsf_{\bbZ}(\fieldp) \cong \pi_{*}(\fieldp \otimes_{\bbZ} \fieldp)
\]
in graded vector spaces. Since any perfect $\bbZ$-module is $\hcat$-projective, by \cref{cor:def-equals-syn} there is a canonical equivalence
\[
\Syn_{\hcat}(\Mod_{\bbZ}) \cong \widecheck{\dcat}_{j_{*}\Hsf}(\Mod_{\bbZ})
\]
and the unseparated derived category associated to $j_{*} \Hsf$. 
\end{notation}

\begin{notation}[Induced homology theory on spectra]
We will write 
\[
\Hsf_{\spectra} \colon \Mod_{\bbZ}(\spectra) \rightarrow \Comod_{\acat_{*}}(\Vect_{*})
\]
for the induced adapted homology theory on spectra. The target is given by the category of comodules over the dual Steenrod algebra
\[
\acat_{*} \colonequals \Hsf_{\spectra}(\fieldp) \cong \pi_{*}(\fieldp \otimes_{\Ss} \fieldp)
\]
in graded vector spaces. Since any finite spectrum of $\hcat$-projective, by \cref{cor:def-equals-syn} there is a canonical equivalence
\[
\Syn_{\hcat}(\spectra) \cong \widecheck{\dcat}_{i_{*} j_{*}\Hsf}(\spectra)
\]
and the unseparated derived category. Here, the left hand side is just the category of $\fieldp$-synthetic spectra of \cite{pstrkagowski2018synthetic}, and we will denote it simply by $\Syn_{\fieldp}(\spectra)$. 
\end{notation}

In addition to the two homology theories on spectra and $\bbZ$-modules given above, we also have the homology theories based on the class of simples $\pcat$ studied in \S\ref{subsection:couples} and \S\ref{subsection:integral_dual-steenrod_algebra}. Since any $\pcat$-epimorphism of $\bbZ$-modules is an $\hcat$-epimorphism, we have a map
\[
q \colon (\Mod_{\bbZ}, \Psf) \rightarrow (\Mod_{\bbZ}, j_{*} \Hsf)
\]
of adapted homology theories, which induces a comparison square as follows. 

\begin{construction}
\label{cnstr:comparison-square}
By \cref{lemma:morphism_of_adjunctions_induced_by_map_of_flat_homology_theories}, the pushforward construction of \cref{definition:homology_theory_induced_by_an_adjunction} yields a right adjointable square of adapted Grothendieck homology theories
\[ 
\begin{tikzcd}
    (\Sp, i_*\Psf) \ar[r, "i^*"] \ar[d, "i_*q"] &
    (\Mod_{\Z}, \Psf) \ar[d, "q"] &
    \widecheck{\dcat}_{i_*\Psf}(\Sp) \ar[r, "i^*"] \ar[d, "i_*q"] &
    \widecheck{\dcat}_{\Psf}(\Mod_{\Z}) \ar[d, "q"] \\
    (\Sp, j_{*} i_*\Hsf) \ar[r, "i^*"] &
    (\Mod_{\Z}, j_*\Hsf) &
    \widecheck{\dcat}_{i_{*} j_* \Hsf}(\Sp) \ar[r, "i^*"] &      
    \widecheck{\dcat}_{j_*\Hsf}(\Mod_{\Z}) 
  \end{tikzcd} 
\]
and an associated right adjointable square of their unseperated derived categories. As a consequence of \cref{lem:syn-def-equiv} and \cref{remark:synthetic_pcat_z_modules_is_an_unseparated_derived_infty_cat}, we can identify the square on the right with the square of left adjoints 
\[
\begin{tikzcd}
	{\Syn_{\pcat}(\spectra)} & {\Syn_{\fieldp}(\spectra)} \\
	{\Syn_{\pcat}(\Mod_{\bbZ})} & {\Syn_{\fieldp}(\Mod_{\bbZ})}
	\arrow[from=2-1, to=2-2]
	\arrow[from=1-2, to=2-2]
	\arrow[from=1-1, to=1-2]
	\arrow[from=1-1, to=2-1]
\end{tikzcd}
\]
between the categories of $\pcat$- and $\hcat$-based synthetic spectra and $\bbZ$-modules. This promotes the square between unseparated derived categories to a square of symmetric monoidal categories. 

\end{construction}

\begin{remark}
\label{remark:functors_in_pullback_square_preserve_compactness}
As a consequence of \cref{lemma:unseparated_def_of_cpt_gen_homology_theory_is_cpt_gen}, the four derived categories in the square of \Cref{cnstr:comparison-square} are compactly generated and that in each case the corresponding synthetic analogue functor $\nu(-)$ preserves compactness. Since objects of the form $\nu(X)$ with $X$ compact are compact generators for each of these categories, this implies that all four arrows in the square preserve compactness.
\end{remark}

\begin{remark}
\label{remark:hearts_in_hz_pullback_diagram}  
Passing to the hearts in the diagram of categories of \cref{cnstr:comparison-square}, we obtain a diagram 
\begin{equation}
\label{equation:hearts_in_hz_pullback_diagram}  
\begin{tikzcd}
	{\Comod_{\quiversteenrod}(\Rep(\pcat))} & {\Comod_{\acat}(\Vect)} \\
	{\Rep(\pcat)} & {\Comod_{\acat(0)}(\Vect)}
	\arrow[from=2-1, to=2-2]
	\arrow[from=1-2, to=2-2]
	\arrow[from=1-1, to=1-2]
	\arrow[from=1-1, to=2-1]
\end{tikzcd}.
\end{equation}
After unwrapping the definitions, we see that the vertical arrows forget the comodule structure, while the horizontal arrows forget the integral variable in the sense that they are given by the association
\[
(A \rightleftarrows V) \mapsto V,
\]
where we visualize representations of $\pcat$ as diagrams as in \cref{proposition:explicit_description_of_augmented_bockstein_modules}. The $\acat(0)$-comodules structure is given by the Bockstein, and a structure of a $\quiversteenrod$-comodule upgrades that to a comodule structure over $\acat$. 
\end{remark}

\begin{theorem} \label{thm:pullback-Z}
The natural symmetric monoidal comparison functor from $\Syn_{\pcat}(\Sp)$ to the pullback $(P)$ displayed below is fully faithful.
  \[ \begin{tikzcd}
    \Syn_{\pcat}(\Sp) \ar[rrd, bend left] \ar[dr] \ar[ddr, bend right] \\
    & (P) \ar[r] \ar[d] \pullback & \Syn_{\pcat}(\Mod_{\bbZ}) \ar[d]  \\
    & \Syn_{\F_p}(\Sp) \ar[r] & \Syn_{\fieldp}(\Mod_{\bbZ})
  \end{tikzcd} \]
\end{theorem}

Before we can prove the theorem we will need a preliminary lemma.

\begin{lemma}
\label{lem:rel-proj-zmod}
Given an $M \in \Mod_{\bbZ}$, its synthetic analogue $\nu_{\Psf}(M) \in \Syn_{\pcat}(\Mod_{\bbZ}) \cong \widecheck{\dcat}_{\Psf}(\Mod_{\bbZ})$
is (co)local along
\[
q \colon \Syn_{\pcat}(\Mod_{\bbZ}) \rightarrow \Syn_{\fieldp}(\Mod_{\bbZ})
\]
if and only if the homotopy groups of $M$ are simple $p$-torsion.
\end{lemma}

\begin{proof}
Throughout the proof, we will make use of \cref{sec:def},  identifying the categories of synthetic objects with the unseparated derived categories. First, note that if $\nu_{\Psf}(M)$ is either $q$-colocal or $q$-local, then we have 
\[ 
\pi_{0} \map(\nu_{\Psf}(M), \nu_{\Psf}(M)/\tau) \cong \pi_{0} \map(\nu_{\fieldp}(M), \nu_{\fieldp}(M)/\tau)
\]
which using that $\nu(-)/\tau$ is in the heart and the description of hearts given in \cref{remark:hearts_in_hz_pullback_diagram} is equivalent to 
\[
\Hom_{\Rep(\pcat)}(\Psf(M), \Psf(M)) \cong \Hom_{\acat(0)}(\Hsf(M), \Hsf(M))
\]
Since every $\acat(0)$-comodule is simple $p$-torsion, it follows that the identity of
\[
\Psf(M) \cong (\pi_{*}(M) \rightleftarrows \pi_{*}(M; \fieldp))
\]
must have order $p$. In particular, this implies that $\pi_*(M)$ consists of simple $p$-torsion.

Next we show that $\nu_{\Psf}(\fieldp)$ is both $q$-local and $q$-colocal. Since the functor is symmetric monoidal, and $\nu_{\Psf}(\fieldp)$ is dualizable, by applying \Cref{lem:dualizing-co-local} we find that it will suffice to show that $\nu_{\Psf}(\F_p)$ is $q$-colocal. Using the fact that $\nu_{\Psf}(\F_p)$ is compact, \Cref{lem:big-to-small} and \cref{prop:colocal-to-abelian} imply that it will suffice to show that $\Psf(\F_p)$ is colocal for the induced functor 
\[
\dcat^b(\Rep(\pcat)) \rightarrow \dcat^b(\Comod_{A(0)})
\]
between bounded derived categories. 

If we write $\pcat_{p} \subseteq \pcat$ for the full subcategory spanned by direct sums of $\fieldp$, then the induced functor on the hearts can be identified with the restriction 
\[
\Rep(\pcat) \rightarrow \Rep(\pcat_{p}).
\]
It follows that the induced functor between unbounded categories can be identified with the restriction 
\[
\Fun_{\Sigma}(h \pcat, \dcat(\bbZ)) \rightarrow \Fun_{\Sigma}(h \pcat_{p}, \dcat(\bbZ)).
\]
This has a fully faithful left adjoint given by left Kan extension along the fully faithful embedding $h \pcat_{p} \hookrightarrow h \pcat$ of homotopy categories and therefore any object in the image of this adjoint, such as $\Psf(\fieldp)$, is colocal. 

We now return to the case of a general $M$ whose homotopy groups are simple $p$-torsion. Since any $\bbZ$-module is formal, any such $M$ is a sum of suspensions of copies of $\F_p$ and since $\nu_{\Psf}(-)$ preserves sums and sends suspensions to shifts, $\nu_{\Psf}(M)$ is a sum of shifts of copies of $\nu_{\Psf}(\F_p)$. Finally, we note since both of the synthetic categories are compactly-generated and the functors preserve compactness by \cref{remark:functors_in_pullback_square_preserve_compactness}, both local and colocal objects are closed under colimits. As a consquence $\nu_{\Psf}(M)$, being a sum of shifts of $\nu_{\Psf}(\F_p)$,
  is $q$-(co)local.
  \qedhere
\end{proof}

\begin{proof}[Proof (of \Cref{thm:pullback-Z}).]
In order to prove the theorem we have to show that for $X, Y \in \Syn_{\Psf}(\Sp)$, the square of mapping spectra
    \[ \begin{tikzcd}
      \map_{\Syn_{\pcat}(\Sp)}(X,Y) \ar[r] \ar[d] &
      \map_{\Syn_{\pcat}(\Mod_{\Z)}}(X,Y) \ar[d] \\
      \map_{\Syn_{\fieldp}(\Sp)}(X,Y) \ar[r] &
      \map_{\Syn_{\fieldp}(\Mod_{\Z)}}(X,Y)
    \end{tikzcd} \]
is a pullback. Here we slightly abusively write $X, Y$ for their images in the corresponding categories. We first make a series of reductions:
\begin{enumerate}
    \item Using the fact that $\Syn_{\pcat}(\Sp)$ is generated under colimits by dualizable objects and that our comparison functors are all symmetric monoidal it suffices to prove the square above is a pullback when $X$ is the unit.
    \item Using the fact that its monoidal unit is compact and that $\Syn_{\pcat}(\Sp)$ is generated under colimits and desuspensions by objects of the form $\nu_{\pcat}(P)$, where $P \in \spectra$ is finite $\pcat$-projective, it will suffice to prove the square is a pullback when $Y$ is in the image of $\nu_{\pcat}$.
    \item Since the vertical arrows in this square are equivalences $\tau$-locally it suffices to prove the square is a pullback after $\tau$-completion.
    \item In order to prove the square is a pullback after $\tau$-completion we only need to show that it is a pullback mod $\tau$.
    \item We can write $\nu_{\pcat}(P)/\tau$ as the limit of an injective resolution, therefore it suffices to argue the square is a pullback when $Y$ is of the form $\nu_{\pcat}(I)/\tau$ where $I$ is a $\Psf$-injective.
    \item Every $\Psf$-injective is the underlying spectum of a $\bbZ$-module.
\end{enumerate}
With all the reduction steps in place, we've now reduced to proving the following: Given a spectrum $M$ which admits a structure of a $\bbZ$-module, the diagram of mapping spectra
    \begin{center}
      \begin{tikzcd}
        \map_{\Syn_{\pcat}(\Sp)}(\monunit, \nu_{\pcat}(M)) \ar[r] \ar[d] & \map_{\Syn_{\F_p}(\Sp)}(\monunit, \nu_{\fieldp}(M)) \ar[d] \\
        \map_{\Syn_{\pcat}(\Z)}(\monunit, \nu_{\pcat}(\Z \otimes M)) \ar[r] & \map_{\Syn_{\F_p}(\Z)}(\monunit, \nu_{\fieldp}(\Z \otimes M))
      \end{tikzcd}
    \end{center}
is a pullback. A choice of a $\bbZ$-module structure on $M$, extends this diagram to
    \begin{center}
      \begin{tikzcd}
        \map_{\Syn_{\pcat}(\Sp)}(\monunit, \nu_{\pcat}(M)) \ar[r] \ar[d] & \map_{\Syn_{\F_p}(\Sp)}(\monunit, \nu_{\fieldp}(M)) \ar[d] \\
        \map_{\Syn_{\pcat}(\Z)}(\monunit, \nu_{\pcat}(\Z \otimes M)) \ar[r] \ar[d] & \map_{\Syn_{\F_p}(\Z)}(\monunit, \nu_{\fieldp}(\Z \otimes M)) \ar[d] \\
        \map_{\Syn_{\pcat}(\Z)}(\monunit, \nu_{\pcat}(M)) \ar[r] & \map_{\Syn_{\F_p}(\Z)}(\monunit, \nu_{\fieldp}(M))
      \end{tikzcd},
    \end{center}
where the vertical arrows in the bottom square are given by multiplication. To show that the top square is a pullback, we will show that the vertical composites are equivalences and then show that the bottom square is a pullback. The first part follows from the right-adjointability of the square of \Cref{cnstr:comparison-square}. 

In order to show the bottom square is a pullback we examine the induced map between vertical fibers. We will write $J \colonequals \fib(\Z \otimes \Z \to \Z)$, which is simple $p$-torsion by \cite{kochman1982integral}. Using the fact the multiplication $\Z \otimes M \to M$ is a split epimorphism of spectra, the map of vertical fibers takes the form
\[ 
\map_{\Syn_{\pcat}(\Mod_{\Z)}}(\monunit, \nu(J \otimes_\Z M)) \to \map_{\Syn_{\fieldp}(\Mod_{\Z})}(\monunit, \nu(J \otimes_\Z M)).
\]
As $J$ is simple $p$-torsion, $V \otimes_\Z M$ is as well, therefore $\nu_{\Psf}(V \otimes_{\Z} M)$ is local by \Cref{lem:rel-proj-zmod}, ending the argument. 
\end{proof}

\subsection{Calculational consequences}
\label{subsection:comparison_of_adams_sseqs}

In \Cref{thm:pullback-Z}, we showed that a certain natural square of synthetic $\infty$-categories is a pullback. In this section, we describe consequences of this statement in concrete calculational terms, describing the relationship between the between the various Adams spectral sequences. In particular, we give a complete description of the $\bbZ$-based Adams spectral sequence for the sphere. 

\begin{corollary}
\label{cor:colocal Z}
Let $X$ be a spectrum whose integral homology groups are simple $p$-torsion. Then, $\nu_{\pcat}(X)$ is (co)local for 
\[
\Syn_{\pcat}(\Sp) \rightarrow \Syn_{\fieldp}(\Sp). 
\]
\end{corollary}

\begin{proof}
This is immediate from \cref{thm:pullback-Z} and \cref{lem:rel-proj-zmod}. 
\end{proof}

\begin{corollary} \label{cor:Fp-Z-iso}
Let $X$ be a spectrum whose integral homology groups are simple $p$-torsion. Then the canonical comparison maps induce isomorphisms 
\[ 
{}^{\Z}E_r^{s,t}(X) \cong {}^{\Psf}E_r^{s,t}(X) \cong {}^{\F_p}E_r^{s,t}(X) 
\]
between, respectively, the classical $\bbZ$-based Adams spectral sequence, the Adams spectral sequence based on $\Psf \colon \spectra \rightarrow \Comod_{\quiversteenrod}(\Rep(\pcat)$, and the $\fieldp$-based Adams spectral sequence. 
\end{corollary}

\begin{proof}
The first isomorphism is \cref{theorem:hbbz_and_hsf_ass_are_iso_when_hsfx_projective}. The second comparison map can be identified up to regrading with 
\[
\pi_{s} \map_{\Syn_{\pcat}(\Sp)}(\nu_{\pcat}(S^{t}), \nu_{\pcat}(X)) \rightarrow \pi_{s} \map_{\Syn_{\fieldp}(\Sp)}(\nu_{\fieldp}(S^{t}), \nu_{\fieldp}(X)) 
\]
which is an isomorphism since $\nu_{\pcat}(X)$ is colocal. 
\end{proof}

The assumption that $\Hrm_{*}(X; \bbZ)$ is simple $p$-torsion is quite strong; the following result shows how to calculate the $E_{2}$-term in the general case. 

\begin{proposition}
\label{proposition:mv_lss_between_hsf_and_hfp_e2_terms}
For any spectra $X, Y$, the canonical comparison morphism 
\begin{equation}
\label{equation:comparison_between_hsf_and_hfp_ass_in_mv_seq}
\tensor[^\Psf]{E}{_2}(X, Y) \rightarrow \tensor[^{\fieldp}]{E}{_2}(X, Y)
\end{equation}
of Adams spectral sequences and the maps between $\Ext$-groups 
\begin{equation}
\label{equation:comparison_between_abm_and_bm_ext_in_mv_seq}
\Ext_{\Rep(\pcat)}(\Psf(X), \Psf(Y)) \rightarrow \Ext_{\acat(0)}(\Hrm_{*}(X; \fieldp), \Hrm_{*}(Y; \fieldp))
\end{equation}
induced by the forgetful functor $\Rep(\pcat) \rightarrow \Comod_{\A(0)}$ fit into a Mayer--Vietoris long exact sequence of the form 
\begin{align*}
\cdots \rightarrow \tensor[^\Psf]{E}{_2^{s, t}}(X,Y)  \rightarrow \Ext^{s, t}_{\Rep(\pcat)}(\Psf(X),\Psf(Y)) \oplus \tensor[^{\fieldp}]{E}{_2^{s, t}}(X,Y) \rightarrow \\ \rightarrow \Ext^{s, t}_{\acat(0)}(H_*(X),H_*(Y)) \rightarrow  \tensor[^\Psf]{E}{_2^{s+1, t}}(X,Y) \rightarrow \cdots.
\end{align*}
In particular, if (\ref{equation:comparison_between_abm_and_bm_ext_in_mv_seq}) are isomorphisms, then so is (\ref{equation:comparison_between_hsf_and_hfp_ass_in_mv_seq}).
\end{proposition}

\begin{proof}
After identifying the $\Ext$-groups with the homotopy of the cofiber of $\tau$, this is the long exact sequence associated to the pullback diagram 
\[\begin{tikzcd}
	{\map_{\Syn_{\pcat}(\Sp)}(\nu(X), \nu(Y) / \tau)} & {\map_{\Syn_{\fieldp}(\Sp)}(\nu(X), \nu(Y)/ \tau)} \\
	{\map_{\Syn_{\pcat}(\Mod_{\bbZ})}(\nu(\Z \otimes X), \nu(\Z \otimes Y) /\tau} & {\map_{\Syn_{\fieldp}(\Mod_{\bbZ})}(\nu(\Z \otimes X), \nu(\Z \otimes Y) / \tau)}
	\arrow[from=2-1, to=2-2]
	\arrow[from=1-1, to=2-1]
	\arrow[from=1-2, to=2-2]
	\arrow[from=1-1, to=1-2]
\end{tikzcd}\]
\end{proof}

The Mayer--Vietoris long exact sequence of \cref{proposition:mv_lss_between_hsf_and_hfp_e2_terms} is a powerful tool for computations. For example, we are able to elegantly determine the $E_{2}$-term of the $\Psf$-based Adams spectral sequence for the sphere, which by \cref{theorem:hbbz_and_hsf_ass_are_iso_when_hsfx_projective} coincides with the classical $\bbZ$-Adams spectral sequence associated to the cobar resolution.

\begin{theorem}
\label{theorem:e2_page_of_the_sphere}
We have a canonical isomorphism 
\[
\tensor[^{\Psf}]{E}{_2^{s, t}}(\Ss, \Ss) \cong \Ext^{s, t}_{\Rep(\pcat)}(\Psf(\Ss), \Psf(\Ss)) \times_{\Ext^{s, t}_{\acat(0)}(\fieldp, \fieldp)} \Ext^{s, t}_{\acat}(\fieldp, \fieldp) \cong \mathbb{Z} \times_{\fieldp[h_{0}]} \Ext^{s, t}_{\acat}(\fieldp, \fieldp)
\]
Explicitly, we have 
\[
\tensor[^{\Psf}]{E}{_2^{s, t}}(\Ss, \Ss) = \begin{cases}
  \bbZ  & s-t = 0, s = 0 \\
  0 & s-t = 0, s \neq 0 \\
  \Ext_{\acat}^{s, t}(\fieldp, \fieldp) & s-t \neq 0
\end{cases}
\]
and the canonical comparison map 
\[
\tensor[^{\Psf}]{E}{_2^{s, t}}(\Ss, \Ss) \rightarrow \tensor[^{\fieldp}]{E}{_2^{s, t}}(\Ss, \Ss)
\]
between $\Psf$- and $\fieldp$-Adams spectral sequences is an isomorphism outside of the line $s-t = 0$. 
\end{theorem}

\begin{proof}
We will use the Mayer-Vietoris long exact sequence of \cref{proposition:mv_lss_between_hsf_and_hfp_e2_terms}. Since $\acat(0)_{*}$ is dual to an exterior algebra, we have 
\[
\Ext_{\acat(0)}^{s, t}(\fieldp) \cong \fieldp[h_{0}],
\]
where $|h_{0}| = (1, 1)$ is the class of the extension
\[
0 \rightarrow \Hrm_{*}(\Ss) \rightarrow \Hrm_{*}(\Ss/p) \rightarrow \Hrm_{*}(\Ss^{1}) \rightarrow 0,
\]
As we presented this short extension as coming from a cofiber sequence of spectra, it lifts to one of $\acat$-comodules, so that the map 
\[
\Ext_{\acat}(\fieldp; \fieldp) \rightarrow \Ext_{\acat(0)}(\fieldp; \fieldp)
\]
is surjective. It follows that the Mayer-Vietoris long exact sequence of \cref{proposition:mv_lss_between_hsf_and_hfp_e2_terms} is a collection of short exact sequences. 

We are only left with calculating $\Ext_{\Psf}^{*}(\Psf(\Ss), \Psf(\Ss))$. However, since $\Psf(\Ss)$ is freely generated as a Bockstein couple in the integral variable by \cref{remark:explicit_presentation_of_cp_generators_of_wvectbeta}, the groups $\Ext_{\Psf}^{s}$ vanish for $s > 0$ and we have
\[
\Ext_{\Psf}^{0}(\Psf(\Ss), \Psf(\Ss) \cong \Hrm_{*}(\Ss; \bbZ) \cong \bbZ. 
\]
This ends the argument. 
\end{proof}

\subsection{An example: the Brown--Peterson spectrum}
\label{subsec:BP}

The moprhism of rings $\bbZ \to \fieldp$ induces a map of spectral sequences from the $\bbZ$-based Adams spectral sequence to the $\fieldp$-based Adams spectral sequence. As $\bbZ$ is closer to the sphere spectrum than $\fieldp$, it is perhaps natural to expect that the integral Adams spectral sequence is ``better behaved'' (or at least ``more convergent'') than the mod $p$ one. This is indeed the case in the case of the sphere  described in \cref{theorem:e2_page_of_the_sphere}, but as we now explain, it is not the case in general. 

Let $\BP$ denote the Brown-Peterson spectrum, whose classical Adams spectral sequence is of signature 
\[
\Ext_{\acat}^{t, s}(\fieldp, \Hrm_{*}(\BP; \fieldp)) \cong \fieldp[v_{0}, v_{1}, \ldots] \Rightarrow \pi_{t-s}\BP
\]
with $|v_{i}| = (2p^{i}-1, 1)$. Since all of these elements are of even Adams degree, there are no possible differentials and the spectral sequence collapses on the second page \cite[\S 3.1]{ravenel_complex_cobordism}. We will show that this is not what happens in the integral case; in fact, we will see that the $\bbZ$-based Adams spectral sequence for $\BP$ does not collapse on any finite page. 

\begin{lemma}
\label{theorem:hz_adams_e2_page_for_bp}
The $E_{2}$-page of the $\bbZ$-Adams spectral sequence for the Brown-Peterson spectrum is given by 
\[
E_{2}^{s, *} \cong \begin{cases} \bbZ[t_{1}, t_{2}, \ldots] \times_{\fieldp[t_{1}, t_{2}, \ldots]} \fieldp & s=0 \\ 0 & s=1 \\ \mathrm{coker}(\fieldp[v_{0}, v_{1}, \ldots] \rightarrow \fieldp[v_{0}, t_{1}, t_{2}, \ldots])_{s-1, *} & s \geq 2 \end{cases}
\]
where $|v_i| = (1, 2p^i - 1)$, $|t_{i}| = (0, 2p^{i}-2)$ and $v_i \mapsto v_0t_i$. As a consequence: 
\begin{enumerate}
    \item $E_2^{0,*}$ is freely generated as an abelian group by generators $1$ and $p \cdot t_I$, where $t_I$ runs through all monomials in the $t_i$.
    \item for $s \geq 2$, $E_2^{s,*}$ is an $\fieldp$-vector space with basis given by elements of the form $\delta(v_{0}^{s-1} \cdot t_I)$, where $t_I$ runs through monomials in the $t_{i}$ of degree at least $s$.
    \end{enumerate}
\end{lemma}

\begin{proof}
Since the sphere is $\pcat$-finite projective, the $\bbZ$-Adams is isomorphic to the $\Psf$-Adams whose $E_{2}$-term is given by $\Ext$-groups in $\quiversteenrod$-comodules in $\Rep(\pcat)$. By \Cref{proposition:mv_lss_between_hsf_and_hfp_e2_terms}, we have a Mayer-Vietoris long exact sequence of the form 
\begin{equation}
\label{equation:mv_sequence_for_hz_term_for_bp}
0 \rightarrow \Ext^{0}_{\quiversteenrod}(\Psf(\BP)) \rightarrow \Ext^{0}_{\Rep(\pcat)}(\Psf(\BP)) \oplus \Ext^{0}_{\acat}(\Hrm_{*}(\BP; \fieldp)) \rightarrow \Ext_{\acat(0)}^{0}(\Hrm_{*}(\BP; \fieldp)) \rightarrow \ldots,
\end{equation}
where we have omitted the source of each $\Ext$-group, which is always the homology of the sphere. The $\Ext$-groups over the Steenrod algebra are well known and we have
\[ \Ext_{\acat}(\fieldp, \Hrm_{*}(\BP; \fieldp)) \cong \fieldp[v_{0}, v_{1}, v_{2}, \ldots] \]
with $|v_{i}| = (1, 2p^i-1)$. Using the fact that $H_*(\BP;\F_p)$ is even we can compute that
\begin{align*}
\Ext_{\acat(0)}(\fieldp, \Hrm_{*}(\BP; \fieldp)) \cong \fieldp[v_{0}, t_{1}, t_{2}, \ldots]
\end{align*}
with $|v_{0}| = (1, 1)$ and $|t_{i}| = (2p^{i}-2, 0)$; this is the $E_{2}$-term of the classical Adams spectral sequence computing $\Hrm_{*}(\BP; \bbZ) \cong \bbZ[t_{1}, t_{2}, \ldots]$, with $v_{0}$ detecting $p$ and $t_{i}$ detecting $t_{i}$. 
The comparison map 
\begin{equation}
\label{equation:comparison_map_for_ass_for_pibp_and_hbp}
\Ext_{\acat}(\fieldp, \Hrm_{*}(\BP; \fieldp)) \rightarrow \Ext_{\acat(0)}(\fieldp, \Hrm_{*}(\BP; \fieldp)) 
\end{equation}
is determined by 
\[ v_{0} \mapsto v_{0} \]
and
\[ v_{i} \mapsto v_{0} a_{i} \] 
for $i \geq 1$, see \cite[\S 3.1]{ravenel_complex_cobordism}. This is an injective ring homomorphism.

Since $\bbZ \otimes \BP$ is a direct sum of $\bbZ$, by \cref{remark:explicit_presentation_of_cp_generators_of_wvectbeta} $\Psf(\BP)$ is projective as a Bockstein couple and 
\[
\Ext^{*}_{\Rep(\pcat)}(\Psf(\BP)) \cong \Hrm_{*}(\BP; \bbZ) \cong \bbZ[t_{1}, t_{2}, \ldots],
\]
concentrated in $\Ext$-degree zero. Together with the injectivity from (\ref{equation:comparison_map_for_ass_for_pibp_and_hbp}),
the result follows.
\end{proof}

\begin{remark}[Differentials in the $\bbZ$-Adams for $\BP$] 
Note that since $\pi_{*}\BP \rightarrow \Hrm_{*}(\BP; \bbZ)$ is injective, with image the subring generated by $p \cdot t_{i}$, all elements in the homotopy groups are of Adams filtration zero, so that none of the positive $s$-degree elements appearing in \cref{theorem:hz_adams_e2_page_for_bp} can survive the spectral sequence.

Informally, $E_{2}$-page correctly detects that $t_{i} \in \Hrm_{*}(\BP; \bbZ)$ itself is not in the image of the Hurewicz map (only $p \cdot t_{i}$ is), but makes the mistake of thinking that $p \cdot t_I$, where $t_I$ is a monomial of degree at least $2$, might be. However, since such an element is not in the image of the Hurewicz map, it must support a non-trivial differential---the elements of positive $\Ext$-degree arise as targets of these differentials. One can show that the differentials are determined by 
\[
p^{k} \cdot t_I \mapsto \partial(v_{0}^{k} \cdot t_I),
\]
Note that such a differential is of length $k-1$. In particular, there are non-zero differentials of arbitrary length. 
\end{remark}

%% \begin{remark}[What is integral Adams good for?]
%% A reader might be disappointed in the fact that, as we show above, the $\bbZ$-Adams computing the homotopy groups of the sphere is essentially isomorphic to the classical Adams spectral sequence, and the one for $\BP$ is in fact somewhat more complex than its classical counterpart. However, even a less-convergent Adams spectral sequence (or one which detects the same element in different bidegree) can often be quite useful. 

%% As a concrete example, in \cref{theorem:integral_todas_obstructions_groups_for_bp_vanish} below we show that the $\bbZ$-based Toda's obstructions groups to the existence of the Brown-Peterson spectrum vanish, unlike their mod p counterparts. The key observation is that in the integral case studied in \cref{theorem:hz_adams_e2_page_for_bp}, almost the whole (that is, outside of $s=0$) of the $E_{2}$-page of the Adams spectral sequence for $\BP$ is concentrated in odd Adams degree. 
%% \end{remark}

%\subsection{The existence of the Brown-Peterson spectrum}

We now describe how our calculation of the integral Adams $E_{2}$-term for the Brown-Peterson spectrum is related to the obstruction theory to its existence. In their fundamental paper titled ``A spectrum whose $\bbZ/p$-cohomology is the algebra of reduced $p$-th powers'' \cite{brown1965spectrum}, Brown and Peterson construct the spectrum we now denote by $\BP$ by a variation on Toda's method of constructing an Adams resolution of the required spectrum. This was the first known construction of $\BP$ (which is now more commonly constructed using Quillen's idempotent on $MU_{(p)}$) and a major advance in stable homotopy theory at the time. 

Toda's method, first introduced in \cite{toda1971spectra}, gives an obstruction to realizing an $\acat_{*}$-comodule $M$ as a homology of a spectrum. The obstructions are defined inductively and lie in 
\begin{equation}
\label{equation:todas_obstructions_groups_for_mod_p_homology}
\theta_{n} \in \Ext_{\acat}^{n+2}(M, M),
\end{equation}
where $n \geq 1$. One can view Toda's work as the linear precursor to Goerss-Hopkins theory, which can more generally give obstructions to existence of spectra with additional structure, such as $\mathbb{E}_{n}$-multiplication. 

Brown and Peterson work using similar methods to those of Toda, and it is natural to expect their existence statement is an immediate application of Toda's obstruction theory. However, as was pointed to out by Dylan Wilson, this is not the case. One can compute the obstruction groups based on mod $p$ homology and show that they do \emph{not} vanish.

The reason Brown and Peterson do not run into the potential non-triviality of this obstruction is that they also keep track of the integral homology of the spectrum they're trying to construct, and its interaction with mod $p$ homology via the Bocksteins, see \cite[Lemma 3.1]{brown1965spectrum}.

Thus, it is natural to expect that what Brown and Peterson are really doing is setting up a variant of Toda's obstruction theory based on integral homology. However, as the relevant obstructions groups are expected to be the same as those appearing on the $E_{2}$-page of the corresponding Adams spectral sequence (as in (\ref{equation:todas_obstructions_groups_for_mod_p_homology}) above), and there was no known general formula for these $E_{2}$-pages, their work was never phrased in this language. 

Feeding $\pcat$-based synthetic spectra of \S\ref{section:synthetic_spectra_based_on_quivers} into the general machinery of Goerss-Hopkins theory \cite{moduli_spaces_of_commutative_ring_spectra}, \cite{pstrkagowski2021abstract}, one obtains an obstruction theory to the existence of a spectrum with prescribed integral homology, with obstructions living on an appropriate $E_{2}$-page of the $\Psf$-based Adams spectral sequence. The latter can be identified with $\Ext$-groups in $\quiversteenrod$-comodules, which are readily computable. 

The following result explains why Brown and Peterson do not run into obstructions in their proof of the existence of of a spectrum with mod $p$ homology isomorphic to $\Hrm_{*}\BP$. 

\begin{theorem}
\label{theorem:integral_todas_obstructions_groups_for_bp_vanish}
The $\bbZ$-based Toda's obstruction groups
\[
\Ext^{n+2, n}_{\quiversteenrod}(\Psf(\BP), \Psf(\BP))
\]
to the existence of the Brown-Peterson spectrum vanish for all $n \geq 0$.
\end{theorem}

\begin{proof}
The idea is to reduce to the case of the groups $\Ext_{\quiversteenrod}(\Psf(\Ss), \Psf(\BP))$, which we computed in  \cref{theorem:hz_adams_e2_page_for_bp}. The $\quiversteenrod$-comodule $\Psf(\BP)$, of the form 
\[
\begin{tikzcd}
	{\Hrm_{*}(X; \bbZ)} & {\Hrm_{*}(X; \mathbb{F}_{p})}
	\arrow["\pi", bend left, from=1-1, to=1-2]
	\arrow["\delta", bend left, from=1-2, to=1-1]
\end{tikzcd}
\]
has a natural increasing filtration where $F_{k} \Psf(\BP)$ is a subquiver of elements of internal degree at most $k$, in both integral and mod $p$ homology. Each associated graded piece is concentrated in a single internal degree and can be identified with a direct sum of homologies of the sphere, indexed by the integral homology of $\BP$. Thus, we can write 
\[
\mathrm{gr}_{\bullet} \Psf(\BP) \cong \Hrm_{*}(\BP; \bbZ) \otimes_{\bbZ} \Psf(\Ss),
\]
where on the right hand side we have a tensor product of an abelian group with an object of an abelian category. 

Mapping this filtration into $\Psf(\BP)$ in the derived category gives a spectral sequence of signature 
\begin{equation}
\label{equation:spectral_sequence_computing_todas_groups_for_bp}
E_{1}^{s, t, k} \cong \Hom_{\bbZ}(\Hrm_{k}(\BP; \bbZ), \Ext^{s, t}_{\quiversteenrod}(\Psf(\Ss), \Psf(\BP))) \Rightarrow \Ext^{s, t-k}(\Psf(\BP), \Psf(\BP))
\end{equation}
Here, the expression $t-k$ on the right hand side comes from the fact that $H_{k}(\BP; \bbZ)$ is in internal degree $k$. 

By \cref{theorem:hz_adams_e2_page_for_bp}, the groups $\Ext^{s, t}_{\quiversteenrod}(\Psf(\Ss), \Psf(\BP))$ are zero for $s > 0$ and in even Adams degree $t-s$. Since $\Hrm_{*}(\BP; \bbZ)$ is concentrated in even degree, the same is true for the first page of (\ref{equation:spectral_sequence_computing_todas_groups_for_bp}), where the trigraded Adams degree is $t-k-s$. It follows that the same is true for $\Ext^{s, t}(\Psf(\BP), \Psf(\BP))$; in particular, these groups vanish for $s=n+2$, $t=n$ and $n \geq 0$, as claimed.
\end{proof}

\section{Further examples}
\label{sec:exm}

In this section we explore several examples of homology theories that are not Adams-type, but do admit a well behaved Adams-type quiver homology theory. 

%Here we should describe some reasonable choices of $R$ and a class of simples that should lead to nice calculations. If any comparisons can be made with what is already in the literature, that would be great. 

\subsection{Connective Morava $K$-theory} 

Connective Morava $K$-theory $k(n)$ is a complex oriented homology theory with coefficients $k(n)_{*} \cong \mathbb{F}_{p}[v_{n}]$. The latter is a local ring in a graded sense (that is, it is a local Dirac ring in the sense of \cite[Definition 2.7]{diracgeometry1}) and the Quillen formal group of $k(n)$ is of infinite height at the special point, corresponding to $ \mathbb{F}_{p}$, and of height $n$ at the generic point corresponding to periodic Morava $K$-theory, $K(n)$. In this sense, connective Morava $K$-theories are the natural positive height analogue of the $p$-local integers. 

This suggests that a natural way to approach the $k(n)$-Adams spectral sequence is through a quiver homology thoery with a class of simples analogous to the one we employed in the integral case. 

\begin{definition}
\label{definition:simple_kn_module}
We say that a $k(n)$-module is simple if it is a finite direct sum of (de)suspensions of $k(n)$ and $\mathbb{F}_{p}$ where we give $\F_p$ the $k(n)$-module structure coming from the isomorphism $\F_p \cong \tau_{\leq 0}k(n)$. 
\end{definition}

\begin{notation}
In this section, we denote the full subcategory of simple $k(n)$-modules by $\pcat \subseteq \Mod_{k(n)}$. 
\end{notation}

A $\pcat$-representation $\X$ for the above class of simples can be readily described analogously to the case of Bockstein couples of \S\ref{subsection:couples}, with the role of $p$ being played by $v_{n}$.
$\pcat$-representations can be visualized as diagrams
\begin{equation}
\label{equation:vn_bockstein_module}
\begin{tikzcd}
	{A \colonequals \X(k(n))} & {\X(\mathbb{F}_{p}) \equalscolon V}
	\arrow["\delta_{v_n}", bend left, from=1-2, to=1-1]
	\arrow["\pi", bend left, from=1-1, to=1-2]
\end{tikzcd},
\end{equation}
where 
\begin{enumerate}
    \item $A$ is a graded $k(n)_{*}$-module, 
    \item $V$ is a graded $\fieldp$-vector space,
    \item the maps $\pi$ and $\delta_{v_n}$
    satisfy relations
    $\pi \circ v_n = 0$,
    $v_n \circ \delta_{v_n} = 0$ and
    $\delta_{v_n} \circ \pi = 0$. 
\end{enumerate}
Thus, one could refer to $\pcat$-reps as "$v_{n}$-Bockstein couples".

%\begin{warning}
%Note that the arrows in (\ref{equation:vn_bockstein_module}) run in the opposite direction than in the case of integral Bockstein couples discussed in \S\ref{subsection:couples}. This is because the Eilenberg-MacLane spectrum $\mathbb{Z}$ admits a $\mathbf{E}_{\infty}$-ring structure, inducing a  canonical self-duality of the category of simples, identifying covariant and contravariant presheaves. 

%It is well-known that connective Morava $K$-theories cannot be made $\mathbf{E}_{2}$, so that no such canonical identification is possible in this setting.
%\end{warning}

\begin{proposition}
Connective Morava $k$-theory $k(n)$ is Adams-type with respect to the class of simples of \cref{definition:simple_kn_module}. In particular, it is $\pcat$-flat. 
\end{proposition}

\begin{proof}
We have $\Map_{k(n)}^{\Sp}(k(n), k(n)) \cong k(n)$ and $\Map_{k(n)}^{\Sp}(\fieldp, k(n)) \cong \Sigma^{-1-q} \fieldp$ as spectra, where $q = |v_{n}| = 2p^{n}-2$. Thus, we just have to check that $k(n)$ and $\fieldp$ can be written as filtered colimits of finite spectra $P$ such that $k(n)_{*}DP \cong k(n)^{-*}P$ is a direct sum of shifts of $k(n)_{*}$ and $k(n)_{*}/(v_{n})$. Note that the latter is equivalent to the same condition on $k(n)_{*}P$, as the coefficient ring is of cohomological dimension one and so the universal coefficient spectral sequence collapses. 

Recall that there exists an isomorphism of $k(n)_{*}$-modules
\[
k(n)_{*}k(n) \cong F \oplus V, 
\]
where $F$ is a free $k(n)_{*}$-module and where $V$ is a $\fieldp \cong k(n)_{*}/(v_{n})$-vector space \cite[Theorem 4.3]{yagita1980steenrod}. Similarly, since $v_{n}$ acts by zero on $\fieldp$, we have 
\[
k(n)_{*}\fieldp \cong 0 \oplus V',
\]
where $V'$ is a $k(n)_{*}/v_{n}$-vector space. 
To complete the proof we will show the following sub-claim:

\begin{itemize}
    \item Let $X$ be a connective finite-type spectrum and let
    \[
    X_{0} \subseteq X_{1} \subseteq \ldots \subseteq X 
    \]
    be a skeletal filtration for $X$.
    If $k(n)_*(X) \cong F \oplus V$ where $F$ is a free $k(n)_*$-module and $V$ is an $\F_p$-vector space (i.e. $v_n$ acts by zero), 
    then the same is true for $X_{k}$ for all $k \geq 0$.
\end{itemize}

%We will show that for any CW-filtration through finite spectra
%\[
%X_{0} \subseteq X_{1} \subseteq \ldots \subseteq X 
%\]
%of a connective finite type spectrum $X$ whose $k(n)$-homology admits a decomposition of this form, the same is true for $X_{k}$ for all $k \geq 0$. Since $X \cong \colim X_{k}$, applying this to $X = k(n)$ and $X = \fieldp$ gives the needed decomposition as a filtered colimit. 

Suppose by contradiction that there exists a $k \geq 0$ and $l \geq 0$ and a class $x \in k(n)_{l}X_{k}$ such that $v_{n} x \neq 0$ but $v_{n}^{s} x = 0$ for some $s > 1$. Choose the smallest $k$ possible. Since 
\[
\mathrm{coker}(k(n)_{*}X_{k-1} \rightarrow k(n)_{*}X_{k}) \subseteq k(n)_{*}(X_{k}/X_{k-1})
\]
is torsion-free, $x$ lifts to an element $\widetilde{x} \in k(n)_{l}X_{k-1}$ where necessarily $v_{n}^{s} \widetilde{x} \neq 0$ as we chose $k$ to be the smallest. It follows that $v_{n}^{s} \widetilde{x}$ is in the image of a non-$v_{n}$-divisible element of 
\[
k(n)_{l+sq}(\Sigma^{-1} X_{k} / X_{k-1})
\]
and thus that $k = l+sq+1$. It follows that $k(n)_{*}X_{k} \rightarrow k(n)_{*}X$ is an isomorphism in degrees $* \leq l+sq$ so that we also have $v_{n} x \neq 0$ and $v_{n}^{s} x = 0$ inside $k(n)_{*}X$. The latter contradicts our assumption about $X$, ending the argument. 
\end{proof}

It follows from \cref{proposition:for_at_rings_r_otimes_descends_to_a_quiver_comonad} that the comonad 
\[
k(n) \otimes - \colon \Mod_{k(n)} \rightarrow \Mod_{k(n)}
\]
descends to a unique exact comonad 
$ \quiversteenrod \colon \Rep(\pcat) \rightarrow \Rep(\pcat) $
on the category of $\pcat$-reps. The latter has the property that the homology theory
\[
\Hsf \colon \spectra \rightarrow \Rep(\pcat) \quad \text{given by} \quad \Hsf(X)(S) \coloneqq \pi_0\Map_{k(n)}(S, k(n) \otimes X)
\]
has a canonical, adapted lift to $\quiversteenrod$-comodules. 

By a colocality argument, the $\Hsf$-based Adams spectral sequence coincides with the classical $k(n)$-based Adams when the source $X$ is a sphere (or more generally, a spectrum whose $k(n)_{*}$-homology has only simple $v_{n}$-torsion). Thus, we deduce the following: 

\begin{theorem}
Let $X$ be a spectrum whose $k(n)_{*}$-homology has only torsion free and simple $v_{n}$-torsion summands. Then, the $k(n)$-based Adams spectral sequence %of signature 
%\[
%E_{1}(X, Y) = [X, k(n)^{\otimes s} \otimes Y]_{t} \Rightarrow [X, Y]_{t-s}
%\]
computing maps from $X$ to $Y$ has $E_{2}$-page given by 
\[
E_{2}^{s,t}(X, Y) \cong \Ext^{s, t}_{\quiversteenrod}(\Hsf(X), \Hsf(Y)),
\]
where the $\Ext$-groups are taken in $\quiversteenrod$-comodules in $\pcat$-representations with respect to class of simples of \cref{definition:simple_kn_module}.
\end{theorem}

The data of the $\pcat$-rep associated to the Morava $K$-theory of a spectrum, which we depicted in (\ref{equation:vn_bockstein_module}), is relatively straightforward. Instead, the complexity of the $k(n)$-based Adams spectral sequence is encoded in the comonad $\quiversteenrod$ and the associated category of comodules. 

It is natural to ask whether it is possible to describe the category of $\quiversteenrod$-comodules directly in terms of familiar structures, similarly to what we have done in the \S\ref{subsection:pullback_decomposition_of_integral_homology} in the case of integral homology. We will not pursue this in this paper, but we believe this is a natural and interesting question, as this decomposition would naturally describe the interaction between homotopy of finite spectra and their $K(n)$-local analogues.
%However, in terms of encouraging future research in this direction, let us describe the germ of the idea. 

\begin{remark}[Decomposing the comodule category for connective Morava $K$-theories] 
The category of $\quiversteenrod$-comodules is the universal target of a homology theory which sends all $k(n)_{*}$- and $\Hrm_{*}(-, \mathbb{F}_{p})$-surjective maps of spectra to epimorphisms. Since both mod $p$ homology and non-connective Morava $K$-theory homology $K(n)_{*}(-) \coloneqq v_{n}^{-1}k(n)_{*}(-)$ have this property, we deduce that we have exact comparison functors 
\[
\Comod_{\quiversteenrod} \rightarrow \Comod_{A_{*}}
\]
and 
\[
\Comod_{\quiversteenrod} \rightarrow \Comod_{K(n)_{*}K(n)}.
\]
These are compatible with the forgetful functors into $\pcat$-reps, and we deduce that there is an exact comparison functor 
\begin{equation}
\label{equation:comparison_functor_out_of_q_modules_for_connective_kn}
\Comod_{\quiversteenrod} \rightarrow \Comod_{\acat_{*}} \times_{\Vect_{Q_{n}}} \Rep(\pcat) \times_{\Vect_{\fieldp[v_{n}^{\pm 1}]}} \Comod_{K(n)_{*}K(n)}.
\end{equation}
encoding that the $\pcat$-rep
\[
k(n)_{*}X \leftrightarrows \Hrm_{*}(X, \fieldp)
\]
associated to a spectrum $X$ has an $\acat_{*}$-comodule structure on the right factor and an $K(n)_{*}K(n)$-comodule structure on the $v_{n}$-localization of the left factor
and that the composite $\pi \circ \delta_{v_n}$ is equal to $Q_n \in \acat$.
\end{remark}

However, we do \emph{not} believe that (\ref{equation:comparison_functor_out_of_q_modules_for_connective_kn}) is an equivalence (or at least, we do not believe that the corresponding functor into the pullback of derived categories is an equivalence, and that's what's truly relevant to the Adams spectral sequence, which concerns $\Ext$-groups). The reason is that there are further relations which objects of the pullback category have to satisfy to lift to a structure of a $\quiversteenrod$-comodule.  

For a simple example, suppose that $X$ has only even cells, so that 
\[
k(n)_{*}X \cong k(n)_{*} \otimes_{\BP_{*}} \BP_{*}X
\]
and similarly for $K(n)$- and $\fieldp$-homology. In this case, the coactions of $K(n)_{*}K(n)$ on $K(n)_{*}X$ and of $\acat_{*}$ on $\Hrm_{*}(X, \fieldp)$ are not independent from each other. In fact, they are both uniquely determined by the coaction of
\[
\sigma(n) := k(n)_{*} \otimes _{\BP_{*}} \BP_{*}\BP \otimes _{\BP_{*}} k(n)
\]
induced by the unit map $\Ss \rightarrow \BP$ applied to 
\[
k(n)_{*}X \rightarrow k(n)_{*}(\BP \otimes X) \cong \sigma(n) \otimes_{k(n)_{*}} k(n)_{*}X
\]
Similar relations necessarily hold in any $\quiversteenrod$-comodule whose $k(n)_{*}$-module term is free and concentrated in even degrees, and these relations become even more involved outside of this case. 

\subsection{Truncated Brown--Peterson spectra} 

Let $\BPn{n}$ denote a truncated Brown-Peterson spectrum with 
\[
\pi_{*} \BPn{n} \cong \bbZ_{(p)}[v_{1}, \ldots, v_{n}].
\]
The corresponding Adams spectral sequence has been the subject of extensive study, see \cite{mahowald1981bo, burklund2019boundaries, beaudry2021telescope, gonzalez2000vanishing, gonzalez_regular_complex}. By a result of Hahn and Wilson \cite{hahn2022redshift}, there exists a form of $\BPn{n}$ which admits an $\mathbb{E}_{3}$-algebra structure (at all primes and heights). It follows that the current work provides another approach to this spectral sequence based on a (braided monoidal, by \cref{remark:monoidality_of_quivers_depending_on_en_structure}) category of quiver representations. 

To get off the ground, we need to pick an appropriate class of simples $\pcat \subseteq \Mod_{\BPn{n}}^\omega$, large enough so that $\BPn{n}$ is $\pcat$-Adams-type. 
This naturally leads to the question 
of which $\BPn{n}$-modules show up in decompositions of tensor powers of $\BPn{n}$.
%easiest way to ensure this is to pick $\pcat$ so that 
%\begin{equation}
%\label{equation:bpn_cooperations_spectrum}
%\BPn{n} \otimes \BPn{n}
%\end{equation}
%is a filtered colimit of simple $\BPn{n}$-modules. 

We believe this is an interesting problem, but we will not try choose such a class of simples here. Instead, let us record that study of decompositions of tensor powers of $\BPn{n}$ has a long history in stable homotopy theory. At $p = 2$, $\BPn{1}$ is closely related to connective real $K$-theory $bo$, and a decomposition of $bo \otimes bo$ was constructed by Mahowald \cite{mahowald1981bo}. At $p > 2$, a decomposition of $\BPn{1} \otimes \BPn{1}$ was given by Kane \cite{kane1981operations}. At $n = 2$ and $p \geq 5$, we have a recent thesis of Tatum \cite{tatum2022spectrum}.

These decomposition results are a starting point to the study of the Adams spectral sequence based on the corresponding homology theory.
In \cite{mahowald1981bo, mahowald1982image}, Mahowald used the $bo$-Adams to prove the telescope conjecture at height $n = 1$ and $p = 2$. This work was subsequently extended in \cite{beaudry20202}. At odd primes, the $\BPn{1}$-Adams is studied in detail in the work of Gonzalez \cite{gonzalez_regular_complex, gonzalez2000vanishing}. At height two and $p = 2$, the closely related $\mathrm{tmf}$-resolutions were studied by B\`{e}audry, Behrens, Bhattacharya, Culver and Xu \cite{beaudry2019tmf, beaudry2021telescope}. 

We believe that it would be interesting to try to understand these results in terms of the quiver approach outlined in the current work. 

\subsection{United K-theory and algebraicity at small primes}

The quiver method presented in the current work has an important antecedent in Bousfield's \emph{united K-theory} \cite{bousfield1990classification}. In this short section, we briefly summarize Bousfield's work from our perspective, and suggest a possible generalization to higher heights. 

Let $\KO$, $\KU$ and $\KT$ denote, respectively, (periodic, topological) real $\K$-theory, complex $\K$-theory and self-conjugate $\K$-theory. In \cite[\S 1]{bousfield1990classification}, Bousfield considers "CRT-modules", which are triples of graded abelian groups 
\[
(M_{*}^{\KO}, M^{\KU}_{*}, M^{\KT}_{*})
\]
with a family of maps between those indexed by homotopy classes of maps of $\KO$-modules between these three spectra. It is not difficult to see that this notion is equivalent to that of a $\pcat$-rep, where $\pcat$ is the class of simple $\KO$-modules generated as an additive category by shifts of $\KO$, $\KU$ and $\KT$. 

The association 
\begin{equation}
\label{equation:bousfield_united_k_theory_functor}
X \in \spectra \mapsto \K_{*}^{\CRT}X := (\KO_{*}X, \KU_{*}X, \KT_{*}X)
\end{equation}
defines a homology theory valued in $\CRT$-modules, which can be identified with $\pcat$-homology as defined in the current work. Bousfield observes that $\K_{*}^{CRT}X$ has a compatible action by Adams operations, which he makes explicit \cite[\S 5]{bousfield1990classification}. 

Bousfield terms CRT-modules with Adams operations "ACRT-modules" and shows that if we consider (\ref{equation:bousfield_united_k_theory_functor}) as valued in ACRT-modules, then the resulting homology theory is adapted \cite[\S 9]{bousfield1990classification}. In the language of the current work, Bousfield's calculation is the determination of the comonad $\quiversteenrod (-)$ appearing in the adapted factorization of (\ref{equation:bousfield_united_k_theory_functor}). He then develops the corresponding Adams spectral sequence, universal coefficient theorem, and obstruction theory. 

Bousfield's ultimate goal is to classify homotopy types of $\KO$-local spectra, generalizing his classification of $\KU_{(p)}$-local spectra for $p > 2$ given in \cite{bousfield1985homotopy}. The main obstacle to extending this work to $p = 2$ is that the category of comodules over the cooperations algebra $\KU_{*}\KU$ has finite cohomological dimension $2$ when $p > 2$, but is of infinie dimension when $p = 2$ due to the existence of the $2$-torsion subgroup $C_{2} \subseteq \mathbb{Z}_{(2)}^{\times}$. 

Since $\KO \cong \KU^{h C_{2}}$, Bousfield's plan of attack is to encode homology groups with respect to the fixed points as part of the data, in the hope that this leads to more amenable homological algebra. This is indeed the case, as he shows that ACRT-modules of the form $\K_{*}^{CRT}X$ for $X$ a spectrum are of injective dimension at most $2$. He then applies this observation to prove a variety of realization results for spectra with prescribed $\K_{*}^{CRT}$-homology \cite[\S 10]{bousfield1990classification}.

Note that Bousfield's work at odd primes was generalized to higher chromatic heights in the thesis of the second author \cite{pstrkagowski2021chromatic} (later improved in joint work with Patchkoria \cite{patchkoria2021adams}) the main result being that the Bousfield localization of spectra at height $n$ Morava $E$-theory admits an algebraic description in terms of $E_{*}E$-comodules when $2p - 2 > n^{2}+n$. The strategy used to prove this fails at smaller primes due to the existence of $p$-torsion in the Morava stabilizer group $\mathbb{G}_{n}$. Keeping in mind Bousfield's work at $p=2$, it is natural to ask if it is possible to prove small prime partial algebraicity results for $E$-local spectra by using quivers indexed by $E^{h G}$, as $G$ runs through finite subgroup of the Morava stabilizer group.

\appendix
\label{sec:appendix} 

\section{Calculus of deformations} 
\label{sec:def}

In this appendix, which can be read independently from the body of the paper,
we collect some material which is complementary to the contents of \cite{patchkoria2021adams}.
In that paper Patchkoria and the second author constructed categories $\dcat^{\omega}_{\Hsf}(\ccat)$ which deform a category $\ccat$ along a homology theory $\Hsf$ in a way which categorifies the $\Hsf$-based Adams spectral sequence.

The unifying theme of this appendix is that whereas in \cite{patchkoria2021adams}, $\dcat^{\omega}_{\Hsf}(\ccat)$ was primarily studied \emph{internally} (that is, through the details of its specific construction), here we study $\dcat^{\omega}_{\Hsf}(\ccat)$ \emph{externally} (that is, through the relationships between the various deformations). Our aims are: 

\begin{enumerate}
\item In \S\ref{subsection:functoriality_of_deformations}, we develop naturality  of deformation $\infty$-categories and various related constructions, showing that they define functors on an appropriate $(\infty, 2)$-category of adapted homology theories, see \cref{construction:deformation_as_a_2functor} and 
\cref{construction:functors_between_deformations_from_universal_property}.
\item In \S\ref{subsection:induced_homology_theories}, we study what we call \emph{induced} homology theories, which are informally obtained by composing a homology theory with an exact functor of stable $\infty$-categories. The main result is \cref{proposition:flat_induction_yields_adjunction_of_adapted_homology_theories} which describes the target of an induced homology theory in terms of comodules. 
\item In \S\ref{subsection:local_and_colocal_objects}, we study (co)local objects for a comparison functor between different $\infty$-categories. This technical section is used in the main body of the paper to prove comparison results between various Adams spectral sequences, whose second pages can be identified with homotopy classes of maps in the appropriate deformations. 
\end{enumerate}

The appendix is written at a somewhat high level of generality with an eye towards future applications. To make it more understandable, before proceeding we give a brief reminder of some of the core results of \cite{patchkoria2021adams} about homology theories and the associated deformations.
 
\begin{recollection}[Epimorphism classes and homology theories]
\label{recollection:epimorphism_classes_and_homology_theories}
\hfill
 \begin{enumerate}
\item Given a stable category $\ccat$, an \emph{epimorphism class} consists of a collection of arrows $\ecat \subseteq \Fun(\Delta^{1}, \ccat)$ closed under composition, (de)suspensions, pullbacks along arbitrary morphisms and such that if $g \circ f$ is in $\ecat$ then $g$ is in $\ecat$. 
\item Given a stable category $\ccat$ and an epimorphism class $\ecat$ on $\ccat$, we say that an object $j \in \ccat$ is \emph{$\Hsf$-injective} if 
\[ 
\pi_{0} \Map_{\ccat}(d,j) \to \pi_{0} \Map_{\ccat}(c,j) 
\]
is injective for every $c \to d$ in $\ecat$. If every object $c \in \ccat$ admits a map $c \rightarrow j$ into an $\ecat$-injective such that $j \rightarrow \mathrm{cofib}(c \rightarrow j)$ is in $\ecat$, then we say that \emph{$\ecat$ has enough injectives}.
\item A \emph{homology theory} on a stable category $\ccat$ consists of a locally graded functor $\Hsf : \ccat \to \acat$ into a locally graded abelian category which sends cofiber sequences to sequences exact in the middle. Here, a \emph{local grading} on a category is an action of $\mathbb{Z}$, and we consider $\ccat$ as locally graded through its suspension functor. 
\item We say that a homology theory $\Hsf$ has \emph{lifts of injectives} if $\acat$ has enough injectives and for every injective $i \in \acat$ there exists an object $c_{i}$ representing the functor $(h\ccat)^{op} \rightarrow \Ab$
\[
c \mapsto \Hom_{\acat}(\Hsf(c), i).
\]
in the homotopy category. We say that $\Hsf$ is \emph{adapted} if moreover the canonical map $\Hsf(c_{i}) \rightarrow i$ is an isomorphism for any injective $i \in \acat$. 
\item Any homology theory $\Hsf$ determines an epimorphism class
\[
\ecat_{\Hsf} \colonequals \{ \ f \ | \ \Hsf(f) \text{ is an epimorphism } \}
\]
If $\ccat$ is idempotent complete\footnote{If $\ccat$ is stable, then any homology theory on $\ccat$ factors uniquely through its idempotent completion, so that for our purposes we can always assume that $\ccat$ is idempotent-complete}, this yields an equivalence of categories 
\[
\{ \text{ adapted homology theories on } \ccat \ \} \cong \{ \text{ epimorphism classes with enough injectives } \},
\]
where the left hand side is the full subcategory of $\Fun(\mathrm{B}\mathbb{Z}, \catinfty)_{\ccat /}$ spanned by adapted homology theories and the right hand side is the poset of epimorphism classes. In particular, the category of adapted homology theories is equivalent to a poset. 
\end{enumerate}
\end{recollection}

\begin{notation}
Taking into account the correspondence between homology theories and epimorphism classes, we will usually not distinguish notationally between the two. That is, if $\Hsf \colon \ccat \rightarrow \acat$ is an adapted homology theory, we will usually denote the corresponding epimorphism class also as $\Hsf$. Conversely, if $\Hsf$ is an epimorphism class, we will denote the target of the associated adapted homology theory by $\acat(\ccat, \Hsf)$. 
\end{notation}

\begin{recollection}[The deformation associated to a homology theory]
\label{recollection:perfect_deformation_associated_to_a_homology_theory}
\hfill
\begin{enumerate}
\item Associated to an adapted homology theory $\Hsf \colon \ccat \rightarrow \acat$ there is a prestable category $\dcat^{\omega}_{\Hsf}(\ccat)_{\geq 0}$, the \emph{perfect derived category}.
\item There is a canonical equivalence $\dcat^{\omega}_{\Hsf}(\ccat)^{\heartsuit} \cong \acat$.
\item The perfect derived category has a canonical local grading 
\[
(-)[1] \colon \dcat^{\omega}_{\Hsf}(\ccat)_{\geq 0} \rightarrow \dcat^{\omega}_{\Hsf}(\ccat)_{\geq 0}
\]
Moreover, there is a canonical natural transformation
\[
\tau \colon \Sigma(-) \rightarrow (-)[1]
\]
of endofunctors of the perfect derived category. 
\item There is a fully faithful locally graded functor $\nu \colon \ccat \rightarrow \dcat^{\omega}_{\Hsf}(\ccat)_{\geq 0}$, where we consider $\ccat$ as locally graded using the suspension functor and a natural isomorphism
\[
\pi_{0}^{\heartsuit}(\nu(-)) \cong \Hsf(-).
\]
\item We say that an object $X \in \dcat_{\Hsf}^{\omega}(\ccat)_{\geq 0}$ is \emph{$\tau$-local} if $\tau \colon X \to \Omega X(1)$ is an equivalence.
\item The essential image of $\nu \colon \ccat \rightarrow \dcat^{\omega}_{\Hsf}(\ccat)_{\geq 0}$ is the subcategory spanned by $\tau$-local objects. 
\item The endofunctor $C\tau(X) \colonequals \mathrm{cofib}(\tau \colon \Sigma X[-1] \rightarrow X)$ admits a canonical structure of a monad on the perfect derived category and identification between the hearts extends to an equivalence
\[
\Mod_{C\tau}(\dcat^{\omega}_{\Hsf}(\ccat)_{\geq 0}) \cong \dcat^{b}(\acat)_{\geq 0}
\]
between the category of modules over this monad and the connective bounded derived category of $\acat$. 
\item Let $\dcat^{\omega}_{\Hsf}(\ccat)$ denote the Spanier--Whitehead stabilization of $\dcat^{\omega}_{\Hsf}(\ccat)_{\geq 0}$.
Then for any pair of objects $c, d \in \ccat$, the spectral sequence associated to the filtered spectrum 
\[
\ldots \rightarrow \map_{\dcat^{\omega}_{\Hsf}(\Hsf)}(d, \Sigma \nu(c)[-1]) \rightarrow \map_{\dcat^{\omega}_{\Hsf}(\Hsf)}(d, \nu(c)) \rightarrow \map_{\dcat^{\omega}_{\Hsf}(\Hsf)}(d, \Sigma^{-1} \nu(c)[1]) \rightarrow \ldots
\]
can be identified with the $\Hsf$-based Adams spectral sequence which is of signature 
\[
\Ext_{\acat}^{s, t}(\Hsf(d), \Hsf(c)) \Rightarrow \pi_{t-s} \map_{\ccat}(d, c).
\]
\item Given a commutative diagram of locally graded categories\footnote{Note that since adapted homology theories on $\ccat$ form a poset by \cref{recollection:perfect_deformation_associated_to_a_homology_theory}, a locally graded functor $q$ making this diagram commute is unique when it exists.}
\[ \begin{tikzcd}
	& \ccat \ar[dl, "{\Hsf_{1}}"'] \ar[dr, "{\Hsf_{2}}"] & \\ 
    {\acat_{1}} \ar[rr, "q"] & & {\acat_{2}}
\end{tikzcd}\]
where $\Hsf_{1}, \Hsf_{2}$ are adapted homology theories, the associated comparison functor 
\[  \dcat_{\Hsf_1}^{\omega}(\ccat)_{\geq 0} \to \dcat_{\Hsf_2}^{\omega}(\ccat)_{\geq 0} \]
is a localization. 
\end{enumerate}
\end{recollection}

\subsection{Functoriality of deformations} 
\label{subsection:functoriality_of_deformations}

We begin with a discussion of the $(\infty, 2)$-categorical functoriality of the construction of the deformation associated to an adapted homology theory.

\begin{construction} \label{dfn:adapted}
Let $\adapted$ denote the $(\infty, 2)$-category with
    \begin{itemize}
        \item objects given by pairs $(\ccat, \Hsf)$ of an idempotent-complete stable category $\ccat$ and an epimorphism class $\Hsf$ on $\ccat$ with enough injectives and
        \item mapping categories are given by the full subcategory
        \[ \Map_{\adapted}((\ccat_1, \Hsf_1), (\ccat_2, \Hsf_2)) \subseteq \Fun(\ccat_1,\ccat_2) \] 
        spanned by the exact functors that send $\Hsf_1$-epis to $\Hsf_2$-epis.
    \end{itemize}
\end{construction}

\begin{remark}
By \cref{recollection:perfect_deformation_associated_to_a_homology_theory}, the fiber of the forgetful functor $\iota : \adapted \to \stableinftycats $ over $\ccat$ is a \emph{poset}. Specifically, it is the poset of epimorphism classes with enough injectives. 
\end{remark}

\begin{lemma}
\label{lemma:extending_exact_functors_to_squares_of_homology_theories}
Let $\Hsf: \ccat \rightarrow \acat$ and $\Hsf': \ccat \rightarrow \acat'$ be adapted homology theories. Then an exact functor  $f: \ccat \rightarrow \ccat'$ can be completed to a commutative square 
\begin{equation}
\label{equation:commuting_square_of_adapted_homology_theories}
\begin{tikzcd}
	\ccat & {\ccat'} \\
	\acat & {\acat'}
	\arrow["\Hsf", from=1-1, to=2-1]
	\arrow["{\Hsf'}", from=1-2, to=2-2]
	\arrow["f", from=1-1, to=1-2]
	\arrow["\overline{f}"', from=2-1, to=2-2]
\end{tikzcd}
\end{equation}
with $\overline{f}$ an exact, locally graded functor of abelian categories, if and only if $f$ takes $\Hsf$-epimorphisms to $\Hsf'$-epimorphisms. If the latter is the case, then $f$ can be completed to such a square in a unique way. 
\end{lemma}

\begin{proof}
By the universal property of the Freyd envelope \cite[Theorem 2.45]{patchkoria2021adams}, to complete $f$ to such a commutative square is the same as to complete the induced functor $A(f)$ between Freyd envelopes to a commutative square 
\[
\begin{tikzcd}
	A(\ccat) & A({\ccat'}) \\
	\acat & {\acat'}
	\arrow["\Hsf", from=1-1, to=2-1]
	\arrow["{\Hsf'}", from=1-2, to=2-2]
	\arrow["A(F)", from=1-1, to=1-2]
	\arrow["\overline{f}"', from=2-1, to=2-2]
\end{tikzcd}.
\]
Since both vertical arrows are localizations by \cite[Theorem 2.56]{patchkoria2021adams}, this can be done precisely when $A(f)(\mathrm{ker}(\Hsf)) \subseteq \mathrm{ker}(\Hsf')$, in which case there is a unique way to do so. As this condition on kernels is equivalent to $f$ taking $\Hsf$-epimorphisms to $\Hsf'$-epimorphisms, the result follows.
\end{proof}

\begin{remark}
\label{remark:functors_of_adapted_homology_theories_in_terms_of_preserving_epis}
As a consequence of \cref{lemma:extending_exact_functors_to_squares_of_homology_theories} and \cite[Theorem 3.24]{patchkoria2021adams} the $(\infty, 2)$-category $\adapted$ is equivalent to the $(\infty, 2)$-category 
whose objects are idempotent-complete stable categories equipped with an adapted homology theory 
and whose morphisms are commuting squares as in (\ref{equation:commuting_square_of_adapted_homology_theories}).
\end{remark}

\begin{remark}
    In view of \Cref{remark:functors_of_adapted_homology_theories_in_terms_of_preserving_epis} we will freely move between the presentation of $\adapted$ in terms of epimorphism classes and the presentation in terms of adapted homology theories.
\end{remark}

\begin{construction} \label{construction:def}
    Let us write $\pcat\mathrm{St}^{\mathrm{ex}}$ for the $(\infty, 2)$-category of prestable categories with finite limits and exact functors. Given an adapted homology theory $(\ccat, \Hsf) \in \adapted$ let 
    \[ \Def_{\Hsf}^{\omega}(\ccat)_{\geq 0} : \pcat\mathrm{St}^{\mathrm{ex}} \to \mathrm{Cat}_{\infty} \]
    denote the $2$-presheaf of categories on $(\pcat\mathrm{St}^{\mathrm{ex}})^{\mathrm{op}}$
    which sends a prestable category $\dcat_{\geq 0}$ to the full subcategory \[
    \Def_{\Hsf}^{\omega}(\ccat)_{\geq 0}(\dcat_{\geq 0}) \subseteq \Fun(\ccat, \dcat_{\geq 0})
    \]
    spanned by the left exact functors which preserve $\Hsf$-exact cofiber sequences.
\end{construction}

\begin{remark}
\label{remark:preserving_h_exact_cofiber_sequences_same_as_being_a_prestable_enhancement}
As observed in \cite[Proof of Theorem 5.35]{patchkoria2021adams}, a left exact functor $f \colon \ccat \rightarrow \dcat_{\geq 0}$ sends $\Hsf$-exact cofiber sequences to cofiber sequences if and only if the homological functor $\tau_{\leq 0} \circ f \colon \ccat \rightarrow \dcat^{\heartsuit}$ factors (necessarily uniquely) through $\acat$. Thus, $f$ preserves $\Hsf$-exact cofiber sequences if and only if it can be made into a prestable enhancement of $\Hsf$ in the sense of \cite[Notation 4.41]{patchkoria2021adams}. 
\end{remark}

\begin{proposition} \label{prop:universal-property}
    Precomposition with $\nu_{\Hsf} : \ccat \to \dcat_{\Hsf}(\ccat)_{\geq 0}$ induces an equivalence 
    \[ \Def_{\Hsf}^{\omega}(\ccat)_{\geq 0}(-) \cong \Map_{\pcat\mathrm{St}^{\mathrm{ex}}}(\dcat^{\omega}_{\Hsf}(\ccat)_{\geq 0}, -) \]
\end{proposition}

\begin{proof}
Note that $\nu_{\Hsf}$ preserves $\Hsf$-exact cofiber sequences by \cite[Proposition 5.34]{patchkoria2021adams}, which implies that precomposition with $\nu_{\Hsf}$ induces a natural transformation
    \[ \Map_{\pcat\mathrm{St}^{\mathrm{ex}}}(\dcat^{\omega}_{\Hsf}(\ccat)_{\geq 0}, -) \to \Def_{\Hsf}^{\omega}(\ccat)(-)_{\geq 0} \]
    Keeping in mind \cref{remark:preserving_h_exact_cofiber_sequences_same_as_being_a_prestable_enhancement}, to say that this is an natural equivalence is a restatement of the universal property of $\dcat^{\omega}(\ccat)$ of \cite[Theorem 5.35]{patchkoria2021adams}.
\end{proof}

\begin{remark}
\label{remark:right_exact_variant_of_the_universal_property_of_def}
The universal property of \cref{prop:universal-property} has a right exact variant which we now describe. Let $\pcat\mathrm{St}^{\mathrm{rex}}$ denote the $(\infty, 2)$-category of prestable categories and right exact functors. Given an adapted homology theory $(\ccat, \Hsf) \in \adapted$ let 
\[ 
\Def_{\Hsf}^{\omega, \mathrm{rex}}(\ccat)_{\geq 0} \colon \pcat\mathrm{St}^{\mathrm{rex}} \to \mathrm{Cat}_{\infty}
\]
denote the functor which sends a prestable category $\dcat_{\geq 0}$ to the full subcategory 
\[
\Def_{\Hsf}^{\omega, \mathrm{rex}}(\ccat)_{\geq 0}(\dcat_{\geq 0}) \subseteq \Fun(\ccat, \dcat_{\geq 0})
\]
spanned by additive functors which preserve $\Hsf$-exact cofiber sequences. Then precomposition with $\nu_{\Hsf} : \ccat \to \dcat_{\Hsf}(\ccat)_{\geq 0}$ induces an equivalence 
\[ 
\Def_{\Hsf}^{\omega, \mathrm{rex}}(\ccat)_{\geq 0}(-) \cong \Map_{\pcat\mathrm{St}^{\mathrm{rex}}}(\dcat^{\omega}_{\Hsf}(\ccat)_{\geq 0}, -).
\]
This is immediate from the definition of $\dcat^{\omega}_{\Hsf}(\ccat)_{\geq 0}$ as a localization of the additive finite colimit completion $A^{\omega}_{\infty}(\ccat)$ of $\ccat$ studied in \cite[\S 4.2]{patchkoria2021adams} along $\Hsf$-exact cofiber sequences. 
\end{remark}

\begin{construction}
\label{construction:deformation_as_a_2functor}
    As a consequence of \Cref{prop:universal-property} $\Def_{\Hsf}^{\omega}(\ccat)_{\geq 0}(-)$ 
    is a representable $2$-presheaf of categories.  
    Allowing the choice of adapted homology theory to vary in \Cref{construction:def} 
    we obtain a $2$-functor of $(\infty,2)$-categories
    \[ \Def_{-}^{\omega}(-)_{\geq 0} \colon \adapted \to \pcat\mathrm{St}^{\mathrm{ex}}. \]
    If $f \colon (\ccat_1, \Hsf_1) \rightarrow (\ccat_2, \Hsf_2)$ is a morphism of adapted homology theories, we denote the induced exact functor of prestable categories by $\widetilde{f} \colon \Def_{\Hsf_1}^{\omega}(\ccat_1)_{\geq 0} \rightarrow \Def_{\Hsf_2}^{\omega}(\ccat_2)_{\geq 0}$.
\end{construction}

\begin{construction}
\label{construction:functors_between_deformations_from_universal_property}
Using the presentation of $\Def_{-}^{\omega}(-)_{\geq 0}$ as a $2$-presheaf of categories on $\pcat\mathrm{St}^{\mathrm{ex}}$ (and its right exact variant of \cref{remark:right_exact_variant_of_the_universal_property_of_def}) we can construct several natural transformations out of it. First, we have the two forgetful functors: 
\begin{enumerate}
    \item Let $\iota(-)$ denote the forgetful functor that sends 
    $(\ccat, \Hsf)$ to $\ccat$.
    \item  Let $\acat(-)$ denote the functor that sends 
    $(\ccat, \Hsf)$ to $\Def_{\Hsf}^{\omega}(\ccat)^{\heartsuit}$.
\end{enumerate}
We then also have the canonical maps into the generic and special fibers: 
\begin{enumerate}[start=3]
\item Let $(-)^{\tau=1}$ denote the natural transformation 
    \[(-)^{\tau=1} : \Def_{-}^{\omega}(-) \Rightarrow \iota(-)\]
    of \cite[Remark 5.47]{patchkoria2021adams} induced by the exact functor $\ccat \xrightarrow{\mathrm{Id}} \ccat$. 
    \item Let $(-)^{\tau=0}$ denote the natural transformation 
    \[ (-)^{\tau=0} \colon \Def_{-}^{\omega}(-) \to \dcat^b(\acat(-)) \]
    obtained by applying the universal property of \cref{remark:right_exact_variant_of_the_universal_property_of_def} to the composite 
    \[
    \iota(-) \rightarrow \acat(-) \rightarrow \dcat^{b}(\acat(-)),
    \]
    where the first arrow is the homology theory and the second one is the inclusion of the heart. This is left adjoint to the exact functor 
    \[
    \dcat^b(\acat(-)) \rightarrow \Def_{-}^{\omega}(-)
    \]
    which exists by the universal property of the bounded derived category. The resulting adjunction is monadic and determines the monad $C\tau \otimes -$ of \cite[\S 5.3]{patchkoria2021adams}.
\end{enumerate}
Finally, we have the local grading and the ``thread structure'' of \cite[\S 5.3]{patchkoria2021adams}:
\begin{enumerate}[start=5]
    \item The identity functor on $\adapted$ has an automorphism given by suspension 
    $\Sigma  \colon \ccat \to \ccat$.
    This equips $\Def_{\Hsf}^{\omega}(\ccat)$ with a local grading natural in $(\ccat, \Hsf)$ which we denote by $(-)(1)$. 
    \item The canonical map $S^1 \to \Map_{\Fun(\ccat, \ccat)}(\mathrm{Id}, \Sigma)$ induces a natural transformation
    \[\tau \colon \mathrm{Id} \Rightarrow \Omega \mathrm{Id} (1) \]
    of functors 
    \[ \Def_{\Hsf}^{\omega}(\ccat)_{\geq 0} \to \Def_{\Hsf}^{\omega}(\ccat)_{\geq 0}. \]
\end{enumerate}
\end{construction}

Giving a careful constructing of the natural transformation $\nu$ with all the desired functoriality is more delicate and requires a preparatory lemma.

\begin{lemma} \label{lem:tau-local-preserved}
If $f \colon (\ccat_1, \Hsf_1) \rightarrow (\ccat_2, \Hsf_2)$ is a morphism of adapted homology theories, then
\[
\widetilde{f} \colon \Def_{\Hsf_1}^{\omega}(\ccat_1) \rightarrow \Def_{\Hsf_2}^{\omega}(\ccat_2)
\]
sends $\tau$-local objects to $\tau$-local objects.
\end{lemma}

\begin{proof}
This follows from the naturality of $\tau$ and $(-)(1)$ in the adapted homology theory together with the fact that $\widetilde{f}$ is exact.
\end{proof}

\begin{construction}
\label{construction:2_functorial_nu_using_universal_properties}
Using \Cref{lem:tau-local-preserved}, we can construct a $2$-functor
\[
\Def_{-}^{\omega}(-)_{\geq 0}^{\tau=1} : \adapted \to \mathrm{Cat}_{\infty}
\]
that sends $(\ccat, \Hsf)$ to the subcategory of $\tau$-local objects in $\Def_{\Hsf}^{\omega}(\ccat)_{\geq 0}$ and an associated natural transformation
\[
\Def_{-}^{\omega}(-)_{\geq 0}^{\tau=1} \Rightarrow \Def_{-}^{\omega}(-)_{\geq 0}
\]
which is pointwise fully faithful. From \cite[Example 6.25]{patchkoria2021adams} (where $\tau$-local objects are referred to as ``potential $\infty$-stages'') we can read off that the composite
\[ 
\Def_{-}^{\omega}(-)_{\geq 0}^{\tau=1} \Rightarrow \Def_{-}^{\omega}(-)_{\geq 0} \xRightarrow{(-)^{\tau=1}} \iota(-) 
\]
is an equivalence. We let $\nu$ denote the natural transformation
\[ 
\nu : \iota(-) \Rightarrow \Def_{-}^{\omega}(-)_{\geq 0} 
\]
obtained by composing the inverse of the natural equivalence above with the inclusion of the $\tau$-local objects. 
\end{construction}

Altogether we obtain the following diagram of natural transformations of $2$-functors
\begin{equation}
    \begin{tikzcd}[sep=huge] 
        & \iota(-) \ar[r, "\Hsf"] \ar[d, hook, "\nu"] \ar[dl, "\mathrm{Id}"] & \acat(-) \ar[d, hook] \\
        \iota(-) & \dcat^{\omega}_{-}(-) \ar[l, "(-)^{\tau=1}"] \ar[r, "(-)^{\tau=0}"'] & \dcat^b(\acat(-)).
    \end{tikzcd}
\end{equation}

\begin{variant}
    Let $\adapted^{\mathrm{prl}}$ be the subcategory of $\adapted$ 
    whose objects are the pairs $(\ccat, \Hsf)$ 
    with $\ccat$ presentable and $\Hsf$ an adapted \emph{Grothendeick} homology theory
    and whose morphisms are the left adjoints which preserve $\Hsf$-epis.
    
    In parallel to the contents of this section the construction of the 
    unseperated derived category from \cite[Theorem 6.40]{patchkoria2021adams} extends to provide a $2$-functor
    \[
    \widecheck{\Def}_{\Hsf}(\ccat)_{\geq 0} \colon \adapted^{\mathrm{prl}} \rightarrow \pcat\mathrm{St}^{\mathrm{L}, \mathrm{ex}}
    \]
    into the $(\infty, 2)$-category of Grothendieck prestable categories and exact left adjoints. Moreover, we have the associated natural transformations $\nu$, $\tau$, $(-)^{\tau=1}$ and  $(-)^{\tau=0}$ as in \cref{construction:functors_between_deformations_from_universal_property} and \cref{construction:2_functorial_nu_using_universal_properties}, all of which preserve filtered colimits. 
\end{variant}

Recall from \cite[6.10.1]{krause_2010} that if $\Hsf \colon \ccat \rightarrow \acat$ is Grothendieck and $\ccat$ is compactly-generated, then the left Kan extension of $\Hsf$ along the restricted Yoneda embedding 
\[
(y_{0} \colonequals \pi_{0} \circ y) \colon \ccat \rightarrow P_{\Sigma}(\ccat^{\omega}; \Ab)
\]
induces a \emph{localization} 
\[
L_{\Hsf} \colon P_{\Sigma}(\ccat^{\omega}; \Ab) \rightarrow \acat.
\]
Moreover, any Grothendieck homology theory on $\ccat$ arises in this way. The following statement characterizes under which conditions the associated deformation is also compactly-generated. 

\begin{lemma}
\label{lemma:unseparated_def_of_cpt_gen_homology_theory_is_cpt_gen}
Let $\Hsf : \ccat \to \acat$ be a Grothendieck adapted homology theory
with $\ccat$ compactly generated. The following conditions are equivalent: 
\begin{enumerate}
\item $\acat$ is compactly generated and $\Hsf$ preserves compactness,
\item the kernel $K \colonequals \mathrm{ker}(L_{\Hsf})$ of the localization $P_{\Sigma}(\ccat^{\omega}; \Ab) \to \acat$ is generated as a localizing subcategory by cokernels of $\Hsf$-epimorphisms $y_{0}(c) \to y_{0}(d)$, where $c \rightarrow d$ is a map of compact objects of $\ccat$, 
\item $\widecheck{\Def}_{\Hsf}(\ccat)$ is compactly generated and $\nu_{\Hsf}$ preserves compactness.
\end{enumerate}
\end{lemma}

\begin{proof}
$(1 \Rightarrow 2)$: Since $P_{\Sigma}(\ccat^{\omega}; \abeliangroups)$ is generated under colimits by representables, any object $k \in K$ can be written as the cokernel $\coker(y_0(a) \to y_0(b))$, where both $a, b \in \ccat$ are direct sums of compact objects.  We now need to approximate $k$ by objects of the form required by (2).    

Write $b \cong \oplus b_{\beta}$ as a direct sum of compact objects $b_{\beta}$.
    Let $a_{\beta}$ be defined by the pullback square
    \[ \begin{tikzcd}
    a_{\beta} \ar[r] \ar[d] & b_\beta \ar[d] \\
    a \ar[r] & b 
    \end{tikzcd}, \]
    where the top horizontal arrow is an $\Hsf$-epi as a pullback. If we write
    \[
    k_\beta \coloneqq \coker(y_0(a_\beta) \to y_0(b_\beta)),
    \]
    then since $\bigoplus y_{0}(b_{\beta}) \rightarrow y_{0}(b)$ is an isomorphism, $\oplus k_{\beta} \rightarrow k$ is an epimorphism.  Since localizing subcategories are closed under direct sums, replacing $a \rightarrow b$ with $a_{\beta} \rightarrow b_{\beta}$, we reduce to the case where $a \rightarrow b$ is an $\Hsf$-epimorphism with $b$ compact. 

    We can now write $a \cong \varinjlim a_{\alpha}$ as a filtered colimit of compact objects. Since $\Hsf(b)$ is compact by assumption and 
    \[
    \varinjlim \Hsf(a_{\alpha}) \cong \Hsf(a) \rightarrow \Hsf(b)
    \]
    is an epimorphism, there exists an $\alpha$ such that $\Hsf(a_{\alpha}) \rightarrow \Hsf(b)$ is also an $\Hsf$-epimorphism. As 
    \[
    \mathrm{coker}(y_{0}(a_{\alpha}) \rightarrow y_{0}(b)) \rightarrow k
    \]
    is an epimorphism, we see that $k$ is a quotient of an object of the required form. 
    
    $(2 \Rightarrow 3)$:  We begin by recalling from \cite[Proposition 6.48]{patchkoria2021adams} that the unseparated deformation can be explicitly described as the Gabriel quotient
    \[
    \widecheck{\Def}_{\Hsf}(\ccat)_{\geq 0} \cong P_{\Sigma}(\ccat^{\omega}) / \kcat
    \]
    where $\kcat$ is the localizing subcategory (in the sense of Grothendieck prestable categories) generated by $K \subseteq P_{\Sigma}(\ccat^{\omega})^{\heartsuit} \subseteq P_{\Sigma}(\ccat^{\omega})$.
    The key point is now that
    \[ \coker(y_{0}(c) \to y_{0}(d)) \cong \cof(\Sigma y(e) \to \cof( y(c) \to y(d) )) \]
    where $e \coloneqq \fib(c \to d)$.
    From this we can read off using (2) that $\kcat$ is compactly generated and the inclusion
    $\kcat \to P_{\Sigma}(\ccat^{\omega})$ preserves compactness.
    It follows that $P_{\Sigma}(\ccat^{\omega})/\kcat$ is compactly generated and the localization $P_{\Sigma}(\ccat^{\omega}) \to P_{\Sigma}(\ccat^{\omega})/\kcat$
    preserves compactness. Writing $\nu_{\Hsf}$ as the composite
    \[ \ccat \xrightarrow{y} P_{\Sigma}(\ccat^{\omega}) \to P_{\Sigma}(\ccat^{\omega})/\kcat \]
    we see that $\nu_{\Hsf}$ preserves compactness as well.

    $(3 \Rightarrow 1)$: The right adjoint to the functor $\pi_0 \colon \widecheck{\Def}_{\Hsf}(\ccat)_{\geq 0} \to \acat$ preserves filtered colimits since $\widecheck{\Def}_{\Hsf}$ is Grothendieck prestable. In particular, $\pi_0$ preserves compact objects. Precomposing with $\nu$ we see that $\Hsf$ preserves compactness. Finally, since $\acat$ is generated under colimits by the image of $\Hsf$ and $\Hsf$ preserves compact objects we see that $\acat$ is compactly generated.
\end{proof}

\subsection{Induced homology theories}
\label{subsection:induced_homology_theories}

In this section we study a simple, but useful, technique for producing new homology theories from old: pushforward along an adjunction.

\begin{convention}
We will assume that all stable categories considered in this section are idempotent-complete. Note that since any abelian category is idempotent-complete, precomposing with $\ccat \rightarrow \mathrm{Idem}(\ccat)$ yields an equivalence between homology theories on $\ccat$ and its idempotent completion, so this convention is not really restrictive. 
\end{convention}

\begin{construction}
\label{construction:class_of_epis_induced_by_an_adjunction}
Let $f^* : \ccat \rightleftharpoons \dcat : f_*$ be an adjunction between stable categories and let $\Hsf \colon \dcat \rightarrow \acat$ be an adapted homology theory. Let us say that an arrow $c_{0} \rightarrow c_{1}$ is an $\Hsf'$-epi if $f^{*}(c_{0}) \rightarrow f^{*}(c_{1})$ is $\Hsf$-epi. It is not difficult to see that $\Hsf'$-epis form an epimorphism class on $\ccat$. 
\end{construction}

\begin{lemma} \label{lem:cnstr-induced-homology}
    The epimorphism class $\Hsf'$ of \cref{construction:class_of_epis_induced_by_an_adjunction}
    has enough $\Hsf'$-injectives 
    and therefore uniquely determines an adapted homology theory 
    $\Hsf' : \ccat \to \acat'$.
    Moreover, $c \in \ccat$ is $\Hsf'$-injective if and only if it is a retract of an object of the form $f_*d$ where $d$ is $\Hsf$-injective.
\end{lemma}

\begin{proof}
    We start by proving that if $d$ is $\Hsf$-injective, then $f_*d$ is $\Hsf'$-injective.
    Given an $\Hsf'$-mono $c_1 \to c_2$ we have a commutative diagram
    \[ \begin{tikzcd}
        {[c_2, f_*d]} \ar[r] \ar[d, no head, "\cong"] & {[c_1, f_*d]} \ar[d, no head, "\cong"] \\
        {[f^*c_2, d]} \ar[r, two heads] & {[f^*c_1, d]} 
    \end{tikzcd} \]
    from which we can read off that $f_*d$ is $\Hsf'$-injective.
    
    Next we show that $\ccat$ has enough $\Hsf'$-injectives.
    Let $c \in \ccat$ and pick an $\Hsf$-mono $f^{*}c \rightarrow i$ into an $\Hsf$-injective. 
    The mate of this map, $c \rightarrow f_{*}i$, 
    is an $\Hsf'$-mono into an $\Hsf'$-injective since the composite
    $f^{*}c \rightarrow f^{*}f_{*}i \to i$
    is $\Hsf$-mono.
    
    To complete the proof we observe that
    given an $\Hsf'$-injective $c$ the construction above provides us with an $\Hsf'$-monomorphism $c \to f_*d$ where $d$ is $\Hsf$-injecitve
    and that $c$ is a retract of $f_*d$ since every $\Hsf'$-mono out of $\Hsf'$-injective admits a retraction.
\end{proof}

\begin{definition}
\label{definition:homology_theory_induced_by_an_adjunction}
We call the adapted homology theory $\Hsf' \colon \ccat \rightarrow \acat'$ determined by the class of $\Hsf'$-epimorphisms the \emph{the pushforward of $\Hsf$ along $f$}.
\end{definition}

\begin{example}
\label{example:monadic_adams_sseq_using_homology_induction}
Let $f^* \colon \ccat \rightleftharpoons \dcat \colon f_*$ be an adjunction between stable categories. The universal homology theory $y: \dcat \rightarrow A(\dcat)$ valued in the Freyd envelope of $\dcat$ is adapted. In this case  \cref{definition:homology_theory_induced_by_an_adjunction} yields an adapted homology theory on $\ccat$ which categorifies the monadic descent spectral sequence based on $f_{*} f^{*}(-)$, see \cite[Example 3.17]{patchkoria2021adams}.
\end{example}

\begin{example}
Specializing \cref{example:monadic_adams_sseq_using_homology_induction} 
to the free-forgetful adjunction 
$\spectra \rightleftharpoons \Mod_{R}(\spectra)$
associated to modules over an $\mathbb{E}_{1}$-ring spectrum $R$
we recover the Adams spectral sequence based on the Amitsur $R$-resolution, as in see \cite[Example 3.18]{patchkoria2021adams}.
\end{example}

\begin{remark} \label{remark:left-comonad}
It follows from \cite[Theorem 3.27]{patchkoria2021adams} that the pushforward homology theory 
    $\Hsf' \colon \ccat \to \acat'$ can be identified with the adapted factorization of the composite $\Hsf \circ f^*$.
    In particular, by \cite[Theorem A.1]{patchkoria2021adams} the functor $\overline{f^{*}} \colon \acat' \rightarrow \acat$ is an exact, comonadic left adjoint and this identifies $\acat'$ as the category of comodules for a left exact comonad on $\acat$.
\end{remark}

\begin{warning}
\label{warning:left_exact_comonads_not_too_useful}
The description of $\acat'$ as comodules for a left exact comonad of \Cref{remark:left-comonad} is not as useful in practice as one might hope, as in general neither the right adjoint nor the comonad are exact. In particular, the right adjoint to $\overline{f^{*}}$ is usually not $\overline{f_{*}}$ in the sense of \cref{lemma:extending_exact_functors_to_squares_of_homology_theories}; in fact, the latter might not exist as $f_{*} \colon \dcat \rightarrow \ccat$ need not take $\Hsf$-epimorphisms to $\Hsf'$-epimorphisms.
\end{warning}

We will now describe a common situation where the induced homology theory is particularly well-behaved and the resulting comonad is exact, avoiding pathological behaviour of \cref{warning:left_exact_comonads_not_too_useful}.

\begin{definition}
\label{definition:flat_adjunction}
Let $f^* \colon \ccat \rightleftharpoons \dcat \colon f_*$ be an adjunction between stable categories and let $\Hsf \colon \dcat \rightarrow \acat$ be an adapted homology theory. We say that $f^{*} \dashv f_{*}$ is \emph{$\Hsf$-flat} if $f^*f_* \colon \dcat \to \dcat$ preserves $H$-epimorphisms. 
\end{definition}

For an explanation of the origin of the terminology ``flat'' in this setting see \Cref{rmk:why-flat}.

\begin{proposition}
\label{proposition:flat_induction_yields_adjunction_of_adapted_homology_theories}
Let $f^{*} \colon \ccat \rightleftharpoons \dcat \colon f_{*}$ be an adjunction of stable categories, $\Hsf \colon \dcat \rightarrow \acat$ an adapted homology theory such that $f$ is $\Hsf$-flat and $\Hsf' \colon \ccat \rightarrow \acat'$ the induced homology theory of \cref{definition:homology_theory_induced_by_an_adjunction}. Then
\begin{enumerate}
    \item $f_{*}$ takes $\Hsf$-epimorphisms to $\Hsf'$-epimorphisms so that $f^{*} \dashv f_{*}$ induces an adjunction 
\[
(\Hsf' \colon \ccat \rightarrow \acat') \rightleftharpoons (\Hsf \colon \dcat \rightarrow \acat)
\]
in the $(\infty, 2)$-category $\adapted$ of adapted homology theories,
\item the comonad $f^{*} f_{*}$ descends uniquely to an exact comonad $Q := \overline{f}^{*} \overline{f}_{*}$ on $\acat$,
\item there is a canonical equivalence $\Comod_{Q}(\acat) \cong \acat'$ which identifies $\overline{f}^{*}$ with the forgetful functor $\Comod_{Q}(\acat) \rightarrow \acat$ and $\overline{f}_{*}$ with the cofree comodule functor $\acat \rightarrow \Comod_{Q}(\acat)$ . 
\end{enumerate}
\end{proposition}

\begin{proof}
Let $d_{1} \rightarrow d_{2}$ be an $H$-epimorphism in $\dcat$. Since $f^{*} \dashv f_{*}$ is $\Hsf$-flat, $f^{*} f_{*} d_{1} \rightarrow f^{*} f_{*} d_{2}$ is $\Hsf$-epi, so that $f^{*} d_{1} \rightarrow f^{*} d_{2}$ is an $\Hsf'$-epi by definition, proving (1). 
By \cref{remark:functors_of_adapted_homology_theories_in_terms_of_preserving_epis} this is enough to lift the adjunction $f^{*} \dashv f_{*}$ to an adjunction of adapted homology theories. 

Since $Q \colonequals \overline{f}^{*} \overline{f}_{*}$ is the composition of a left and right adjoint, both of which are exact, it has a canonical structure of an exact comonad. Moreover, since $\overline{f^{*}} \colon \acat' \rightarrow \acat$ is a conservative exact left adjoint, it is comonadic by \cite[Theorem A.1]{patchkoria2021adams}, so that the left and right adjoint can be identified with the forgetful and cofree functors. 
\end{proof}

We now verify that the adjunction obtained from \cref{proposition:flat_induction_yields_adjunction_of_adapted_homology_theories} is functorial in the epimorphism class. 

\begin{lemma}
\label{lemma:morphism_of_adjunctions_induced_by_map_of_flat_homology_theories}
Let $f^* \colon \ccat_1 \leftrightharpoons \ccat_2 \colon f_*$ be an adjunction between stable categories and let $q \colon \Hsf_{2a} \to \Hsf_{2b}$ a map of adapted homology theories on $\ccat_2$ such that $f$ is both $\Hsf_{2a}$-flat and $\Hsf_{2b}$-flat.
  Let $\Hsf_{1a}$ be the pushforward of $\Hsf_{2a}$ and
  let $\Hsf_{1b}$ be the pushforward of $\Hsf_{2b}$. Then, the induced square of adapted homology theories
  \[ \begin{tikzcd}
    (\ccat_1, \Hsf_{1a}) \ar[r, "f^*"] \ar[d, "q"] &
    (\ccat_2, \Hsf_{2a}) \ar[d, "q"] &
    \dcat^\omega_{\Hsf_{1a}}(\ccat_1) \ar[r, "f^*"] \ar[d, "q"] &
    \dcat^\omega_{\Hsf_{2a}}(\ccat_2) \ar[d, "q"] \\
    (\ccat_1, \Hsf_{1b}) \ar[r, "f^*"] &
    (\ccat_2, \Hsf_{2b}) &
    \dcat^\omega_{\Hsf_{1b}}(\ccat_1) \ar[r, "f^*"] &
    \dcat^\omega_{\Hsf_{2b}}(\ccat_2) 
  \end{tikzcd} \]
and an associated diagram of deformations are both right adjointable. 
\end{lemma}

\begin{proof}
The square in $\adapted$ and the adjoints to $f^*$ are obtained via \Cref{proposition:flat_induction_yields_adjunction_of_adapted_homology_theories}. For the left square, right adjointability follows from the fact that the underlying category $2$-functor $\adapted \to \pcat\mathrm{St}$ given by $(\ccat, \Hsf) \mapsto \ccat$ is fully faithful on $2$-cells  and the fact that the underlying square of stable categories is right adjointable, since the vertical maps are given by the identity. The associated right adjointable square of deformations is obtained by applying the $2$-functor $\dcat^\omega_{-}(-)$.
\end{proof}

\begin{corollary}
Let $\Hsf \colon \dcat \rightarrow \acat$ be an adapted homology theory, 
let $f^{*} \colon \ccat \rightleftharpoons \dcat \colon f_{*}$ be an $\Hsf$-flat adjunction of stable categories
and let $\Hsf': \ccat \rightarrow \acat'$ be the pushforward of $\Hsf$. 
Then the resulting commutative diagram 
\[
\begin{tikzcd}
\ccat & {\Def_{\Hsf}(\ccat)} & \ccat \\
	\dcat & {\Def_{\Hsf}(\dcat)} & \dcat
	\arrow["{(-)^{\tau=1}}", from=1-2, to=1-3]
	\arrow["{(-)^{\tau=1}}", from=2-2, to=2-3]
	\arrow["{\nu_{H'}}", from=1-1, to=1-2]
	\arrow["{\nu_{H}}", from=2-1, to=2-2]
	\arrow["{f^{*}}"', from=1-1, to=2-1]
	\arrow["{\widetilde{f}^{*}}"', from=1-2, to=2-2]
	\arrow["{f^{*}}"', from=1-3, to=2-3]
\end{tikzcd}
\]
is vertically right adjointable.
\end{corollary}

\begin{proof}
By \cref{proposition:flat_induction_yields_adjunction_of_adapted_homology_theories}, $f^{*} \dashv f_{*}$ induces an adjunction of adapted homology theories. Since the construction of deformation categories is functorial at the level of $(\infty,2)$-categories by \cref{construction:deformation_as_a_2functor} (together with $\nu$ and $\tau$-localization, by \cref{construction:functors_between_deformations_from_universal_property} and 
\cref{construction:2_functorial_nu_using_universal_properties}), it preserves adjunctions. 
%Thus, vertically passing to right adjoints gives the commutative diagram 
%\[
%\begin{tikzcd}
%	\ccat & {\Def_{H}(C)} & \ccat \\
%	\dcat & {\Def_{H}(D)} & \dcat
%	\arrow["{(-)^{\tau=1}}", from=1-2, to=1-3]
%	\arrow["{(-)^{\tau=1}}", from=2-2, to=2-3]
%	\arrow["{\nu_{H'}}", from=1-1, to=1-2]
%	\arrow["{\nu_{H}}", from=2-1, to=2-2]
%	\arrow["{f_{*}}", from=2-1, to=1-1]
%	\arrow["{\widetilde{f}_{*}}", from=2-2, to=1-2]
%	\arrow["{f_{*}}", from=2-3, to=1-3]
%\end{tikzcd}.
%\]
\end{proof}

We observe that the pushforward construction of \cref{definition:homology_theory_induced_by_an_adjunction} preserves the property of being Grothendieck. 

\begin{lemma}
Let $f^* \colon \ccat \rightleftharpoons \dcat \colon f_*$ be an adjunction between stable presentable categories and let $\Hsf \colon \dcat \to \acat$ be a Grothendieck adapted homology theory on $\dcat$. Then, the pushforward of $\Hsf$ along $f$ is also a Grothendieck adapted homology theory.
\end{lemma}

\begin{proof}
    In \Cref{remark:left-comonad} we identified $\Hsf' \colon \ccat \to \acat'$ can be identified with the adapted factorization of $\Hsf(f^*-)$.
    From \cite[Remark 6.45]{patchkoria2021adams} we know that this adapted factorization is itself an adapted Grothendieck homology theory.
\end{proof}

\subsection{Local and colocal objects} 
\label{subsection:local_and_colocal_objects}

Suppose that we have a stable category $\ccat$ and two adapted homology theories $\Hsf_1$ and $\Hsf_2$ on with targets 
the locally graded abelian categories $\acat_1$ and $\acat_2$ and a map $q$ of adapted homology theories.
This data fits into a commutative diagram 
\[ \begin{tikzcd}
	& \ccat \ar[dl, "{\Hsf_{1}}"'] \ar[dr, "{\Hsf_{2}}"] & \\ 
    {\acat_{1}} \ar[rr, "q"] & & {\acat_{2}}.
\end{tikzcd}\]
Given objects $c, d \in \ccat$, we have a comparison morphism 
\[
\Ext_{\acat_1}^{*, *}(\Hsf_{1}(c), \Hsf_{2}(d)) \rightarrow \Ext_{\acat_2}^{*, *}(\Hsf_{2}(c), \Hsf_{2}(d)) 
\]
and a corresonding comparison between Adams spectral sequences. It is natural to ask under what conditions these two spectral sequences agree. In this section we set up a framework for studying this question and prove some technical results which allow us to streamline our presentation in the body of the paper.

\begin{definition}
\label{definition:local_and_colocal_objects_for_a_functor}
    Given a functor $F \colon \ccat \to \dcat$ we say that an object $X \in \ccat$ is \emph{$F$-local} if the map
    \[ \Map_{\ccat}(Z, X) \to \Map_{\dcat}(F(Z), F(X)) \]
    induced by $F$ is an equivalence for every $Z \in \ccat$.
    Dually, 
    we say that $X$ is \emph{$F$-colocal} if  
    \[ \Map_{\ccat}(X, Y) \to \Map_{\dcat}(F(X), F(Y)) \]
    is an equivalence for every $Y \in \ccat$.
\end{definition}

\begin{remark}
\label{remark:testing_locality_and_colocality_on_generators}
If $F$ is a cocontinuous, to test of $X$ is $F$-local it suffices to allow $Z$ to range over a collection of generators of $\ccat$ under colimits. Dually, if $F$ is continuous to check if $X$ is $F$-colocal it is enough to allow $Y$ to range over a collection of generators under limits. 
\end{remark}

Our interest in local and colocal objects comes from the following example:

\begin{example}
    Given a map of adapted homology theories
    \[ \begin{tikzcd}
        & \ccat \ar[dl, "\Hsf_1"'] \ar[dr, "\Hsf_2"] & \\
        \acat_1 \ar[rr, "q"] & & \acat_2 
    \end{tikzcd} \] 
    we obtain a comparison functor between deformations
    \[ \widetilde{q}: \dcat_{\Hsf_1}^{\omega}(\ccat) \to \dcat_{\Hsf_2}^{\omega}(\ccat) \]
    which categorifies the natural comparison map from the $\Hsf_1$-based Adams spectral sequence to the $\Hsf_2$-based Adams spectral sequence.
    In particular, if $\nu_{\Hsf_1}(X)$ is $\widetilde{q}$-local, then 
    for every $Z \in \ccat$ the natural comparison map of Adams spectral sequences
    \[ {}^{\Hsf_1}E_r^{s,t}(Z,X) \to {}^{\Hsf_2}E_r^{s,t}(Z,X) \]
    is an isomorphism.
    Dually, if $\nu_{\Hsf_1}(X)$ is $\widetilde{q}$-colocal, then 
    for every $Y \in \ccat$ the natural comparison map of spectral sequences
    \[ {}^{\Hsf_1}E_r^{s,t}(X,Y) \to {}^{\Hsf_2}E_r^{s,t}(X,Y) \]
    is an equivalence.
\end{example}

\begin{lemma}
If $F$ preserves colimits, then
    colocal objects are closed under colimits. Dually, if $F$ preserves limits, then local objects are closed under limits.
\end{lemma}

%\begin{proof}
%This is clear. 
%\end{proof}

\begin{lemma} \label{lem:square-colocal}
    Suppose we are given a horizontally right adjointable square
    \[ \begin{tikzcd}
      \acat \ar[r, "f^*"] \ar[d, "h"] & \bcat \ar[d, "h"] \\
      \ccat \ar[r, "f^*"] & \dcat. 
    \end{tikzcd} \]
    \begin{enumerate}
        \item If $X \in \acat$ is $h$-colocal, then $f^*X$ is $h$-colocal.
        \item If the essential image of $f_*$ cogenerates $\acat$ and $f^*X$ is $h$-colocal, then $X$ is $h$-colocal.
        \item If $Y \in \bcat$ is $h$-local, then $f_*Y$ is $h$-local.
        \item If the essential image of $f^*$ generates $\bcat$ and $f_*Y$ is $h$-local, then $Y$ is $h$-local.
    \end{enumerate}
\end{lemma}

\begin{proof}
  Given $X \in \acat$ and $Y \in \bcat$ we have a commuting diagram of mapping spaces
  \[
  \begin{tikzcd}
    \Map_{\acat}(X, f_*Y) \ar[r] \ar[ddd, no head, "\cong"] & 
    \Map_{\ccat}(h(X), h(f_*Y)) \ar[d, no head, "\cong"] \\
    & \Map_{\ccat}(h(X), f_* h(Y)) \ar[d, no head, "\cong"] \\
    & \Map_{\dcat}(f^* h(X), h(Y)) \ar[d, no head, "\cong"] \\
    \Map_{\bcat}(f^*X, Y) \ar[r] & 
    \Map_{\dcat}(h(f^*X), h(Y)).
  \end{tikzcd}
  \]
  Using the fact that the top horizontal arrow is an isomorphism 
  if and only if the bottom horizontal arrow is an isomorphism 
  we are able to prove (1) and (3).
  Proving (2) and (4) requires the additional observation from \cref{remark:testing_locality_and_colocality_on_generators} that it suffices to restrict to a generating collection of test objects.
\end{proof}

\begin{corollary}
    Suppose we are given
    a symmetric monoidal functor $F \colon \ccat \to \dcat$ and a dualizable object $X \in \ccat$ and an object $Y \in \ccat$.
    Then 
    \begin{enumerate}
        \item if $Y$ is $F$-colocal, then $X \otimes Y$ is $F$-colocal and
        \item if $Y$ if $F$-local, then $X \otimes Y$ is $F$-local.
    \end{enumerate}
\end{corollary}

\begin{proof}
  Using the assumption that $F$ is symmetric monoidal and $X$ is dualizable we can construct right adjointable squares of the form 
  \[ \begin{tikzcd}
      \ccat \ar[rr, "- \otimes X"] \ar[d, "F"] & & 
      \ccat \ar[d, "F"] &
      \ccat \ar[rr, "- \otimes X^{\vee}"] \ar[d, "F"] & & 
      \ccat \ar[d, "F"] \\
      \dcat \ar[rr, "- \otimes F(X)"] & & 
      \dcat &
      \dcat \ar[rr, "- \otimes F(X)^{\vee}"]  & & 
      \dcat. 
  \end{tikzcd} \]
  The corollary now follows from \Cref{lem:square-colocal}, parts $(1)$ and $(3)$.
\end{proof}

\begin{lemma} \label{lem:dualizing-co-local}
    Suppose we are given
    \begin{enumerate}
        \item a stable, presentably symmetric monoidal category $\ccat$ that is compactly generated by dualizable objects,
        \item a stable, presentably symmetric monoidal category $\dcat$ and 
        \item a symmetric monoidal functor $F: \ccat \to \dcat$ that preserves colimits and compactness.
    \end{enumerate}
    Then, a dualizable object $X$ is $F$-colocal if and only if $X^{\vee}$ is $F$-local.
\end{lemma}

\begin{proof}
    Assume $X^\vee$ is $F$-local. 
    Then we would like to show that for $Y \in \ccat$ we have 
    \[ \Map_{\ccat}(X,Y) \cong \Map_{\dcat}(F(X),F(Y)). \]
    Using our assumptions on $\ccat$ we can rewrite the monoidal unit $\monunit$ and $Y$ as colimits $\colim_\alpha Z_\alpha$ and $\colim_\beta Y_\beta$ of compact dualizable objects.
    Then, using our hypotheses on $F$ and the assumption that $X^\vee$ is $F$-local we have equivalences
    \begin{align*}
        \Map_{\ccat}&(X, Y)
        \cong \Map_{\ccat}(X \otimes \colim_\alpha Z_\alpha, \colim_\beta Y_\beta)
        \cong \Map_{\ccat}(\colim_\alpha Z_\alpha, X^\vee \otimes \colim_\beta Y_\beta) \\
        &\cong \lim_\alpha \Map_{\ccat}( Z_\alpha, \colim_\beta X^\vee \otimes  Y_\beta)
        \cong \lim_\alpha \colim_\beta \Map_{\ccat}( Z_\alpha, X^\vee \otimes  Y_\beta) \\
        &\cong\lim_\alpha \colim_\beta \Map_{\ccat}(Y_\beta^{\vee} \otimes Z_\alpha, X^\vee)
        \cong\lim_\alpha \colim_\beta \Map_{\dcat}(F(Y_\beta^{\vee} \otimes Z_\alpha), F(X^\vee)) \\
        &\cong\lim_\alpha \colim_\beta \Map_{\dcat}(F(Y_\beta)^{\vee} \otimes F(Z_\alpha), F(X)^\vee)
        \cong\lim_\alpha \colim_\beta \Map_{\dcat}(F(Z_\alpha), F(X)^\vee \otimes F(Y_\beta)) \\
        &\cong \Map_{\dcat}(\colim_\alpha  F(Z_\alpha), F(X)^\vee \otimes \colim_\beta  F(Y_\beta))
        \cong \Map_{\dcat}(F(X) \otimes F(\monunit), F(Y)) \\
        &\cong \Map_{\dcat}(F(X), F(Y)).
    \end{align*}
    
    For the reverse direction, assume that $X$ is $F$-colocal.
    We must show that  
    \[ \Map_{\ccat}(Z, X^\vee) \cong \Map_{\dcat}(F(Z),F(X^\vee)) \]
    for all $Z \in \ccat$.
    As above we expand the auxiliary object $Z$ as a colimit $\colim_\alpha Z_\alpha$ of compact dualizable objects.
    Then we have equivalences
    \begin{align*}
        \Map_{\ccat}&(Z, X^\vee) 
        \cong \Map_{\ccat}(\colim_\alpha Z_\alpha, X^\vee) 
        \cong \lim_\alpha \Map_{\ccat}(Z_\alpha, X^\vee) 
        \cong \lim_\alpha \Map_{\ccat}(X, Z_\alpha^\vee) \\
        &\cong \lim_\alpha \Map_{\dcat}(F(X), F(Z_\alpha^\vee)) 
        \cong \lim_\alpha \Map_{\dcat}(F(Z_\alpha), F(X^\vee)) \\
        &\cong  \Map_{\dcat}(\colim_\alpha F(Z_\alpha), F(X^\vee)) 
        \cong  \Map_{\dcat}(F(Z), F(X^\vee)). 
    \end{align*}
\end{proof}

\begin{lemma} \label{lem:big-to-small}
    Suppose we are given a map of adapted Grothendieck homology theories
    \[ \begin{tikzcd}
      & \ccat \ar[dl, "\Hsf_1"'] \ar[dr, "\Hsf_2"] &  &
      \dcat^\omega_{\Hsf_1}(\ccat) \ar[r, "q"] \ar[d] & \dcat^\omega_{\Hsf_2}(\ccat) \ar[d] \\
      \acat_1 \ar[rr, "q"] & & \acat_2 &
      \widecheck{\dcat}_{\Hsf_1}(\ccat) \ar[r, "\widecheck{q}"] & \widecheck{\dcat}_{\Hsf_2}(\ccat) 
    \end{tikzcd} \]
    and the associated square of comparison functors between the
    perfect and unseperated Grothendieck deformations along $\Hsf_1$ and $\Hsf_2$ respectively.
    Given an object $X \in \ccat$:
    \begin{enumerate}
    \item If $\nu_{\Hsf_1}(X)$ is $q$-colocal,
      both $\widecheck{\nu}_{\Hsf_1}(X)$ and $\widecheck{\nu}_{\Hsf_2}(X)$ are compact and
      $\widecheck{q}$ preserves compactness,
      then $\widecheck{\nu}_{\Hsf_1}(X)$ is $q$-colocal.
    \item If $\nu_{\Hsf_1}(X)$ is $q$-local, then $\widecheck{\nu}_{\Hsf_1}(X)$ is $q$-local.
    \item If $\widecheck{\nu}_{\Hsf_1}(X)$ is $q$-colocal, then $\nu_{\Hsf_1}(X)$ is $q$-colocal.
    \item If $\widecheck{\nu}_{\Hsf_1}(X)$ is $q$-local, then $\nu_{\Hsf_1}(X)$ is $q$-local.
    \end{enumerate}    
\end{lemma}

\begin{proof}
  We start by recalling from \cite[Thm 6.51]{patchkoria2021adams} that the vertical maps from the perfect deformation to the useperated deformation are fully faithful.
  Parts (3) and (4) of the lemma follow from this fully faithfulness.
  For part (1) and (2) it suffices to observe that
  $\widecheck{\dcat}_{\Hsf_1}(\ccat)$ is generated under colimits by the image of
  $\dcat^\omega_{\Hsf_1}(\ccat)$ and that the comparison map $\widecheck{q}$
  is a left adjoint between presentable categories.  
\end{proof}

\begin{proposition}
\label{prop:colocal-to-abelian}
    Suppose we are given a map of adapted homology theories
    \[ \begin{tikzcd}
        & \ccat \ar[dl, "\Hsf_1"'] \ar[dr, "\Hsf_2"] & \\
        \acat_1 \ar[rr, "q"] & & \acat_2 
    \end{tikzcd} \]
    and an object $X \in \ccat$.
    Then $\nu_{\Hsf_1}(X)$ is $\tilde{q}$-(co)local 
    if and only if $\Hsf_1(X) \in \dcat^b(\acat_1)$ is $q$-(co)local.
\end{proposition}

\begin{proof}
  Given two objects $X,Y$ in $\dcat_{\Hsf_1}^{\omega}(\ccat)$ we have a diagram of mapping spectra
  \[ \begin{tikzcd}
    \map_{\dcat_{\Hsf_1}^{\omega}(\ccat)}(\Sigma^{k,s+1} X, Y) \ar[r, "\tau"] \ar[d] & 
    \map_{\dcat_{\Hsf_1}^{\omega}(\ccat)}(\Sigma^{k,s} X, Y) \ar[r] \ar[d] & \map_{\dcat_{\Hsf_1}^{\omega}(\ccat)}(\Sigma^{k,s} X, Y/\tau) \ar[d] \\
    \map_{\dcat_{\Hsf_2}^{\omega}(\ccat)}(\Sigma^{k,s+1} X, Y) \ar[r, "\tau"] & \map_{\dcat_{\Hsf_2}^{\omega}(\ccat)}(\Sigma^{k,s} X, Y) \ar[r] & \map_{\dcat_{\Hsf_2}^{\omega}(\ccat)}(\Sigma^{k,s} X, Y/\tau)
  \end{tikzcd} \]
  where each row is a cofiber sequence.
  Using the perfectness of $X$ and $Y$ we know that $Y/\tau$ is bounded and therefore that for $s \ll 0$ we have equivalences
  \[ [\Sigma^{k,s}X, Y]_{\dcat_{\Hsf_1}^{\omega}(\ccat)} \xrightarrow{\cong} [\Sigma^{k,s}X, Y]_{\dcat_{\Hsf_2}^{\omega}(\ccat)} \xrightarrow{\cong} [\Sigma^{k,s}X^{\tau=1}, Y^{\tau=1}]_{\ccat}. \]
  This lets us ground an upward induction on $s$ to show that
  \[ [\Sigma^{k,s} X, Y]_{\dcat_{\Hsf_1}^{\omega}(\ccat)} \to [\Sigma^{k,s} X, Y]_{\dcat_{\Hsf_2}^{\omega}(\ccat)} \]
  is an equivalence for all $k,s$ if and only if 
  \[ [\Sigma^{k,s} X, Y/\tau]_{\dcat_{\Hsf_1}^{\omega}(\ccat)} \to [\Sigma^{k,s} X, Y/\tau]_{\dcat_{\Hsf_2}^{\omega}(\ccat)} \]
  is an equivalence for all $k,s$.
  Then, using the homology adjunctions 
  $\dcat_{\Hsf_i}^{\omega}(\ccat) \rightleftharpoons \dcat^b(\acat_i)$ of \cite[\S 5.3]{patchkoria2021adams}) we obtain isomorphisms
  \[ [\Sigma^{k,s} X, Y/\tau]_{\dcat_{\Hsf_i}^{\omega}(\ccat)} \cong [\Sigma^{k,s} X, H_*H^* (Y)]_{\dcat_{\Hsf_i}^{\omega}(\ccat)} \cong [\Sigma^{k,s} H^*(X), H^*(Y)) ]_{\dcat^b(\acat_i)} \]
  which ends the argument. 
\end{proof}

\begin{example} \label{exm:H projective}
    Given an adapted homology theory $\Hsf \colon \ccat \to \acat$ there is a unique map of homology theories $\eta \colon A(\ccat) \to \Hsf$ from the universal adapted homology theory to $\Hsf$.
    Since $y_{0}(X) \in A(\ccat)$ is injective for all $X \in \ccat$ 
    we can use \Cref{prop:colocal-to-abelian} to conclude that the following are equivalent: 
    \begin{enumerate}
        \item $\nu_y(X)$ is $\widetilde{\eta}$-colocal
        \item $\Hsf(X)$ is a projective object of $\acat$ and the map 
    $[X, Y]_{\ccat} \rightarrow \Hom_{\acat}(\Hsf(X), \Hsf(Y))$ is an isomorphism for any $Y$.
    \end{enumerate}
Note that the second condition is equivalent to saying that $X$ is an $\Hsf$-injective for the opposite homology theory $\ccat^{op} \rightarrow \acat^{op}$. 
\end{example}

\begin{lemma} \label{lem:P colocal}
  Let $f^* \colon \ccat_1 \leftrightharpoons \ccat_2 \colon f_*$ be an adjunction between stable categories and let $q \colon \Hsf_{2a} \to \Hsf_{2b}$ a map of adapted homology theories on $\ccat_2$ such that $f$ is both $\Hsf_{2a}$-flat and $\Hsf_{2b}$-flat.
  Let $\Hsf_{1a}$ be the pushforward of $\Hsf_{2a}$ and
  let $\Hsf_{1b}$ be the pushforward of $\Hsf_{2b}$, so that by \cref{lemma:morphism_of_adjunctions_induced_by_map_of_flat_homology_theories} we we have a right adjointable square 
  \[ \begin{tikzcd}
    \dcat^\omega_{\Hsf_{1a}}(\ccat_1) \ar[r, "f^*"] \ar[d, "q"] &
    \dcat^\omega_{\Hsf_{2a}}(\ccat_2) \ar[d, "q"] \\
    \dcat^\omega_{\Hsf_{1b}}(\ccat_1) \ar[r, "f^*"] &
    \dcat^\omega_{\Hsf_{2b}}(\ccat_2) 
  \end{tikzcd} \]
of deformations. Suppose we are given an $X \in \ccat_1$ such that one of the following equivalent conditions
  \begin{enumerate}
  \item $\nu(f^*X)$ is $q$-colocal or
  \item $\Hsf_{2a}(f^*X)$ is colocal for the map
    $\dcat^b(\acat(\ccat_2, \Hsf_{2a})) \to \dcat^b(\acat(\ccat_2, \Hsf_{2b}))$
  \end{enumerate}
  is satisfied. Then, $\nu(X)$ is $q$-colocal.
\end{lemma}

\begin{proof}
  Using \Cref{prop:colocal-to-abelian} we reduce the lemma to showing that $\Hsf_{1a}(X) \in \dcat^b(\acat(\ccat_1, \Hsf_{1a}))$ is $q$-colocal.
  Applying $\dcat^b(\acat(-))$ to the right adjointable square in $\adapted$ we obtain a right adjointable square
  \[ \begin{tikzcd}
    \dcat^b(\acat(\ccat_1,\Hsf_{1a})) \ar[r, "f^*"] \ar[d, "q"] &
    \dcat^b(\acat(\ccat_2,\Hsf_{2a})) \ar[d, "q"] \\
    \dcat^b(\acat(\ccat_1,\Hsf_{1b})) \ar[r, "f^*"] &
    \dcat^b(\acat(\ccat_2,\Hsf_{2b})) 
  \end{tikzcd} \]
  where the horizontal adjunctions are comonadic by \Cref{proposition:flat_induction_yields_adjunction_of_adapted_homology_theories}.
  In particular, comonadicity implies that the image of $f_*$ cogenerates $\dcat^b(\acat(\ccat_1,\Hsf_{1a}))$.
  Now we may use \Cref{lem:square-colocal} to conclude that
  it will suffice to show that $\Hsf_{2a}(f^*X)$ is $q$-colocal.
  Applying \Cref{prop:colocal-to-abelian} again if necessary we find that we are left with the our given hypothesis on $X$.
\end{proof}

%% \section{Leftovers}
%% \input{old.tex}

\bibliographystyle{amsalpha}
\bibliography{quivers_and_the_adams_sseq_bibliography}

\newcommand{\etalchar}[1]{$^{#1}$}
\providecommand{\bysame}{\leavevmode\hbox to3em{\hrulefill}\thinspace}
\providecommand{\MR}{\relax\ifhmode\unskip\space\fi MR }
\providecommand{\MRhref}[2]{%
  \href{http://www.ams.org/mathscinet-getitem?mr=#1}{#2}
}
\providecommand{\href}[2]{#2}
\begin{thebibliography}{BBB{\etalchar{+}}21}

\bibitem[Ada95]{adams1995stable}
John~Frank Adams, \emph{Stable homotopy and generalised homology}, University
  of Chicago press, 1995.

\bibitem[BBB{\etalchar{+}}19]{beaudry2019tmf}
Agnes Beaudry, Mark Behrens, Prasit Bhattacharya, Dominic Culver, and Zhouli
  Xu, \emph{On the {$tmf$}-resolution of {$Z$}}, arXiv preprint
  arXiv:1909.13379 (2019).

\bibitem[BBB{\etalchar{+}}20]{beaudry20202}
\bysame, \emph{On the {$E_2$}-term of the {$bo$}-{A}dams spectral sequence},
  Journal of Topology \textbf{13} (2020), no.~1, 356--415.

\bibitem[BBB{\etalchar{+}}21]{beaudry2021telescope}
\bysame, \emph{The telescope conjecture at height 2 and the tmf resolution},
  Journal of Topology \textbf{14} (2021), no.~4, 1243--1320.

\bibitem[BHS19]{burklund2019boundaries}
Robert Burklund, Jeremy Hahn, and Andrew Senger, \emph{On the boundaries of
  highly connected, almost closed manifolds}, arXiv preprint arXiv:1910.14116
  (2019).

\bibitem[Bou85]{bousfield1985homotopy}
Aldridge~K Bousfield, \emph{On the homotopy theory of {$K$-local} spectra at an
  odd prime}, Amer. J. Math. \textbf{107} (1985), no.~4, 895--932.

\bibitem[Bou90]{bousfield1990classification}
\bysame, \emph{A classification of {$K$-local} spectra}, Journal of Pure and
  Applied Algebra \textbf{66} (1990), no.~2, 121--163.

\bibitem[BP65]{brown1965spectrum}
Edgar~H Brown and Franklin~P Peterson, \emph{A spectrum whose 2, cohomology is
  the algebra of reduced pth powers}.

\bibitem[Car56]{cartan1955seminaire}
H~Cartan, \emph{{Alg\`ebres d'Eilenberg-MacLane et Homotopie}}, Expos\'es 2 \`a
  16, S\'eminaire Henri Cartan, Ecole Normale Sup\'erieure, Paris (1956).

\bibitem[CM19]{chan2019representation}
Aaron Chan and Ren{\'e} Marczinzik, \emph{On representation-finite
  gendo-symmetric biserial algebras}, Algebras and Representation Theory
  \textbf{22} (2019), 141--176.

\bibitem[GH04]{moduli_spaces_of_commutative_ring_spectra}
P.~G. Goerss and M.~J. Hopkins, \emph{Moduli spaces of commutative ring
  spectra}, Structured ring spectra, London Math. Soc. Lecture Note Ser., vol.
  315, Cambridge Univ. Press, Cambridge, 2004, pp.~151--200. \MR{2125040
  (2006b:55010)}

\bibitem[Gon98]{gonzalez_regular_complex}
Jes\'{u}s Gonz\'{a}lez, \emph{The regular complex in the {$BP\langle 1
  \rangle$}-{A}dams spectral sequence}, Trans. Amer. Math. Soc. \textbf{350}
  (1998), no.~7, 2629--2664.

\bibitem[Gon00]{gonzalez2000vanishing}
Jes{\'u}s Gonz{\'a}lez, \emph{A vanishing line in the {$\mathrm{BP}\langle 1
  \rangle$}-{Adams} spectral sequence}, Topology \textbf{39} (2000), no.~6,
  1137--1153.

\bibitem[Hov]{hovey2003homotopy}
Mark Hovey, \emph{Homotopy theory of comodules over a {Hopf} algebroid,
  homotopy theory: relations with algebraic geometry, group cohomology, and
  algebraic {K}-theory, 261--304}, Contemp. Math \textbf{346}.

\bibitem[HP22]{diracgeometry1}
Lars Hesselholt and Piotr Pstragowski, \emph{Dirac geometry {I}: Commutative
  algebra}, arXiv preprint arXiv:2207.09256 (2022).

\bibitem[HP23]{hesselholt2023dirac}
\bysame, \emph{Dirac geometry ii: Coherent cohomology}, arXiv preprint
  arXiv:2303.13444 (2023).

\bibitem[HW22]{hahn2022redshift}
Jeremy Hahn and Dylan Wilson, \emph{Redshift and multiplication for truncated
  {Brown--Peterson} spectra}, Annals of Mathematics \textbf{196} (2022), no.~3,
  1277--1351.

\bibitem[Kan81]{kane1981operations}
Richard~M Kane, \emph{Operations in connective {$K$-Theory}}, vol. 254,
  American Mathematical Soc., 1981.

\bibitem[Koc82]{kochman1982integral}
Stanley~O Kochman, \emph{Integral cohomology operations}, Current trends in
  algebraic topology (1982), 437--478.

\bibitem[Kra10]{krause_2010}
Henning Krause, \emph{Localization theory for triangulated categories}, London
  Mathematical Society Lecture Note Series, p.~161–235, Cambridge University
  Press, 2010.

\bibitem[LH]{lurie_hopkins_brauer_group}
Jacob Lurie and Michael Hopkins, \emph{On {Brauer} groups of {Lubin-Tate}
  spectra {I}}, http://www.math.harvard.edu/~lurie/papers/Brauer.pdf.

\bibitem[Lura]{higher_algebra}
Jacob Lurie, \emph{Higher {A}lgebra},
  \href{https://www.math.ias.edu/~lurie/papers/HA.pdf}{{A}vailable online}.

\bibitem[Lurb]{lurie_spectral_algebraic_geometry}
\bysame, \emph{Spectral algebraic geometry},
  \url{http://www.math.harvard.edu/~lurie/papers/SAG-rootfile.pdf}.

\bibitem[Mah81]{mahowald1981bo}
Mark Mahowald, \emph{$bo$-resolutions}, Pacific Journal of Mathematics
  \textbf{92} (1981), no.~2, 365--383.

\bibitem[Mah82]{mahowald1982image}
\bysame, \emph{The image of {J} in the {EHP} sequence}, Annals of Mathematics
  (1982), 65--112.

\bibitem[Mos68]{moss1968composition}
Robert~MF Moss, \emph{On the composition pairing of {Adams} spectral
  sequences}, Proceedings of the London Mathematical Society \textbf{3} (1968),
  no.~1, 179--192.

\bibitem[PP21]{patchkoria2021adams}
Irakli Patchkoria and Piotr Pstr{\k{a}}gowski, \emph{Adams spectral sequences
  and {Franke}'s algebraicity conjecture}, arXiv preprint arXiv:2110.03669
  (2021).

\bibitem[Pst21]{pstrkagowski2021chromatic}
Piotr Pstr{\k{a}}gowski, \emph{Chromatic homotopy theory is algebraic when {$p
  > n^2+ n+ 1$}}, Advances in Mathematics \textbf{391} (2021), 107958.

\bibitem[Pst22]{pstrkagowski2018synthetic}
\bysame, \emph{Synthetic spectra and the cellular motivic category},
  Inventiones mathematicae (2022), 1--129.

\bibitem[PV21]{pstrkagowski2021abstract}
Piotr Pstr{\k{a}}gowski and Paul VanKoughnett, \emph{Abstract {Goerss-Hopkins}
  theory}, Advances in Mathematics (2021), 108098.

\bibitem[Rav86]{ravenel_complex_cobordism}
Douglas~C. Ravenel, \emph{Complex cobordism and stable homotopy groups of
  spheres}, Pure and Applied Mathematics, vol. 121, Academic Press, Inc.,
  Orlando, FL, 1986. \MR{860042 (87j:55003)}

\bibitem[Str03]{stroppel2003category}
Catharina Stroppel, \emph{Category $\mathcal{O}$: quivers and endomorphism
  rings of projectives}, Representation Theory of the American Mathematical
  Society \textbf{7} (2003), no.~15, 322--345.

\bibitem[Tat22]{tatum2022spectrum}
Elizabeth~Ellen Tatum, \emph{On a spectrum-level splitting of the
  {$\mathrm{BP}\langle n \rangle$}-cooperations algebra}, arXiv preprint
  arXiv:2208.11772 (2022).

\bibitem[Tod71]{toda1971spectra}
Hirosi Toda, \emph{On spectra realizing exterior parts of the {Steenrod}
  algebra}, Topology \textbf{10} (1971), no.~1, 53--65.

\bibitem[Yag80]{yagita1980steenrod}
Nobuaki Yagita, \emph{On the {S}teenrod algebra of {M}orava {$K$}-theory},
  Journal of the London Mathematical Society \textbf{2} (1980), no.~3,
  423--438.

\end{thebibliography}

\end{document}